\documentclass[10pt,a4paper]{article}
\usepackage[T1]{fontenc}
\usepackage[a4paper, margin=1in]{geometry}
\usepackage{graphicx}
\usepackage{mathtools}
\usepackage{amssymb}
\usepackage{amsthm}
\usepackage{thmtools}
\usepackage{xcolor}
\usepackage{nameref}
\usepackage[english]{babel}
\usepackage{enumerate}
\usepackage{hyperref}
\usepackage{lipsum}
\usepackage{amsmath}
\usepackage{xfrac}
\usepackage{faktor}
\usepackage{mathrsfs}
\usepackage{tkz-graph}
\usepackage{latexsym}
\usetikzlibrary{arrows}
\usepackage{tikz}
\usepackage{adjustbox}
\usepackage{comment}
\usepackage{multirow}
\usepackage{subcaption} 
\usepackage{dsfont}
\usepackage{pagecolor}
\usepackage{cleveref} 
\usepackage{titlesec}
\newcounter{lil1}
\newenvironment{steps}
{\begin{list} { \bf Step (\Alph{lil1})}
		{ \usecounter{lil1}
			\setlength{\leftmargin}{0.0cm}
			\setlength{\topsep}{0.2cm}
			\setlength{\itemsep}{0.0cm}
			\setlength{\parsep}{0.1cm}
			\setlength{\itemindent}{0.8cm}
			\setlength{\parskip}{0.0cm}}}
	{\end{list}}

\newcounter{lil33}

\titleformat{\section}
[hang] 
{\normalfont\Large\bfseries} 
{\thesection.} 
{0.4em} 
{} 

\titleformat{\subsection}
[hang] 
{\normalfont\large\bfseries} 
{\thesubsection.} 
{0.3em} 
{} 

\definecolor{myred}{RGB}{255,0,0}

\crefname{equation}{(equation)}{(equations)}
\Crefname{equation}{(Equation)}{(Equations)}

\crefformat{(equation)}{\textcolor{myred}{(#1)}}
\hypersetup{
	colorlinks=true,
	linkcolor=myred,
	citecolor=myred,
	urlcolor=myred
}

\allowdisplaybreaks

\theoremstyle{plain}
\newtheorem{thm}{Theorem}[section]

\newtheorem{lem}{Lemma}[section]

\newtheorem{prop}[thm]{Proposition}

\theoremstyle{definition}
\newtheorem{defi}{Definition}[section]

\newtheorem{rem}{Remark}

\usepackage[framemethod=TikZ]{mdframed}
\DeclareMathOperator*{\esssup}{ess\,sup}
\renewcommand{\leq}{\leqslant}
\renewcommand{\geq}{\geqslant}

\newcommand{\grad}{\ensuremath{\nabla}}

\Crefname{section}{Section}{Sections}
\Crefname{subsection}{Subsection}{Subsections}
\numberwithin{equation}{section}

\newmdenv[linewidth=2pt,
linecolor=black,
innertopmargin=0.2cm,
innerbottommargin=0.5cm,
innerleftmargin=0.7cm,
innerrightmargin=0.7cm,
tikzsetting={fill=blue!30!white}] {squarebox}
\title{\textbf{Large deviation principle for a stochastic nonlinear damped Schr\"odinger equation}}
\author{
	Sandip Roy\footnote{Department of Mathematics, Indian Institute of Technology Indore, Indore 453552, India\\ \textsuperscript{\, \quad\tiny{+}}Department of Mathematics, Indian Institute of Technology Roorkee, Roorkee 247667, India}\textsuperscript{\,\ ,}\footnote{Corresponding author.\newline
		E-mail addresses: phd2301141004@iiti.ac.in (S. Roy), 
		debopriya@iiti.ac.in (D. Mukherjee), 
		maniltmohan@ma.iitr.ac.in (M. T. Mohan)} , 
	Debopriya Mukherjee\footnotemark[1] ,
	Manil T. Mohan\textsuperscript{\tiny{+}} \
}
\begin{document}
	\maketitle
	\begin{abstract}
		The present paper focuses on the stochastic nonlinear Schr\"odinger equation with polynomial nonlinearity, and a zero-order (no derivatives involved) linear damping. Here, the random forcing term appears as a mix of a nonlinear noise in the It\^o sense and a linear multiplicative noise in the Stratonovich sense. We prove the Laplace principle for the family of solutions to the stochastic system in a suitable Polish space, using the weak convergence framework of Budhiraja and Dupuis. This analysis is non-trivial, since it requires uniform estimates for the solutions of the associated controlled stochastic equation  in the underlying solution space in order to verify the weak convergence criterion. The Wentzell-Freidlin type large deviation principle is proved using Varadhan's lemma and Bryc’s converse to Varadhan's lemma. The local well-posedness of the skeleton  equation (deterministic controlled system) is established by employing the Banach fixed point theorem, and the global well-posedness is established via Yosida approximation. We show that the conservation law holds in the absence of the linear damping and It\^o noise. The well-posedness of the stochastic controlled equation is also nontrivial in this case. We use a truncation method, a stopping time argument, and the Yosida technique to get the global well-posedness of the stochastic controlled equation. 
	\end{abstract} 
	\textbf{Keywords:} Nonlinear Schr\"odinger equation, Stratonovich noise, Large deviation principle, Laplace principle, Yosida approximation, Strichartz estimates. 
	
	\noindent \textbf{MSC 2020:} 35Q55, 35R60, 60H15, 60F10. 
	\section{Introduction} 
	This article is concerned with the study of the large deviation principle (LDP) for solutions of the stochastic nonlinear Schrödinger equation (SNLSE) given by
	\begin{equation}	\label{perturbed}
		\begin{cases*}
			du(t)=-\big[i Au(t) +i \mathcal{N}(u(t))+ \beta u(t)\big]dt - i \sqrt{\varepsilon} B(u(t)) \circ d\mathcal{W}_1(t)- i \sqrt{\varepsilon} G(u(t)) d\mathcal{W}_2(t) , \ t>0,\\
			u(0)= u_0,
		\end{cases*}	
	\end{equation}
	where $u(x,t,\cdot)$ is the unknown complex-valued function of space and time. Here, $A$ is a non-negative self-adjoint operator, $\beta$ is a damping constant (non-negative), $B$ is a linear bounded operator, $\mathcal{N}$ and $G$ are nonlinearities, and $\varepsilon>0$ characterizes the intensity of noise dispersion. Also, $\mathcal{W}_1$ and $\mathcal{W}_2$ are two independent Wiener processes defined on a filtered probability space \((\Omega, \mathcal{F}, \{\mathcal{F}_t \}_{t \geq 0}, \mathbb{P})\), where $\mathcal{W}_1$ is considered in the (multiplicative) Stratonovich form and $\mathcal{W}_2$ in the (multiplicative) It\^o sense. 
	The nonlinearity $\mathcal{N}(u) = \lambda |u|^{\alpha - 1} u, \, \alpha \in (1,\infty)$ is termed \emph{defocusing} when $\lambda > 0$ and \emph{focusing} when $\lambda < 0$. Our results apply to both situations. Since the required estimates do not depend on the sign of $\lambda$, we set $\lambda = 1$ for simplicity in the calculations. All mathematical arguments developed in this manuscript remain valid for both the focusing and defocusing cases.
	

	The Schr\"odinger equation is the fundamental equation of quantum mechanics, while its nonlinear counterpart governs the dynamics of waves influenced by both dispersive and nonlinear effects. The nonlinear Schr\"odinger equation (NLSE) serves as a versatile mathematical model with a broad range of physical applications. It is widely used to describe important phenomena in quantum mechanics, fluid dynamics, plasma physics, and molecular biology \cite{agrawal2000nonlinear,chen1984introduction,shankar2012principles,sulem2007nonlinear}. For more specific applications, such as hydrodynamics, superconductivity, Bose-Einstein condensation, nonlinear optics, and the Gross-Pitaevskii equation, we refer the reader to \cite{PhysRevE.75.066607,PhysRevA.60.1507,karjanto2019nonlinear}.

	The propagation of nonlinear dispersive waves in inhomogeneous or random media can be more accurately described by introducing suitable noise terms into the governing equation. The choice of noise depends on the underlying physical situation, representing, for instance, temporal or spatial fluctuations of system parameters. Different types of noise give rise to distinct solution behaviors. Additive Gaussian noise typically models external random forcing and is independent of the solution, whereas multiplicative noise depends explicitly on the solution itself. The interpretation of multiplicative noise varies with the chosen framework: in the Stratonovich formulation it resembles classical calculus, while in the It\^o setting additional correction terms appear in the dynamics; see \cite{goodair2025stratonovich}. One may also consider jump-type noise driven by L\'evy processes, which accounts for sudden or impulsive perturbations \cite{MR2356959}.

	\emph{The conservative case:} When the system is described using the Stratonovich differential and damping is absent, the wave function normalization is preserved. More precisely, applying the Itô formula shows that the mass is conserved along each path. In other words, the $L^2$ norm of the solution remains constant: $\|u(t)\|_{L^2}^2 = \|u_0\|_{L^2}^2$ for all $t \geq 0$. Therefore, if the initial state is normalized ($\|u_0\|_{L^2}^2 = 1$), the quantum system evolves on the unit ball of $L^2$, remains on the unit sphere, and satisfies the probability conservation (in the sense that the total probability of finding the particle is always equal to $1$). Applications of this property can be found in \cite{bang1994temperature,MR1382938}. The physical modeling of nonlinear dispersive wave transmission through nonlinear or stochastic media was first proposed in \cite{rasmussen1995influence}.

	\emph{The nonconservative case:} In this scenario, the wave function is no longer normalized. Specifically, the norm of the wave function is not preserved along individual paths; however, the mean-square norm can remain constant. Mathematically, the map $t \mapsto \|u(t)\|_{L^2}^2$ is no longer time-independent but instead forms a continuous martingale. This ensures conservation of the mean-square norm, $\mathbb{E}[\|u(t)\|^2_{L^2}]$, for $t \in [0, T]$, allowing one to define the “physical” probability measure; see \cite{zhang2014stochastic}. This property plays a key role in quantum trajectory theory and in the stochastic unraveling of open quantum systems \cite{barchielli2009quantum}.

	For the SNLSE with multiplicative noise interpreted in the Stratonovich sense, mass conservation is preserved, as can be shown using It\^o's formula. Moreover, in the case of linear multiplicative noise, both mass and energy are conserved \cite{MR3980316}, reflecting behavior similar to the corresponding deterministic equation. In the present study, however, the inclusion of a damping term and It\^o noise leads to the loss of both mass and energy conservation \cite{MR4678996}. Notably, when $\beta = 0$ and $G \equiv 0$, the conservation law is recovered.

	The linear component governs the dispersive behavior of the system, allowing the use of Strichartz and local smoothing estimates. These dispersive properties provide the framework to apply the Banach fixed-point theorem in suitable function spaces and to control the nonlinearity for $\alpha$ within an appropriate range. Incorporating dissipative terms further facilitates the analysis, helping to establish the well-posedness of the system. Therefore, in our setting, the damping term plays a crucial role.

	\subsection{Literature review}
	\subsubsection{SNLSE}
	Over the years, the SNLSE with various types of noise,additive, multiplicative, or jump noise, has been extensively studied. Considerable work has focused on the existence and uniqueness of solutions on both bounded and unbounded domains \cite{MR3705702, MR3215081, MR3474409}.
	The long-term behavior of solutions, including stability and other qualitative properties, has also been extensively investigated. The presence of randomness, through different types of noise, introduces various challenges and leads to diverse behaviors. To address these challenges, several analytical tools have been developed, such as stochastic Strichartz estimates and the construction of invariant measures \cite{MR3232027, MR3859442, MR4249116, MR4678996, MR2225449}.
	In \cite{MR3859442}, the author investigates the SNLSE with nonlinear Stratonovich noise, establishing local well-posedness for both subcritical and critical nonlinearities and proving global existence in the subcritical case under an additional condition on the noise. The existence of a unique global mild solution for the SNLSE driven by pure jump noise is established in \cite{MR4558968}. In \cite{MR1706888}, the authors consider a conservative SNLSE with multiplicative noise and prove global well-posedness in $L^2(\mathbb{R}^d)$ for subcritical nonlinearities. Later, in \cite{MR1954077}, they extend these results to establish local and global well-posedness in the energy space $H^1(\mathbb{R}^d)$ for SNLSE with either additive or multiplicative noise.
	The study in \cite{MR3232027} focuses on the existence and uniqueness of a local maximal solution to the SNLSE with multiplicative noise on a compact $d$-dimensional Riemannian manifold. In \cite{MR4547405}, the authors investigate the SNLSE with strong damping, proving the existence of a unique invariant measure for sufficiently large damping coefficients. The work \cite{MR4249116} addresses global well-posedness for an SNLSE in Marcus form in $L^2(\mathbb{R}^d)$, introducing a new version of the stochastic Strichartz estimate for convolutions. Finally, \cite{MR4678996} studies a defocusing SNLSE with linear damping and a combination of linear Stratonovich and nonlinear Itô noise, constructing a martingale solution via a modified Faedo-Galerkin method, proving pathwise uniqueness, and thereby establishing the existence of an invariant measure.

	\subsubsection{LDP} 
	The theory of large deviations is a classical area of probability theory with broad applications in physics and engineering. Large deviation principles (LDP) study the exponential decay of probabilities, with respect to some relevant parameter, of rare events, events that are highly unlikely but can have catastrophic consequences if they occur. The theory characterizes the asymptotic behavior of the laws of solutions as the noise intensity tends to zero, using a rate function. Large deviation methods are particularly useful for analyzing randomly perturbed dynamical systems, for example, in estimating the first exit time from a neighborhood of an asymptotically stable equilibrium, determining the exit location, or describing transitions between distinct equilibria. Such questions naturally arise in fields like statistical mechanics, quantum mechanics, and chemical reaction modeling. This area gained substantial attention following the foundational work of Stroock \cite{MR755154} and Varadhan \cite{MR758258}. Budhiraja and Dupuis later introduced the weak convergence approach in \cite{MR1785237}, which has become a powerful tool, often simplifying LDP proofs compared to classical techniques. For more details, see \cite{MR2571413, MR1431744}. In recent decades, several researchers have applied the weak convergence approach to establish the LDP for classes of infinite-dimensional stochastic differential equations \cite{MR1135925, hong2021freidlin, kumar2023well, MR4877479, MR2575313, MR4258867}.

	The study of large deviation principles (LDP) for the Navier-Stokes equations (NSE) under various noise structures is well established. In \cite{MR2269220}, a Wentzell-Freidlin type LDP is investigated for the 2D NSE with multiplicative noise on both bounded and unbounded domains. Related contributions include \cite{MR2982734, MR1377394, MR2330725, MR4760311}. The LDP for 2D stochastic NSE driven by L\'evy noise has been studied in \cite{MR4546792, MR2541279, MR3378470}.
	In \cite{MR2519356}, the authors consider a Boussinesq model for B\'enard convection, a coupled system of stochastic NSE and a stochastic transport equation for temperature with additive noise, and study the LDP using the weak convergence approach. The work \cite{MR2609596} addresses a class of abstract stochastic nonlinear models encompassing many 2D hydrodynamical systems, including the 2D NSE, 2D magnetic Bénard problem, 2D magneto-hydrodynamic (MHD) models, and certain shell models of turbulence. For this abstract class in unbounded domains, the LDP is established (avoiding compact embeddings) under slightly stronger time regularity assumptions on the diffusion coefficient.
	The article \cite{MR3062446} focuses on the LDP for the 2D stochastic quasi-geostrophic equation on a periodic domain. LDP results for stochastic reaction-diffusion equations can be found in \cite{MR2044675, MR4805051}, while \cite{MR2570011, MR3080986, MR2525514} study the LDP for stochastic shell models of turbulence with small multiplicative noise. For additional references, see the books \cite{MR2295103, MR3236753, MR1465436}.

	\subsubsection{LDP for SNLSE:} 
	In the context of the SNLSE, relatively little work has been done on large deviation principles. One of the earliest contributions is \cite{MR2148961}, where the author studies the SNLSE with additive Gaussian noise and establishes the LDP using classical methods. This line of work was later extended to the uniform LDP with multiplicative Gaussian noise \cite{MR2178501} and the LDP with fractional additive noise \cite{MR2318412}. In \cite{MR4256725}, the authors investigate the LDP for the SNLSE with multiplicative noise in one dimensional setting with $H^1$ regularity on the initial data. However, due to the lack of compactness and the cancellation of the diffusion term when applying Itô’s lemma, they were unable to prove the LDP in the full solution space. Instead, they established it in the space $C([0, T]; H)$ using the weak convergence approach. The work \cite{zhu2023large} represents the first study of the LDP for the SNLSE with nonlinear noise. Considering the SNLSE with multiplicative pure jump noise in the Marcus canonical form with a nonlinear coefficient, they establish a Freidlin-Wentzell type LDP. The particular structure of this noise preserves a conservation law, which allows them to prove the LDP in the entire solution space.

	\subsection{Novelties and difficulties}
	The main challenge in our problem arises from the noise structure, which combines both linear Stratonovich and nonlinear It\^o components. A second difficulty stems from the power-type nonlinearity and the lack of a regularizing effect from the Schr\"dinger operator. To overcome these issues, we employ the Yosida approximation operator, along with Strichartz estimates and local smoothing estimates, which provide the necessary control over the nonlinearity and compensate for the limited regularization of the Schr\"odinger operator.

	Another challenge in our analysis is the absence of compact embeddings as a Gelfand triple. In contrast, the methods used to establish LDP in \cite{MR4140079, MR4657197, MR4806981} rely heavily on such compactness conditions. Consequently, these approaches cannot be applied to prove LDP in the full solution space in our case. Similarly, in \cite{MR2570011, MR2982734, MR4256725}, the LDP was established only in the space $C([0, T]; H)$, rather than in the full solution space $C([0, T]; H) \cap L^p(0, T; H^1)$, due to the lack of compact embeddings. To overcome this difficulty, we adopt the mild solution approach introduced by Zhu et al. in \cite{zhu2023large}. In \cite{MR4256725}, the authors establish the LDP for the SNLSE with multiplicative noise in the one-dimensional case, assuming $H^1$ regularity for the initial data. In contrast, our result is valid in arbitrary spatial dimensions $d$ for any $\alpha \in (1, \tfrac{4}{d}+1)$. This restriction on $\alpha$ arises from the well-posedness of the skeleton equation and the stochastic controlled equation. Moreover, we work with less regular initial data, requiring only $u_0 \in H$. Nevertheless, establishing the LDP within this framework remains a significant challenge.

	The final challenge in proving the LDP is the need to verify the weak convergence criterion, which requires uniform boundedness of the solution to the controlled stochastic equation in the relevant solution spaces. Establishing the existence of a global solution to the controlled stochastic equation is itself nontrivial. Unlike \cite{zhu2023large}, in our setting the conservation law does not hold, which complicates establishing uniform bounds in $L^p(0, T; L^r(\mathbb{R}^d))$. To address this, we employ a truncation function, stopping time arguments, and the Yosida regularization operator to obtain uniform bounds for the solution in $L^p(0, T; L^r(\mathbb{R}^d))$, following ideas in \cite{MR4249116, MR1706888}. However, in those works, the conservation law plays a central role in obtaining the results, whereas in our case, the presence of damping and Itô noise prevents such conservation. A more careful analysis is therefore required. Notably, we prove the conservation law in the special case $\beta = 0, G \equiv 0$. Our current work is inspired by \cite{zhu2023large} and builds on the weak convergence approach to LDP developed by Budhiraja and Dupuis \cite{MR1785237}.

	\subsection{Organization of the paper}
	The paper is organized as follows. In \Cref{sectionformulation}, we present the abstract formulation of the mathematical model in detail. Subsection \ref{sub2space} introduces the notation for the relevant functional spaces, while Subsection \ref{Ass} states the necessary assumptions. In Subsection \ref{completeito}, we provide a step-by-step derivation of the Stratonovich correction term for \eqref{perturbed}. \Cref{preliminaries} offers a concise overview of definitions and basic results related to large deviation principles (LDP). This section concludes with a general criterion for the LDP (\Cref{generalcriterialemma}), followed by the statement of the main result of this article (\Cref{mainresultth}).

	\Cref{sectionskeleton} focuses on establishing the existence of a unique global mild solution to the skeleton equation. We begin by defining an integral operator on a suitably chosen Banach space and show that it is a strict contraction. By the Banach fixed-point theorem, this operator admits a unique fixed point, which corresponds to the unique local mild solution of the skeleton equation on $[0, T_0]$ for some $T_0 > 0$. Next, we employ the Yosida approximation technique along with Strichartz estimates to extend this result and obtain a global mild solution on $[0, T]$ for any $T > 0$. Additionally, \Cref{sub4conservation} provides a brief review of the \emph{conservative case}, particularly when $\beta = 0$ and $G \equiv 0$.

	In \Cref{sectionsce}, we establish the existence of a unique global mild solution for a stochastic controlled equation. Due to the system’s nonlinearity, we first consider a truncated version and employ suitable stopping time arguments. Solvability is then obtained by applying a stochastic adaptation of the Banach fixed-point theorem in a carefully chosen Banach space. This approach allows us to prove the local existence of a mild solution to the truncated equation on $[0, T_0]$ for some $T_0 > 0$ (\Cref{loctrun}), which is then extended to a unique global-in-time mild solution via mathematical induction (\Cref{truncatedglobalexistence}). Next, using stopping time arguments, we establish the local existence of a solution for the original equation (\Cref{locoriginal}), and finally, the global existence is obtained through Yosida approximation (\Cref{globoriginal}).

	Finally, in \Cref{sectionldp}, we establish the Freidlin-Wentzell type LDP for the family of solutions in a suitable Polish space, this constitutes our main result, using the weak convergence approach. This method relies on the equivalence between the LDP and the Laplace principle in Polish spaces, which follows from Varadhan’s Lemma (\Cref{lemvaradhan}) and Bryc’s converse (\Cref{lembryc}). For further details, we refer the reader to \cite{MR1785237}.
	
	In \Cref{sectionappendix}, we collect several technical results that are used throughout the manuscript to support the proofs of our main results.
	
	\subsection{Open questions}
	\begin{itemize}
		\item[(A)]
		The mathematical structure of the Schr\"odinger equation is closely related to diffusion equations and can be reformulated by considering Gaussian probability measures over particle trajectories in the space of all possible paths of a quantum particle \cite{MR2077514}. Within this framework, replacing the Gaussian measure with non-Gaussian distributions, such as L\'evy $\alpha$-stable laws, leads to generalized Schrödinger equations, including time- or space-fractional formulations. This approach naturally accommodates the heavy-tailed jump statistics characteristic of L\'evy processes. \emph{A natural question then arises: is it possible to study the LDP for such systems?}

		\item[(B)] {\it Is it possible to establish a uniform LDP for such systems under the general noise considered in the present work?}
		\item[(C)] {\it Our future goal is to extend the analysis of this paper to the one-dimensional rough SNLSE, where stochastic integration is interpreted in the sense of a controlled rough path.}
	\end{itemize}
	
	\section{Formulation of mathematical framework and main result}\label{sectionformulation}
	In this section, we define some basic functional spaces and their properties. Also, we take the assumptions on the operators and formulate the mathematical framework for our problem. We recall some basic definitions and preliminary results. Then we present a general criteria and state the main result of our work.
	\subsection{Functional spaces and basic definitions}\label{sub2space}
	We denote the space of equivalence classes of $\mathbb{C}$-valued $p$-Lebesgue integrable functions by $L^p(\mathbb{R}^d)$, for $p \in [1,\infty)$. For $p=2$, we use the notation $L^2(\mathbb{R}^d)=H$. The space $L^p(0,T;L^r(\mathbb{R}^d))$ consists of all strongly measurable functions $f : [0, T] \to L^r(\mathbb{R}^d)$ with the norm defined by,
	\begin{align*}
		\|f\|_{L^p(0,T;L^r(\mathbb{R}^d))} := \Big(\int_{0}^{T}\|f(t)\|^p_{L^r(\mathbb{R}^d)}dt\Big)^{\frac{1}{p}}, \quad \text{for}\,\, 1\leq p< \infty,
	\end{align*}
	and 
	\begin{align*}
		\|f\|_{L^\infty(0,T;L^r(\mathbb{R}^d))} :=\displaystyle \esssup_{0\leq t\leq T}  \|u(t)\|_{L^r(\mathbb{R}^d)}.
	\end{align*}
	By $C([0, T];H)$, we represent the space of all continuous functions $f : [0, T] \to H$ endowed with the uniform convergence topology. The space $L^p(\Omega;X)$ is the collection of all measurable functions $f : \Omega \to X$ with the norm defined by (see \cite[5.9.2]{MR2597943}), 
	\begin{align*}
		\|f\|_{L^p(\Omega;X)}:= \Big(\int_{\Omega}\|f(\omega)\|^p_{X}d\mathbb{P}\Big)^{\frac{1}{p}}, \quad \text{for}\,\, 1\leq p< \infty.
	\end{align*}
	Let $U$ and $V$ be two separable Hilbert spaces and the space of all bounded linear operators on $U$ is denoted by $\mathscr{L}(U)$. The space of all Hilbert-Schmidt operators from $U$ to $V$ is represented by $\mathcal{L}_2(U,V)$ equipped with the norm,
	\begin{align*}
		\|S\|_{\mathcal{L}_2}= \Big(\sum_{k=1}^{\infty} |Se_k|^2\Big)^{\frac{1}{2}},\quad \text{ for } S \in \mathcal{L}_2(U,V),
	\end{align*}
	where $\{e_k\}_{k\in \mathbb{N}}$ is a complete orthonormal basis of the Hilbert space $U$.
	\subsection{Assumptions}\label{Ass}
	\begin{itemize}
		\item[\textbf{(a)}] \textbf{Assumptions on the operator $A$:} We assume the operator A is $-\Delta$ (as an operator defined on $H$), negative Laplacian operator.\label{AssA}
		\item[\textbf{(b)}] \textbf{Assumptions on the nonlinear term $\mathcal{N}$:} Assume that \( \alpha \in (1, \infty) \) and consider the polynomial nonlinearity $$\mathcal{N}(u) = \lambda|u|^{\alpha - 1} u.$$ It is called defocusing if the coefficient $\lambda>0$ and focusing if $\lambda<0$. Since the estimates are independent of the sign of $\lambda$, without loss of generality we set $\lambda = 1$ and all the mathematical arguments remain valid. Therefore our result is true for both the focusing and defocusing cases and from now on, we work with $$\mathcal{N}(u) =|u|^{\alpha - 1} u.$$\label{AssF}	
		\item[\textbf{(c)}] \textbf{Assumptions on the noise and its coefficients:}\label{AssN}
		\begin{itemize}\item 	[(i)] We consider a filtered probability space \((\Omega, \mathcal{F}, \mathbb{F}, \mathbb{P})\) satisfying the usual conditions, where \(\mathbb{F} = \{ \mathcal{F}_t \}_{t \geq 0}\). Let \(Y_1\) and \(Y_2\) be two separable real Hilbert spaces, with orthonormal bases \(\{e^1_m\}_{m \in \mathbb{N}}\) and \(\{e^2_m\}_{m \in \mathbb{N}}\), respectively. Moreover, $\mathcal{W}_1$ and $ \mathcal{W}_2$ are two independent canonical cylindrical $\mathbb{\mathcal{N}}$-Weiner processes $Y_1$ and $ Y_2$ valued, respectively. \\
			\item [(ii)] Let \(B : H \to \mathcal{L}_2 (Y_1, H)\) be a linear operator and consider a self-adjoint operator   \(B_m \in \mathscr{L}(H)\) such that  \(B_m u = B(u) e^1_m\), for every \(u \in H\) and \(m \in \mathbb{N}\). We assume (to make sense of the Stratonovich correction terms),
			\begin{equation}\label{BmH}
				\sum_{m=1}^{\infty} \|B_m\|^2_{\mathscr{L}(H)} < \infty. 
			\end{equation}
			\item 	[(iii)] We consider the nonlinear coefficient \(G : H \to \mathcal{L}_2 (Y_2, H)\) as the function $G$ has the form
			\begin{align}\label{perticularg}
				G(u)=\tilde{G}(\|u\|_{H}^2)u,
			\end{align}
			where $\tilde{G}:[0, \infty) \to \mathbb{R}$ is a measurable real-valued function. \\
			\noindent We also cosider, $G$ is Lipschitz continuous, that is, there exists a constant \(L_G > 0\) such that:
			\begin{equation}\label{LipG}
				\|G(u_1) - G(u_2)\|_{\mathcal{L}_2 (Y_2,H)} \leq L_G \|u_1 - u_2\|_H, \quad \text{for all } u_1, u_2 \in H.
			\end{equation}
			Lipschitz continuity implies the linear growth condition for $G$, i.e., there exist positive constants $C_1, C_2$ such that 
			\begin{equation}\label{LgG}
				\|G(u)\|_{\mathcal{L}_2 (Y_2, H)} \leq C_1 +C_2 \|u\|_H, \quad \text{for all } u \in H.
			\end{equation}
		Moreover, we assume the following weak continuity assumption of the diffusion coefficient $G$: for every $m \in \mathbb{N}$ the map
				\begin{align*}
					H \ni \varphi \mapsto G(\varphi)e^2_m \in H
				\end{align*}
				extends uniquely to a continuous map from $H^{-1}$ to $H^{-1}$, i.e.
				\begin{align}\label{wcofg}
					H^{-1} \ni \varphi \mapsto G(\varphi)e^2_m \in H^{-1}\ \quad\text{is continuous}.
			\end{align}
			\begin{rem}
			The class of functions (known as saturation nonlinearities) satisfying \eqref{perticularg},  \eqref{LipG}, \eqref{LgG} and \eqref{wcofg} simultaneously remains quite extensive, encompassing a wide range of nonlinear (as well as linear) mappings. For example, in the representation \eqref{perticularg}, one may choose the function $\tilde{G}$ as 
			\begin{align*}
				\tilde{G}(y)=\frac{y}{1+ay}\quad \text{for } a>0,
			\end{align*}
			which characterizes photorefractive media \cite{PhysRevLett.15.1005}. Another example is
			\begin{align*}
				\tilde{G}(y)=1 - \frac{1}{(1+y)^{1/2}},
			\end{align*}
			 which describes narrow-gap semiconductors \cite[equation 4]{PhysRevB.61.10201}. One can also choose 
			 \begin{align*}
			 	\tilde{G}(y) =e^{-y}\quad \text{or} \quad \tilde{G}(y) = \frac{y(ay+2)}{(1+ay)^2}
			 	\quad \text{or} \quad
			 	\tilde{G}(y) = \frac{\log(1+ay)}{1+\log(1+ay)},
			 	\quad y \ge 0, \ a > 0.
			 \end{align*}
			 The first Hermite function $H(y)=\frac{\sqrt{2}\,y}{\pi^{1/4}} e^{-\frac{|y|^2}{2}}$, which has very important application in quantum mechanics, is associated to  $\tilde{G}(y) =e^{-y}$, with some modification.
		\end{rem}

			\vspace{2cm}
			By the above assumption, we can represent the Wiener processes as 
			\begin{align}\label{BasisW}
				\mathcal{W}_1(t) = \sum_{m=1}^{\infty} e^1_m \beta^1_m(t) \,\, \text{and}\,\, 	\mathcal{W}_2(t) = \sum_{m=1}^{\infty} e^2_m \beta^2_m(t), \quad t\geq0.
			\end{align}
			where \( \{\beta^1_m\}_{m\in\mathbb{N}} \) and \( \{\beta^2_m\}_{m\in\mathbb{N}} \) are sequences of independent standard real valued Wiener processes.
		\end{itemize}
	\end{itemize}
	
	\subsection{Complete It\^ {o} form of the equation}\label{completeito}
	We can transform the Stratonovich noise into It\^ {o} form as,
	\begin{align*}
		- i\sqrt{\varepsilon}  B(u(t)) \circ d\mathcal{W}_1(t) &= - i \sqrt{\varepsilon}  B(u(t)) d\mathcal{W}_1(t) +
		\frac{1}{2} \sum_{m=1}^{\infty} -i\sqrt{\varepsilon}  B'[u]\big(-i\sqrt{\varepsilon}  B(u(t))e^1_m\big)e^1_m dt\\
		&=- i\sqrt{\varepsilon}  B(u(t)) d\mathcal{W}_1(t) -
		\frac{\varepsilon}{2} \sum_{m=1}^{\infty}B\big(B(u(t))e^1_m\big)e^1_m dt\\
		&=- i\sqrt{\varepsilon}  B(u(t)) d\mathcal{W}_1(t) -
		\frac{\varepsilon}{2} \sum_{m=1}^{\infty}B_m^2 u(t)dt,
	\end{align*}
	where $B'[u](h)$ denotes the Fr\'echet derivative of $B$ in the direction of $h$. For more information, we  refer to see \cite{goodair2025stratonovich}.\\
	Hence, equation \eqref{perturbed} can be rewritten in the complete It\^{o} form,
	\begin{equation}
		du(t)=-\big[i Au(t) +i \mathcal{N}(u(t))+ \beta u(t) + \varepsilon b(u(t))\big]dt - i\sqrt{\varepsilon}  B(u(t)) d\mathcal{W}_1(t)- i\sqrt{\varepsilon}  G(u(t)) d\mathcal{W}_2(t) , \quad t\ge0 
	\end{equation}
	where the  Stratonovich correction term is, \[b(u(t))= 	\frac{1}{2} \sum_{m=1}^{\infty}B_m^2 u(t).\]
	\subsection{Preliminaries:}\label{preliminaries}
	In this subsection, we recall some basic definitions and preliminary results. The following definitions are taken from \cite{MR2435853}. Let $\mathcal{E}$ be a Polish space (complete separable metric space) with the Borel $\sigma$-field $\mathcal{B}(\mathcal{E})$. 
	\begin{defi}[Rate function]\label{ratefunc}
		A function $I : \mathcal{E} \to [0,\infty]$ is called a \emph{rate function} if $I$ is lower semicontinuous. A rate function $I$ is a \emph{good rate function} if for arbitrary $M \in [0,\infty)$, the level set $K_M := \big\{x \in \mathcal{E} : I(x) \leq M\big\}$ is compact in $\mathcal{E}$.
	\end{defi}
	\begin{defi}[Large deviation principle]\label{ldp}
		A family $\{X^\varepsilon \}_{\varepsilon > 0}$ of $\mathcal{E}$-valued random elements is said to satisfy the \emph{large deviation principle }(LDP) on $\mathcal{E}$ with a good rate function $I : \mathcal{E} \to [0,\infty]$ if:
		\begin{enumerate}
			\item For each closed set $F$ of $ \mathcal{E}$,
			\begin{align*}
				\limsup_{\varepsilon \to 0} \varepsilon \log {\mathbb{P}}( X^\varepsilon \in F) \leq - \inf_{x \in F} I(x),
			\end{align*}
			\item For each open set $G$ of $ \mathcal{E}$,
			\begin{align*}
				\liminf_{\varepsilon \to 0} \varepsilon \log {\mathbb{P}}(X^\varepsilon \in G) \geq - \inf_{x \in G} I(x).
			\end{align*}	
		\end{enumerate}
	\end{defi}
	\begin{defi}[Laplace principle]\label{laplaceprinciple}
		Let $I$ be a rate function on $\mathcal{E}$. A family $\{X^\varepsilon \}_{\varepsilon > 0}$ of $\mathcal{E}$-valued random elements is said to satisfy the {\emph{Laplace principle}} on $\mathcal{E}$ with rate function $I$ if for each real-valued, bounded and continuous function $h$ defined on $\mathcal{E}$, 
		\begin{align}
			\lim_{\varepsilon \to 0} \varepsilon \log \mathbb{E} \Big\{ \exp \Big[ - \frac{1}{\varepsilon} h(X^\varepsilon) \Big]\Big\}
			= - \inf_{x \in \mathcal{E}}\big \{ h(x) + I(x)\big \}.
		\end{align}
	\end{defi}
	\begin{lem}[Varadhan’s Lemma \cite{MR758258}] Let $\mathcal{E}$ be a Polish space and $\big\{X^\varepsilon\big\}_{\varepsilon > 0}$ be a family of $\mathcal{E}$-valued random elements satisfying LDP with rate function $I$. Then $\big\{X^\varepsilon\big\}_{\varepsilon > 0}$ satisfies the Laplace principle on $\mathcal{E}$ with the same rate function $I$.
	\end{lem}\label{lemvaradhan}
	\begin{lem}[Bryc’s Lemma \cite{MR1431744}] The Laplace principle implies the LDP with the same rate function. More precisely, if $\big\{X^\varepsilon\big\}_{\varepsilon > 0}$ satisfies the Laplace principle with a rate function $I$ on a Polish space $\mathcal{E}$ and the limit
		\begin{align*}
			\lim_{\varepsilon \to 0} \varepsilon \log \mathbb{E} \Big\{ \exp[-\frac{1}{\varepsilon} h(X^\varepsilon)] \Big\} = -\inf_{x \in \mathcal{E}} \big\{ h(x) + I(x) \big\}
		\end{align*}
		is valid for all bounded continuous functions $h$, then $\big\{X^\varepsilon\big\}_{\varepsilon > 0}$ satisfies the LDP on $\mathcal{E}$ with rate function $I$.
	\end{lem}\label{lembryc}
	\subsection{General criteria}\label{generalcriteria}
	In this subsection, we introduce sufficient conditions for a sequence of Wiener functionals to satisfy the
	Laplace principle.
	Let  \( \mathbb{S} \) be the collection of control functions defined as,
	\begin{align*}
		\mathbb{S} := \Big\{ \rho: \rho \text{ is } Y_1 \times Y_2\text{\,-valued \( \mathbb{F} \)-predictable process and } 	\|\rho\|_{L^\infty(\Omega;L^2(0, T; Y_1 \times Y_2) )} < \infty\Big\}
	\end{align*}
	
	where,
	\begin{equation*}
		\|\rho\|_{L^\infty(\Omega;L^2(0, T; Y_1 \times Y_2) )}^2 = \esssup_{\omega \in \Omega} \Big\{ \int_0^T \| \rho(t, \omega) \|^2_{Y_1 \times Y_2} \, dt \Big\}
		=  \esssup_{\omega \in \Omega} \Big\{ \int_0^T \Big( \| \rho_1(t, \omega) \|^2_{Y_1} + \| \rho_2(t, \omega) \|^2_{Y_2} \Big) \, dt \Big\}. 
	\end{equation*}
	
	For $r \in \mathbb{N}$ define, 
	\begin{align*}
		D^1_r:= \Big\{ \rho_1 \in L^2(0, T; Y_1) : \int_0^T \| \rho_1(t) \|^2_{Y_1} \, dt \leq r \Big\}, 
	\end{align*}
	\begin{align*}
		D_r^2 :=\Big\{ \rho_2 \in L^2(0, T;Y_2) : \int_0^T \| \rho_2(t) \|^2_{Y_2} \, dt \leq r\Big\},
	\end{align*}
	and denote
	\begin{align*}
		\mathbb{D}_r := D_r^1 \times D_r^2 \quad \text{and} \quad \mathbb{D} := \bigcup_{r \geq 1} \mathbb{D}_r.
	\end{align*}
	On \(\mathbb{D}_r\), we consider the topology induced by the weak topology on the Hilbert space \(L^2(0, T; Y_1 \times Y_2)\). Note that this topology is metrizable.
	Define, \begin{align*}
		\mathbb{S}_r := \Big\{ \rho=(\rho_1, \rho_2) \in \mathbb{S} : \rho(\omega) \in \mathbb{D}_r,\quad \text{for a.e.}\,\,\,\omega \in \Omega\Big\}.
	\end{align*}
	Note that, \begin{align*}
		\mathbb{S} = \bigcup_{r \geq 1} \mathbb{S}_r.
	\end{align*}
	Here, we state a general criterion for the LDP (established by Budhiraja and Dupuis in \cite{MR1785237}), modified according to the context of our interest (one can see  \cite[Lemma 2.1]{MR4806981} also):
	\begin{lem} \label{generalcriterialemma}
		Assume that $\mathcal{G}^\varepsilon : C\big([0,T]; Y_1 \times Y_2\big) \to \mathcal{E}$ is a family of measurable mappings. If $X^\varepsilon = \mathcal{G}^\varepsilon(W)$ and there exists a measurable map $\mathcal{G}^0 : C\big([0,T]; Y_1 \times Y_2\big) \to \mathcal{E}$ such that the following two conditions hold:
		\begin{itemize}
			\item[(a)] Let $\{\rho^\varepsilon\}_{\varepsilon > 0} \subset \mathbb{S}_M$, $M < \infty$. If $\rho^\varepsilon$ converges to $\rho$ in distribution as $\mathbb{D}_M$-valued random elements, then
			$$\mathcal{G}^\varepsilon\bigg( W + \frac{1}{\sqrt{\varepsilon}} \int_0^\cdot \rho_s^\varepsilon \, ds \bigg)
			\to \mathcal{G}^0 \bigg( \int_0^\cdot \rho_s \, ds \bigg)$$
			in distribution as $\varepsilon \to 0$.
			\item[(b)] For any $M < \infty$, the set
			$$K_M := \bigg\{ \mathcal{G}^0 \left( \int_0^\cdot \rho_s \, ds \right) : \rho \in \mathbb{D}_M \bigg\}$$
			is a compact subset in $\mathcal{E}$.
		\end{itemize}
		Then the family $\big\{X^\varepsilon\big\}_{\varepsilon > 0}$ satisfies the Laplace principle (hence LDP) on $\mathcal{E}$ with the following good rate function
		$$	I(f) = \inf_{\substack{\rho \in L^2(0,T; Y_1 \times Y_2):\\ f = \mathcal{G}^0(\int_0^\cdot \rho_s ds)}} \bigg\{ \frac{1}{2} \int_0^T \|\rho_s\|_{Y_1 \times Y_2}^2 \, ds \bigg\}.$$
	\end{lem}
	\subsection{Main result}\label{mainresult}
	The main result of this work is stated below and is proven in the subsequent sections. 
	\begin{thm}\label{mainresultth}
		For \(1< \alpha< \frac{4}{d}+1, r = \alpha+1\), the family of solutions $\{u^\varepsilon \}_{\varepsilon>0}$ to \eqref{perturbed} satisfies the LDP on the Polish space $\mathcal{E}:=C([0, T]; H) \cap L^p(0, T; L^r(\mathbb{R}^d))$ with the good rate function $I$ given by
		\begin{align*}
			I(f):= \inf_{\rho \in L^2(0,T; Y_1\times Y_2)} \Big\{\frac{1}{2}\int_{0}^{T}\|\rho(s)\|_{Y_1 \times Y_2}^2 ds : u^\rho=f\Big\},
		\end{align*}
		for $f \in C([0, T]; H) \cap L^p(0, T; L^r(\mathbb{R}^d))$,	where $u^\rho$ is the unique global mild solution (in the sense of \Cref{mildskl}) of the skeleton equation \eqref{Skeleton}.
	\end{thm} 
	Before moving to the proof of the main result, we will prove the well-posedness of the skeleton equation and stochastic controlled equation.
	\section{Skeleton equation}\label{sectionskeleton}
	In this section, we will prove the global existence of a unique mild solution to the skeleton equation \eqref{Skeleton}. To do that, first, we prove the local existence of a mild solution by defining an integral operator and using the Banach fixed point theorem. Then, by using the Yosida approximation and the Strichartz estimates we prove the existence of a global solution. Finally, we establish the uniqueness of the global solution.
	\subsection{Well-posedness of skeleton equation}\label{wellposeskl}
	For \( \rho = (\rho_1,{\rho_2}) \in \mathbb{D}\), we consider the following  deterministic controlled equation (skeleton equation),
	\begin{align}
		\left\{
		\begin{aligned}
			du^\rho(t)&=-\big[i Au^\rho(t) +i \mathcal{N}(u^\rho(t))+ \beta u^\rho(t)\big]dt - i B(u^\rho(t))\rho_1(t)dt\\
			& \quad- i  G(u^\rho(t))\rho_2(t)dt , \quad t>0,	\\
			u^\rho(0)&=u_0.
		\end{aligned}
		\right. \label{Skeleton}
	\end{align}
	Our first aim is to prove the existence and uniqueness of a global mild solution of the skeleton equation (\ref{Skeleton}). Let us define the mild solution.
	\begin{defi}[Mild solution]\label{mildskl}
		Given any $u_0 \in H$, a \emph{ mild solution} to the skeleton equation (\ref{Skeleton}) is a function $u^\rho = \big\{u^\rho(t): t\in [0,T]\big\} \in C([0, T]; H) \cap L^p(0, T; L^r(\mathbb{R}^d))$ satisfying 
		\begin{align}
			u^\rho(t)&=S(t)u_0	- \int_{0}^{t} S(t-s) \big[i \mathcal{N}(u^\rho(s))+ \beta u^\rho(s)\big]ds - i\int_{0}^{t} S(t-s)\big[B(u^\rho(s))\rho_1(s)+G(u^\rho(s))\rho_2(s) \big]ds. \label{mildformsk}
		\end{align}
	\end{defi}
	\begin{defi}[Admissible pair]\label{Admissible}
		A pair \((p, r) \in [2, \infty]\times [2, \infty] \) is said to be \emph{admissible} if
		\begin{align*}
			\frac{2}{p} = \frac{d}{2} - \frac{d}{r}, \quad (p, r, d) \neq (2, \infty, 2).
		\end{align*}
	\end{defi}
	\noindent With this definition, in different dimensions, we have 
	\begin{align*}
		\left\{
		\begin{aligned}
			&2\leq r\leq\infty, \quad &&\text{if} \quad d=1,\\
			& 2\leq r<\infty, \quad &&\text{if} \quad d=2,	\\
			&2\leq r\leq \frac{2d}{d-2},\quad &&\text{if} \quad d \geq3.
		\end{aligned}
		\right.
	\end{align*}
	Note that, $(\infty,2)$ is always admissible. Now, we state the Strichartz estimates \cite[Theorem 2.3.3]{MR2002047}, which will be used several times in our proof:
	\noindent\begin{prop}[Strichartz's estimates] \label{Strichartz}
		Let \((p, r)\) and \((\gamma, \sigma)\) be admissible pairs and $\gamma'$, $\sigma'$ are conjugate of $\gamma, \sigma$ i.e., $\frac{1}{\gamma}+\frac{1}{\gamma'}=1$ and $\frac{1}{\sigma}+\frac{1}{\sigma'}=1$.
		\begin{enumerate}
			\item[(i)] For any \( \phi \in H \), the function \( t \mapsto S_t \phi \) belongs to \( L^p(\mathbb{R}; L^r(\mathbb{R}^d)) \cap L^\infty(\mathbb{R}; H) \) and there exists a constant \( C \) such that
			\begin{equation*}
				\|S_\cdot \phi\|_{L^p(\mathbb{R}; L^r(\mathbb{R}^d))} \leq C \|\phi\|_{H}. 
			\end{equation*}
			
			\item[(ii)] Let \( I \) be an interval of \( \mathbb{R} \) and \( 0 \in {\overline{I}} \). Then for every \( g \in L^{\gamma'}(I; L^{\sigma'}(\mathbb{R}^d)) \), the function \( t \mapsto \Phi_g(t) = \int_0^t S_{t-s} g(s) \, ds \) belongs to \( L^p(I; L^r(\mathbb{R}^d)) \cap L^\infty({\overline{I}}; H) \) and there exists a constant \( C \) independent of \( I \) such that
			\begin{equation}\label{Strichartzinf}
				\|\Phi_g\|_{L^\infty({\overline{I}}; H)} \leq C \|g\|_{L^{\gamma'}(I; L^{\sigma'}(\mathbb{R}^d))}, 
			\end{equation}
			\begin{equation}\label{Strichartzp}
				\|\Phi_g\|_{L^p(I; L^r(\mathbb{R}^d))} \leq C \|g\|_{L^{\gamma'}(I; L^{\sigma'}(\mathbb{R}^d))}.
			\end{equation}
		\end{enumerate}
	\end{prop}
	\begin{proof}
		For a detailed proof, we refer to \cite[Theorem 2.3.3]{MR2002047}.
	\end{proof}
	Given an admissible pair \((r, p)\) and \(s, t\) such that \(0 \leq s < t\), denote the space
	\begin{align*}
		\mathbb{L}^{\infty, p}_{2, r}(s,t) := L^\infty(s,t; H) \cap L^p(s, t; L^r(\mathbb{R}^d)),
	\end{align*}
	with the norm,
	\begin{align}\label{normassum}
		\|u\|_{\mathbb{L}^{\infty, p}_{2, r}(s,t)}= \|u\|_{L^\infty(s,t; H)}+\|u\|_{ L^p(s, t; L^r(\mathbb{R}^d))}.
	\end{align}
	For simplicity, we denote \(\mathbb{L}^{\infty, p}_{2, r}(0,T)\) as \(\mathbb{L}^{\infty, p}_{2, r}(T)\). The next theorem provides the existence of a global mild solution. Throughout the paper, $C$ denotes a generic positive constant, possibly varying from line to line. Furthermore, we write $C_*$, when the dependence on `$*$' is essential.
	\begin{thm}\label{wellskl}
		Let \(1< \alpha< \frac{4}{d}+1, r = \alpha+1\) and \((p, r)\) be an admissible pair. For
		any \(\rho \in \mathbb{D}\), there exists a unique global mild solution \(u^\rho = \big\{u^\rho(t): t\in [0,T]\big\} \in C([0, T]; H) \cap L^p(0, T; L^r(\mathbb{R}^d))\) to the skeleton equation \eqref{Skeleton}.	
	\end{thm} 	
	\begin{proof}	
		In this proof, we will show the existence of a unique global mild solution (in the sense of \Cref{mildskl}) to the skeleton equation \eqref{Skeleton}. The proof is split into three main steps. First, we define an integral operator on an appropriate space and use the Banach fixed point theorem to obtain the local existence of a solution. In the second step, we have shown the uniqueness of the solution (global in time) if any two solutions satisfying \eqref{Skeleton} exist with the same initial data, then they are identical. Finally, we have achieved the global existence of a solution by using Yosida approximation and the Strichartz estimates.
		\begin{steps}
			\item {\textbf{(Local existence of the solution):}}\label{stepaskl} 	We consider the set
			\begin{equation}\label{localset}
				V_{M,\tilde{T}} :=
				\Big\{
				u \in \mathbb{L}^{\infty, p}_{2, r}(\tilde{T}) : \|u\|_{L^\infty(0,\tilde{T}; H)} + \|u\|_{L^p(0, \tilde{T}; L^r(\mathbb{R}^d))}
				\leq M
				\Big\}, \quad \tilde{T} \in (0,T] \text{ and } M>0,
			\end{equation}
		where $M=5C\|u_0\|_{H}$ with $C$ chosen appropriately later.
			Define the mapping $\mathfrak{T}: V_{M,\tilde{T}} \to  V_{M,\tilde{T}}$ by,
			\begin{align}\label{Iosk}
				&\mathfrak{T}(u^\rho, u_0)(t)\nonumber\\
				&:=S(t)u_0	- \int_{0}^{t} S(t-s) \big[i \mathcal{N}(u^\rho(s))+ \beta u^\rho(s)\big]ds - i\int_{0}^{t} S(t-s)\big[B(u^\rho(s))\rho_1(s)+ G(u^\rho(s))\rho_2(s) \big]ds \nonumber\\
				&= S(t)u_0 + I_1(u^\rho)(t)+  I_2(u^\rho)(t),\quad t\in [0,T],
			\end{align}
			where 
			$$I_1(u^\rho)(t):= - \int_{0}^{t} S(t-s) \big[i \mathcal{N}(u^\rho(s))+ \beta u^\rho(s)\big]ds,$$
			and  
			$$	I_2(u^\rho)(t):=-i\int_{0}^{t} S(t-s)\big[B(u^\rho(s))\rho_1(s)+ G(u^\rho(s))\rho_2(s)\big]ds.$$
			First we show that, the operator $\mathfrak{T}$ is well-defined.
			
			\vskip 0.1cm 
			\noindent 	\textbf{Estimates for $I_1$:} 
			Using the Strichartz estimates \eqref{Strichartzinf}, \eqref{Strichartzp} with $(\gamma, \sigma)=(p,r)$ and H\"older's inequality with $s'={\frac{pr'}{p'r}}$ and $t'= \frac{p - 1}{p - r}$, and considering the relation $r'\alpha=r$, we conclude,
			\begin{align}\label{boundF}
				\begin{aligned}
					&\bigg\| 	- \int_{0}^{.} S(.-s) i \mathcal{N}(u^\rho(s))ds \bigg\|_{\mathbb{L}^{\infty, p}_{2, r}(\tilde{T})} \leq C\|  |u^\rho|^{\alpha-1}u^\rho\|_{L^{p'}(0, \tilde{T}; L^{r'}(\mathbb{R}^d))}\\
					&= C \Big( \int_0^{\tilde{T}} \Big( \int_{\mathbb{R}^d} |u^\rho(s)|^{\alpha r'} \, dx \Big)^{p'/r'} \, ds \Big)^{1/p'}
					= C \Big( \int_0^{\tilde{T}} \Big( \int_{\mathbb{R}^d} |u^\rho(s)|^r \, dx \Big)^{p'/r'} \, ds \Big)^{1/p'} \\
					&\leq C\tilde{T}^{ \frac{1}{t'}} \Big( \int_0^{\tilde{T}} \Big( \int_{\mathbb{R}^d} |u^\rho(s)|^r \, dx \Big)^{\frac{p's'}{r'}} \, ds \Big)^{\frac{1}{s'p'}} 
					\leq C\tilde{T}^{ \frac{p - r}{p - 1}} \Big( \int_0^{\tilde{T}} \Big( \int_{\mathbb{R}^d} |u^\rho(s)|^r \, dx \Big)^{p/r} \, ds \Big)^{\frac{r}{pr'}} 
					\\
					& =C\tilde{T}^{ \frac{p - r}{p - 1}} \|u^\rho\|_{L^p(0, \tilde{T}; L^r(\mathbb{R}^d))}^{\alpha},
				\end{aligned}
			\end{align}
			where $p'$ and $r'$ are the conjugate of $p$ and $r$, respectively. It can be easily checked that $$	\frac{1}{s'}+\frac{1}{t'}=\frac{p'r}{pr'}+\frac{p-r}{p-1}=\frac{pr(r-1)}{(p-1)pr}+\frac{p-r}{p-1}=1.$$
			Again, we use the Strichartz estimates \eqref{Strichartzinf}, \eqref{Strichartzp} with $(\gamma, \sigma)=(\infty,2)$ to get,
			\begin{align}\label{boundbeta}
				&\bigg\| - \int_{0}^{.} S(.-s)\beta u^\rho(s)ds\bigg\|_{\mathbb{L}^{\infty, p}_{2, r}(\tilde{T})}\nonumber\\
				&=	\bigg\|  \int_{0}^{.} S(.-s)\beta u^\rho(s)ds\bigg\|_{L^\infty(0,\tilde{T}; H)} + 	\bigg\|  \int_{0}^{.} S(.-s)\beta u^\rho(s)ds\bigg\|_{L^p(0, \tilde{T}; L^r(\mathbb{R}^d))}\nonumber\\
				&\leq \beta \sup_{t \in [0, \tilde{T}]}  \int_{0}^{t}\| S(t-s) u^\rho(s)\|_{H}ds+ C \beta\|u^\rho\|_{L^1(0, \tilde{T}; H)}\nonumber\\
				&\leq \beta\int_0^{\tilde{T}} \|  u^\rho(s)\|_{H}ds+ C \beta\int_0^{\tilde{T}} \|  u^\rho(s)\|_{H} ds\nonumber\\
				&\leq (1+C)\beta \tilde{T} \|  u^\rho\|_{L^\infty(0,\tilde{T}; H)}\leq (1+C)\beta \tilde{T} \|  u^\rho\|_{\mathbb{L}^{\infty, p}_{2, r}(\tilde{T})}.
			\end{align}  
			\vskip 0.1cm 
			\noindent 
			\textbf{Estimates for $I_2$ defined in \eqref{Iosk}\,:} Using the Strichartz estimates \eqref{Strichartzinf}, the Cauchy-Schwartz inequality, Young's inequality and linear growth property (\ref{LgG}) of $G$,  we have the following estimate for $I_2$ as
			\begin{align}\label{boundI2a}
				&\sup_{t \in [0, \tilde{T}]} \|I_2(u^\rho)(t)\|_{H}\nonumber\\
				&=\sup_{t \in [0, \tilde{T}]} \bigg\|\int_{0}^{t} S(t-s)\Big[B(u^\rho(s))\rho_1(s)+ G(u^\rho(s))\rho_2(s)  \Big]ds\bigg\|_{H}\nonumber\\
				&\leq \int_{0}^{\tilde{T}}\Big\| S(t-s)\Big[B(u^\rho(s))\rho_1(s)+ G(u^\rho(s))\rho_2(s)\Big ]\Big\|_{H}ds\nonumber\\
				&\leq C \int_{0}^{\tilde{T}} \Big[\|B(u^\rho(s))\|_{\mathcal{L}_2 (Y_1,H)}\|\rho_1(s)\|_{Y_1}+\|G(u^\rho(s))\|_{\mathcal{L}_2 (Y_2,H)}\|\rho_2(s)\|_{Y_2}\Big]ds\nonumber\\
				&\leq C \int_{0}^{\tilde{T}}\Big[\|B\|_{\mathscr{L}(H,\mathcal{L}_2 (Y_1,H))}\|u^\rho(s)\|_H \|\rho_1(s)\|_{Y_1}+ (C_1+C_2\|u^\rho(s)\|_H )\|\rho_2(s)\|_{Y_2}\Big]ds\nonumber\\
				&\leq C \|u^\rho\|_{\mathbb{L}^{\infty, p}_{2, r}(\tilde{T})} \int_{0}^{\tilde{T}} \big\{\|B\|_{\mathscr{L}(H,\mathcal{L}_2 (Y_1,H))} \|\rho_1(s)\|_{Y_1}+ C_2 \|\rho_2(s)\|_{Y_2} \big\}ds+ C C_1 \int_{0}^{\tilde{T}}\|\rho_2(s)\|_{Y_2}ds.
			\end{align}
			Again, employing the Strichartz estimates \eqref{Strichartzp} with $(\gamma, \sigma)=(\infty,2)$, linear growth property (\ref{LgG}) of $G$, we estimate $I_2$ in $L^p(0, \tilde{T}; L^r(\mathbb{R}^d))$ as
			\begin{align}\label{boundI2b}
				&\|I_2(u^\rho)(\cdot)\|_{L^p(0,\tilde{T};L^r(\mathbb{R}^d))}\nonumber\\
				&= \bigg\|-i\int_{0}^{.} S(.-s)\big[B(u^\rho(s))\rho_1(s)+ G(u^\rho(s))\rho_2(s)  \big]ds\bigg\|_{L^p(0,\tilde{T}L^r(\mathbb{R}^d))}\nonumber\\
				&\leq  C \Big\| \Big[Bu^\rho\rho_1+ G(u^\rho)\rho_2 \Big]\Big\|_{L^1(0,\tilde{T};H)} \nonumber\\
				&\leq  C\int_{0}^{\tilde{T}}\Big\| \Big[B(u^\rho(s))\rho_1(s)+ G(u^\rho(s))\rho_2(s)  \Big]\Big\|_{H}ds\nonumber\\
				&\leq  C \int_{0}^{\tilde{T}}\Big[\|B(u^\rho(s))\|_{\mathcal{L}_2 (Y_1,H)}\|\rho_1(s)\|_{Y_1}+\|G(u^\rho(s))\|_{\mathcal{L}_2 (Y_2,H)}\|\rho_2(s)\|_{Y_2}\Big]ds\nonumber\\
				&\leq  C \int_{0}^{\tilde{T}}\Big[\|B\|_{\mathscr{L}(H,\mathcal{L}_2 (Y_1,H))}\|u^\rho(s)\|_H \|\rho_1(s)\|_{Y_1}+ (C_1+C_2\|u^\rho(s)\|_H )\|\rho_2(s)\|_{Y_2}\Big]ds\nonumber\\
				&\leq  C \|u^\rho\|_{\mathbb{L}^{\infty, p}_{2, r}(\tilde{T})} \int_{0}^{\tilde{T}} \big\{\|B\|_{\mathscr{L}(H,\mathcal{L}_2 (Y_1,H))} \|\rho_1(s)\|_{Y_1}+ C_2 \|\rho_2(s)\|_{Y_2} \big\}ds+ C C_1 \int_{0}^{\tilde{T}}\|\rho_2(s)\|_{Y_2}ds.
			\end{align}
			Combining all the estimates \eqref{boundF}-\eqref{boundI2b}, we deduce
			\begin{align}\label{localbound}
				&\|	\mathfrak{T}(u^\rho, u_0)(\cdot)\|_{\mathbb{L}^{\infty, p}_{2, r}(\tilde{T})}\nonumber\\
				&\leq  C\|u_0\|_{H} + (1+C)\beta \tilde{T} \|  u^\rho\|_{\mathbb{L}^{\infty, p}_{2, r}(\tilde{T})}+ C\tilde{T}^{ \frac{p - r}{p - 1}} \|u^\rho\|_{L^p(0, \tilde{T}; L^r(\mathbb{R}^d))}^{\alpha}+ 2C C_1 \int_{0}^{\tilde{T}}\|\rho_2(s)\|_{Y_2}ds\nonumber\\
				&\quad+2C \|u^\rho\|_{\mathbb{L}^{\infty, p}_{2, r}(\tilde{T})} \int_{0}^{\tilde{T}} \big\{\|B\|_{\mathscr{L}(H,\mathcal{L}_2 (Y_1,H))} \|\rho_1(s)\|_{Y_1}+ C_2 \|\rho_2(s)\|_{Y_2}\big\}ds. 
			\end{align}
			Next, we will show that the integral operator is a strict contraction to use the Banach fixed point theorem.
			From the assumptions (\Cref{AssF}) on the nonlinear term, we have,	
			\begin{align*}
				&\mathcal{N}(u)= |u|^{\alpha-1} u , \quad \mathcal{N}'(u)(h)=|u|^{\alpha-1}h+(\alpha-1)u|u|^{\alpha-3} \operatorname{Re}(u\bar{h}).
			\end{align*}
			Employing the fundamental theorem of calculas and Lagrange's mean value theorem, we infer
			\begin{align*}
				&\big|\mathcal{N}(u) - \mathcal{N}(v)\big| =\Big|\int_{0}^{1}\frac{d}{d\theta}\mathcal{N}(\theta u+(1-\theta)v)d\theta\Big|=\Big|\int_{0}^{1} \mathcal{N}'(\theta u+(1-\theta)v) (u-v) d\theta\Big|\\
				&\leq \int_{0}^{1} \big|\mathcal{N}'(\theta u+(1-\theta)v) (u-v)\big| d\theta \leq\alpha \int_{0}^{1} \big|\theta u+(1-\theta)v\big|^{\alpha-1} \big|u-v\big| d\theta\\
				&\leq\alpha \int_{0}^{1} (|u|+|v|)^{\alpha-1} |u-v| d\theta \leq\alpha (|u|+|v|)^{\alpha-1} |u-v|.
			\end{align*}
			This implies,
			\begin{align}
				&\big\|\mathcal{N}(u) - \mathcal{N}(v)\big\|_{L^{r'}} \leq \Big(\int_{\mathbb{R}^d}\big|\mathcal{N}(u) - \mathcal{N}(v)\big|^{r'} \Big)^{\frac{1}{r'}}\leq \Big(\int_{\mathbb{R}^d}(\alpha (|u|+|v|)^{\alpha-1} |u-v|)^{r'} \Big)^{\frac{1}{r'}}\nonumber\\
				&\leq \alpha \Big(\int_{\mathbb{R}^d}(|u|+|v|)^{(\alpha-1)r'} |u-v|^{r'} \Big)^{\frac{1}{r'}}\leq \alpha \| (|u|+|v|)^{\alpha-1} \|_{L^{\frac{r}{r-2}}} \| u - v \|_{L^r} \nonumber\\
				&= \alpha \|(|u|+|v|)^{\alpha-1} \|_{L^{\frac{r}{\alpha-1}}} \| u - v \|_{L^r}= \alpha \| (|u|+|v|) \|_{L^r}^{\alpha-1} \| u - v \|_{L^r} \nonumber\\
				&\leq \alpha \Big( \|u\|_{L^r} + \|v\|_{L^r} \Big)^{\alpha-1} \| u - v \|_{L^r}. \label{lipF}
			\end{align} 
			Using the estimate \eqref{lipF} together with the Strichartz estimates \eqref{Strichartzinf}, \eqref{Strichartzp} and H\"older's inequality (taking $Q=\frac{p-1}{p-r}$) such that $\frac{1}{Q}+\frac{\alpha-1}{(p-1)}+ \frac{1}{(p-1)}=1 $, we deduce
			\begin{align}\label{liptypeF}
				&\bigg\| 	- \int_{0}^{.} S(.-s) i( \mathcal{N}(u_1^\rho(s))- \mathcal{N}(u_2^\rho(s)))ds \bigg\|_{\mathbb{L}^{\infty, p}_{2, r}(\tilde{T})}\nonumber\\
				&\leq C \| \mathcal{N}(u_1) - \mathcal{N}(u_2) \|_{L^{p'}(0, \tilde{T}; L^{r'}(\mathbb{R}^d))}\nonumber \\
				&\leq C \Big( \int_0^{\tilde{T}} \| \mathcal{N}(u^\rho_1(s)) - \mathcal{N}(u^\rho_2(s)) \|^{p'}_{L^{r'}(\mathbb{R}^d)} ds \Big)^{\frac{1}{p'}} \nonumber\\
				&\leq C \Big[ \int_0^{\tilde{T}}\Big\{ \Big( \|u^\rho_1(s)\|_{L^r} + \|u^\rho_2(s)\|_{L^r} \Big)^{\alpha-1} \| u^\rho_1(s) - u^\rho_2(s) \|_{L^r}\Big\}^{p'} ds \Big]^{\frac{1}{p'}} \nonumber\\
				&\leq C \Big[ \int_0^{\tilde{T}} \Big( \|u^\rho_1(s)\|_{L^r} + \|u^\rho_2(s)\|_{L^r} \Big)^{(\alpha-1)p'} \| u^\rho_1(s) - u^\rho_2(s) \|_{L^r}^{p'} ds \Big]^{\frac{1}{p'}} \nonumber\\
				&\leq C \Big[	 \int_0^{\tilde{T}} \Big( \|u^\rho_1(s)\|_{L^r} + \|u^\rho_2(s)\|_{L^r} \Big)^{(\alpha-1)p'\frac{p-1}{\alpha-1}}	\Big]^{\frac{\alpha-1}{(p-1)p'}}	\Big[	 \int_0^{\tilde{T}} \Big(  \| u^\rho_1(s) - u^\rho_2(s) \|_{L^r} \Big)^{p'(p-1)}	\Big]^{\frac{1}{(p-1)p'}}	(\tilde{T})^{\frac{1}{Qp'}}\nonumber\\
				&\leq 	C	\Big[	 \int_0^{\tilde{T}} \Big( \|u^\rho_1(s)\|_{L^r} + \|u^\rho_1(s)\|_{L^r} \Big)^p	\Big]^{\frac{\alpha-1}{p}}  \| u^\rho_1 - u^\rho_2 \|_{L^{p}(0, \tilde{T}; L^{r}(\mathbb{R}^d))}(\tilde{T})^{\frac{1}{Qp'}}\nonumber\\
				&\leq C	\Big(\|u^\rho_1\|_{L^{p}(0, \tilde{T}; L^{r}(\mathbb{R}^d))}+ \|u^\rho_2\|_{L^{p}(0, \tilde{T}; L^{r}(\mathbb{R}^d))}\Big)^{\alpha-1} \| u^\rho_1 - u^\rho_2 \|_{L^{p}(0, \tilde{T}; L^{r}(\mathbb{R}^d))}(\tilde{T})^{\frac{1}{Qp'}}\nonumber\\
				& \leq C (\tilde{T})^{\frac{1}{Qp'}} (2M)^{\alpha -1} \| u^\rho_1 - u^\rho_2 \|_{L^{p}(0, \tilde{T}; L^{r}(\mathbb{R}^d))}\nonumber\\
				& \leq C (\tilde{T})^{\frac{p-r}{p}} (2M)^{\alpha -1} \| u^\rho_1 - u^\rho_2 \|_{\mathbb{L}^{\infty, p}_{2, r}(\tilde{T})}.
			\end{align} 
			Again, we use the Strichartz estimates \eqref{Strichartzinf}, \eqref{Strichartzp} to conclude,
			\begin{align}\label{liptypebeta}
				\bigg\| - \int_{0}^{.} S(.-s)\beta( u_1^\rho(s)- u_2^\rho(s))ds\bigg\|_{\mathbb{L}^{\infty, p}_{2, r}(\tilde{T})}
				&\leq (1+C)\beta \tilde{T} \|  u_1^\rho- u_2^\rho\|_{\mathbb{L}^{\infty, p}_{2, r}(\tilde{T})}.
			\end{align}     
			Combining (\ref{liptypeF}) and (\ref{liptypebeta}), we get
			\begin{align}\label{diffI1}
				&\| I_1(u_1^\rho)-I_1(u_2^\rho) \|_{\mathbb{L}^{\infty, p}_{2, r}(\tilde{T})} \nonumber\\
				&\leq  	\bigg\| - \int_{0}^{.} S(.-s)\beta\big( u_1^\rho(s)- u_2^\rho(s)\big)ds\bigg\|_{\mathbb{L}^{\infty, p}_{2, r}(\tilde{T})} +\bigg\| 	- \int_{0}^{.} S(.-s) i\big( \mathcal{N}(u_1^\rho(s))- \mathcal{N}(u_2^\rho(s))\big)ds \bigg\|_{\mathbb{L}^{\infty, p}_{2, r}(\tilde{T})}\nonumber\\
				&\leq  \big\{(1+C)\beta \tilde{T} + C (\tilde{T})^{\frac{p-r}{p}} (2M)^{\alpha -1}\big\}\big\| u^\rho_1 - u^\rho_2 \big\|_{\mathbb{L}^{\infty, p}_{2, r}(\tilde{T})}.
			\end{align}
			Let us now estimate for the terms involving $I_2$: 
			Using the linearity of $B$, Lipschitz continuity \eqref{LipG} of $G$, we estimate
			\begin{align}\label{diffI2a}
				&\sup_{t \in [0, \tilde{T}]} \| I_2(u^\rho_1)(t) - I_2(u^\rho_2)(t) \|_{H}\nonumber\\ 
				&= \sup_{t \in [0, \tilde{T}]} \left\|-i \int_0^{t} S(t - s) \Big[ \big\{B(u^\rho_1(s)) - B(u^\rho_2(s))\big\}\rho_1(s) +\big\{G(u^\rho_1(s)) - G(u^\rho_2(s))\big\}\rho_2(s)  \Big] \, ds \right\|_{H} \nonumber\\
				&\leq C\int_0^{\tilde{T}} \left\| B(u_1(s)) - B(u_2(s)) \right\|_{\mathcal{L}_2 (Y_1, H)} \| \rho_1(s) \|_{Y_1} \, ds + C\int_0^{\tilde{T}} \left\| G(u_1(s)) - G(u_2(s)) \right\|_{\mathcal{L}_2 (Y_2, H)} \| \rho_2(s) \|_{Y_2} \, ds \nonumber\\
				&\leq C \int_0^{\tilde{T}}\|B\|_{\mathscr{L}(H,\mathcal{L}_2 (Y_1,H))}\|u_1(s)) -u_2(s)\|_H \|\rho_1(s)\|_{Y_1}ds +  C\int_0^{\tilde{T}}L_G \|u_1(s)) -u_2(s)\|_H\|\rho_2(s)\|_{Y_2} ds \nonumber\\
				&\leq C \| u_1 - u_2 \|_{\mathbb{L}^{\infty, p}_{2, r}(\tilde{T})}\int_0^{\tilde{T}} \Big[\|B\|_{\mathscr{L}(H,\mathcal{L}_2 (Y_1,H))}\|\rho_1(s)\|_{Y_1}+L_G \|\rho_2(s)\|_{Y_2}\Big]ds.
			\end{align}
			%
			In a similar way, using the linearity of $B$, Lipschitz continuity \eqref{LipG} of $G$ and the Strichartz estimates \eqref{Strichartzinf}, \eqref{Strichartzp}, we infer
			\begin{align}\label{diffI2b}
				&\| I_2(u^\rho_1)(\cdot) - I_2(u^\rho_2)(\cdot) \|_{L^p(0,\tilde{T};L^r(\mathbb{R}^d))}\nonumber\\
				&= \bigg\|-i\int_{0}^{.} S(.-s)\Big[B\big(u^\rho_1(s)-u^\rho_2(s)\big)\rho_1(s)+ \big\{G(u_1^\rho(s))-G(u^\rho_2(s))\big\}\rho_2(s)  \Big]ds\bigg\|_{L^p(0,\tilde{T}L^r(\mathbb{R}^d))}\nonumber\\
				&\leq  C \Big\| \Big[B\big(u^\rho_1-u^\rho_2\big)\rho_1+\big\{G(u_1^\rho)-G(u^\rho_2)\big\}\rho_2  \Big]\Big\|_{L^1(0,\tilde{T};H)} \nonumber\\
				&\leq  C\int_{0}^{\tilde{T}}\Big\| \Big[B(u^\rho_1(s)-u^\rho_2(s))\rho_1(s)+\big\{G(u_1^\rho(s))-G(u^\rho_2(s))\big\}\rho_2(s)  \Big]\Big\|_{H}ds\nonumber\\
				&\leq  C \int_{0}^{\tilde{T}}\Big[\|B(u^\rho_1(s)-u^\rho_2(s))\|_{\mathcal{L}_2 (Y_1,H)}\|\rho_1(s)\|_{Y_1}+\|G(u_1^\rho(s))-G(u^\rho_2(s))\|_{\mathcal{L}_2 (Y_2,H)}\|\rho_2(s)\|_{Y_2}\Big]ds\nonumber\\
				&\leq  C \int_{0}^{\tilde{T}}\Big[\|B\|_{\mathscr{L}(H,\mathcal{L}_2 (Y_1,H))}\|u_1^\rho(s)-u^\rho_2(s)\|_H \|\rho_1(s)\|_{Y_1}+ L_G\|u_1^\rho(s)-u^\rho_2(s)\|_H \|\rho_2(s)\|_{Y_2}\Big]ds\nonumber\\
				&\leq  C \|u_1^\rho-u^\rho_2\|_{\mathbb{L}^{\infty, p}_{2, r}(\tilde{T})}\int_{0}^{\tilde{T}}\Big[\|B\|_{\mathscr{L}(H,\mathcal{L}_2 (Y_1,H))} \|\rho_1(s)\|_{Y_1}+ L_G \|\rho_2(s)\|_{Y_2}\Big]ds.
			\end{align}
			Combining all the estimates \eqref{diffI1}-\eqref{diffI2b}, we deduce 
			\begin{align}\label{lipIo}
				&\|	\mathfrak{T}(u_1^\rho, u_0)-	\mathfrak{T}(u_2^\rho, u_0)\|_{\mathbb{L}^{\infty, p}_{2, r}(\tilde{T})}\nonumber\\
				&\leq   \big\{(1+C)\beta \tilde{T} + C (\tilde{T})^{\frac{p-r}{p}} (2M)^{\alpha -1}\big\}\|u_1^\rho-u^\rho_2\|_{\mathbb{L}^{\infty, p}_{2, r}(\tilde{T})} \nonumber\\
				& \quad+ 2C \|u_1^\rho-u^\rho_2\|_{\mathbb{L}^{\infty, p}_{2, r}(\tilde{T})}\int_{0}^{\tilde{T}}\Big[\|B\|_{\mathscr{L}(H,\mathcal{L}_2 (Y_1,H))} \|\rho_1(s)\|_{Y_1}+ L_G \|\rho_2(s)\|_{Y_2}\Big]ds.
			\end{align}
			We choose, $M= 5C \|u_0\|_{H}$, where $C$ is the constant appeared in \eqref{localbound} and \eqref{lipIo}. Since, 
			\begin{align*}
				\int_{0}^{\tilde{T}} \Big\{\|B\|_{\mathscr{L}(H,\mathcal{L}_2 (Y_1,H))} \|\rho_1(s)\|_{Y_1}+ C_2 \|\rho_2(s)\|_{Y_2} \Big\}ds < \infty,
			\end{align*}
			for $\epsilon_1=\frac{1}{10C}$, there exists $\delta_1>0$ (by absolute continuity of Lebesgue integration, \Cref{abscts}) such that, 
			\begin{equation}
				\sup_{t-s \leq \delta_1}\int_{s}^{t} \Big\{\|B\|_{\mathscr{L}(H,\mathcal{L}_2 (Y_1,H))} \|\rho_1(s)\|_{Y_1}+ C_2 \|\rho_2(s)\|_{Y_2} \Big\}ds < \epsilon_1.
			\end{equation}
			Similarly, for $\epsilon_2=\frac{M}{10C C_1}$, there exists $\delta_2>0$ (by absolute continuity of Lebesgue integration, \Cref{abscts}) such that 
			\begin{equation}
				\sup_{t-s \leq \delta_2}\int_{s}^{t} \|\rho_2(s)\|_{Y_2}ds < \epsilon_2.
			\end{equation}
			We choose $\tilde{T}=T_0^1$, where
			\begin{align}\label{T01}
				T_0^1=\min \Big\{\frac{1}{5 \beta(1+C)}, \frac{1}{(5C)^{\frac{p-1}{p-r}}M^{\frac{(\alpha-1)(p-1)}{p-r}}}, \delta_1, \delta_2, T\Big\}.
			\end{align}
			This implies,
			\begin{align}\label{T1}
				&T_0^1 \leq  \frac{1}{(5C)^{\frac{p-1}{p-r}}M^{\frac{(\alpha-1)(p-1)}{p-r}}},\nonumber\\
				&C(T_0^1)^{ \frac{p - r}{p - 1}} M^{(\alpha-1)}\leq \frac{1}{5},\nonumber\\
				&C(T_0^1)^{ \frac{p - r}{p - 1}} M^{\alpha}\leq  \frac{M}{5},
			\end{align}
			and
			\begin{align}\label{T11}
				2C\int_{0}^{T_0^1} \Big\{\|B\|_{\mathscr{L}(H,\mathcal{L}_2 (Y_1,H))} \|\rho_1(s)\|_{Y_1}+ C_2 \|\rho_2(s)\|_{Y_2} \Big\}ds<\frac{1}{5}.
			\end{align}
			Therefore, from \eqref{localbound} using \eqref{T1} and \eqref{T11}, for all $u \in 	V_{M,T_0^1}$,
			\begin{align}\label{T1m}
				\|	\mathfrak{T}(u^\rho, u_0)(\cdot)\|_{\mathbb{L}^{\infty, p}_{2, r}(T_0^1)}
				&\leq  M.
			\end{align}
			Since,
			\begin{align*}
				\int_{0}^{\tilde{T}}\Big[\|B\|_{\mathscr{L}(H,\mathcal{L}_2 (Y_1,H))} \|\rho_1(s)\|_{Y_1}+ L_G \|\rho_2(s)\|_{Y_2}\Big]ds < \infty,
			\end{align*}
			for $\epsilon_3=\frac{1}{12C}$, there exists $\delta_3>0$ (by absolute continuity of Lebesgue integration, \Cref{abscts}) such that 
			\begin{equation}
				\sup_{t-s \leq \delta_3}\int_{s}^{t}\Big[\|B\|_{\mathscr{L}(H,\mathcal{L}_2 (Y_1,H))} \|\rho_1(s)\|_{Y_1}+ L_G \|\rho_2(s)\|_{Y_2}\Big]ds < \epsilon_3.
			\end{equation}	
			Now, we choose $\tilde{T}=T_0^2$, with $Q=\frac{p-1}{p-r}$, where
			\begin{align*}
				T_0^2=\min \Big\{\frac{1}{6 \beta(1+C)}, \Big(\frac{1}{6C(2M)^{\alpha-1}}\Big)^{Qp'}, \delta_3,T \Big\}.
			\end{align*}
			This implies,
			\begin{align}\label{T2}
				2C \int_{0}^{T_0^2}\Big[\|B\|_{\mathscr{L}(H,\mathcal{L}_2 (Y_1,H))} \|\rho_1(s)\|_{Y_1}+ L_G \|\rho_2(s)\|_{Y_2}\Big]ds< \frac{1}{6},
			\end{align}
			and
			\begin{align}\label{T21}
				C (T_0^2)^{\frac{p-r}{p}} (2M)^{\alpha -1} \leq \frac{1}{6}.
			\end{align}
			For $1< \alpha< \frac{4}{d}+1,$ we have $\frac{p-r}{p}\geq0.$ Hence, for any $u_1, u_2 \in V_{M,T_0^2}$, from \eqref{lipIo} using \eqref{T2} and \eqref{T21}, we conclude, 
			\begin{align}\label{T2m}
				\|	\mathfrak{T}(u_1^\rho, u_0)-	\mathfrak{T}(u_2^\rho, u_0)\|_{\mathbb{L}^{\infty, p}_{2, r}(T_0^2)}	\leq \frac{1}{2} \|u_1^\rho-u^\rho_2\|_{\mathbb{L}^{\infty, p}_{2, r}(T_0^2)}.
			\end{align}
			Now, choose $T_0= \{\min{T_0^1,T_0^2}\}$. From \eqref{T1m} and \eqref{T2m}, we have,
			for all $u \in V_{M,T_0}$,
			\begin{align}\label{bm}
				\|	\mathfrak{T}(u^\rho, u_0)(\cdot)\|_{\mathbb{L}^{\infty, p}_{2, r}(T_0)}
				&\leq  M,
			\end{align}
			and for any $u_1, u_2 \in V_{M,T_0}$,
			\begin{align}\label{lipm}
				\|	\mathfrak{T}(u_1^\rho, u_0)(\cdot)-	\mathfrak{T}(u_2^\rho, u_0)(\cdot)\|_{\mathbb{L}^{\infty, p}_{2, r}(T_0)}	\leq \frac{1}{2} \|u_1^\rho-u^\rho_2\|_{\mathbb{L}^{\infty, p}_{2, r}(T_0)}.
			\end{align}
			Therefore $\mathfrak{T}(\cdot, u_0)$ is a $\frac{1}{2}$-contraction on $V_{M,T_0}$, where $T_0$ depends on $\beta, \alpha, M, \|u_0\|_{H}$.
			Thus by the Banach fixed point theorem, $\mathfrak{T}(\cdot, u_0)$ has a unique fixed point $u^*\in V_{M,T_0}$. It is easy to check that $u^* \in C([0,T_0];H)$, and it is the unique solution to (\ref{Skeleton}) on $[0,T_0]$.\\
			\indent In the next step, we show that the global mild solution of \eqref{Skeleton} is \emph{unique if it exists}. We use the uniqueness of the local mild solution of \eqref{Skeleton} as the main tool. In the next step, we show the existence of global solution.
			\item\textbf{(Uniqueness of the global mild solution):}\label{stepbskl}
			Let $u_1, u_2 \in \mathbb{L}^{\infty, p}_{2, r}(T) \cap C([0,T];H)$ be any two solutions satisfying the skeleton equation \eqref{Skeleton} with the same initial data $u_0$. We define, 
			\begin{align*}
				T^*= \sup \big\{t\in [0,T]: u_1(s)=u_2(s) \text{ for all }  s\in[0,T^*]\big\}.
			\end{align*}
			If $T^*= T$, then the uniqueness is trivial. If not, then we have $\tilde{u}_1(\cdot)=u_1(T^*+ \cdot)$ and $\tilde{u}_2(\cdot)=u_2(T^*+ \cdot)$ are two solutions satisfying the mild form (\ref{mildformsk}) with initial condition $u_0$ replaced by $u_1(T^*)=u_2(T^*)$ on the interval $[0, T-T^*]$. Using similar arguments as above to get \eqref{localbound} and \eqref{lipIo}, there exists a small time $T_0^*$ such that $\tilde{u_1}(s)=\tilde{u_2}(s), s\in [0,T_0^*]$, i.e., $u_1(s)=u_2(s) \text{ for all }  s\in[0,T^*+T_0^*]$, which is a contradiction to the definition of $T^*.$
			This guarantees the uniqueness of the solution in $[0,T].$ 
			\item  \textbf{(Global existence of the solution):}\label{stepcskl} Here we prove the solution is global. To do so, we first introduce the Yosida approximation operator and its properties.
			We fix $ \mu > 0$ and define the Yosida approximation operator $ J_\mu : H^{-1}(\mathbb{R}^d) \to H^1(\mathbb{R}^d) $ by 
			\begin{align*}
				J_\mu = \mu (\mu I - \Delta)^{-1}.
			\end{align*}
			In other words, for every $f\in H^{-1}(\mathbb{R}^d)$, $u_{\mu}=J_\mu f \in H_0^1(\mathbb{R}^d)$ is the unique solution of 
			\begin{align*}
				u_{\mu} -\frac{1}{\mu} \Delta u_{\mu}=f.
			\end{align*}	
			We also define, for the nonlinear term,
			\begin{align*}
				\mathcal{N}_\mu(u):= J_\mu \mathcal{N}  (J_\mu u), \quad G_\mu(u):= J_\mu G(J_\mu u),\quad u\in H.
			\end{align*}
			Now, we define the Yosida approximated form of the equation \eqref{Skeleton},	
			\begin{equation}\label{approxg}
				\left\{
				\begin{aligned}
					du_\mu^\rho(t)&=-\big[i Au_\mu^\rho(t) +i \mathcal{N}_\mu(u_\mu^\rho(t))+ \beta u_\mu^\rho(t)\big]dt - i B (u_\mu^\rho(t))\rho_1(t) dt- i   G_\mu(u_\mu^\rho(t))\rho_2(t)dt , \quad t>0,	\\
					u_\mu^\rho(0)&=J_\mu u_0.
				\end{aligned}
				\right.
			\end{equation}
			\vskip 0.1cm\noindent
			\textbf{Uniform energy estimates (in time):}
			We take the $H^{-1}-H^1$ duality product of equation (\ref{approxg}) by $u_\mu^\rho$ and use Lions-Magenes lemma (\Cref{lionsmgnslem}) to deduce, 
			\begin{align*}
				\frac{1}{2}\frac{d}{dt}\|u_\mu^\rho(t)\|_H^2&= \operatorname{Re} {_{H^1}}\langle u_\mu^\rho(t),\partial_t u_\mu^\rho(t)\rangle{_{H^{-1}}}\\
				&= \operatorname{Re} _{H^1}\langle u_\mu^\rho(t),-i A u_\mu^\rho(t) -i  \mathcal{N}_\mu(u_\mu^\rho(t))-\beta u_\mu^\rho(t)- i B (u_\mu^\rho(t))\rho_1(t)- i   G_\mu(u_\mu^\rho(t))\rho_2(t) \rangle{_{H^{-1}}}.
			\end{align*}
			We can easily check,
			\begin{align}\label{energyA}
				\operatorname{Re} _{H^1}\langle u_\mu^\rho,-i A u_\mu^\rho \rangle{_{H^{-1}}}&= i	\operatorname{Re} \|\grad u_\mu^\rho\|^2_{H}=0,
			\end{align}
			and \begin{align*}
				\operatorname{Re} _{H^1}\langle v,-i  \mathcal{N}(v)\rangle{_{H^{-1}}}&=\operatorname{Re} _{H^1}\langle v,-i  |v|^{\alpha - 1} {v} \rangle{_{H^{-1}}} =-i \operatorname{Re} (|v|^{\alpha - 1} \|v\|^2_{H})=0.
			\end{align*}
			This gives,
			\begin{align}\label{energyF}
				\operatorname{Re} _{H^1}\langle u_\mu^\rho,-i  \mathcal{N}_\mu(u_\mu^\rho)\rangle{_{H^{-1}}}=\operatorname{Re} _{H^1}\langle u_\mu^\rho,-i  J_\mu \mathcal{N}  (J_\mu u_\mu^\rho) \rangle{_{H^{-1}}} &=\operatorname{Re} _{H^1}\langle J_\mu u_\mu^\rho,-i \mathcal{N}  (J_\mu u_\mu^\rho) \rangle{_{H^{-1}}}=0.
			\end{align}				
			For the term related to $B$, we get,
			\begin{align*}
				\operatorname{Re} _{H^1}\langle u_\mu^\rho,- i  Bu_\mu^\rho \rho_1 \rangle{_{H^{-1}}}=\operatorname{Re} \big ( u_\mu^\rho,- i  Bu_\mu^\rho \rho_1 \big )&=	-i \operatorname{Re}\big ( u_\mu^\rho, Bu_\mu^\rho\textstyle \sum_{m=1}^\infty c_m e^1_m \big )\\
				&=	-i \textstyle \sum_{m=1}^\infty c_m \operatorname{Re} \big ( u_\mu^\rho, B_m(u_\mu^\rho) \big ).
			\end{align*}
			Since, $B_m$ is self adjoint operator, we have,
			\begin{align}\label{energyB}
				\big (u_\mu^\rho, B_m(u_\mu^\rho) \big )=\overline{\big ( u_\mu^\rho, B_m(u_\mu^\rho) \big )}.
			\end{align}
			This implies,
			\begin{align*}
				\operatorname{Re} _{H^1}\langle u_\mu^\rho,- i  Bu_\mu^\rho \rho_1 \rangle{_{H^{-1}}}=0.
			\end{align*}
				For the term related to $G$, using \eqref{perticularg} and the decomposition of $\rho_2$ in $Y_2$ as $\rho_2(t)=\sum_{m\in \mathbb{N}}l_m(t) e^2_m$, we estimate
				\begin{align*}
					&\operatorname{Re} _{H^1}\langle u_\mu^\rho(t),- i J_\mu G(J_\mu u_\mu^\rho(t))\rho_2(t) \rangle{_{H^{-1}}}=\operatorname{Re} \big (  J_\mu u_\mu^\rho(t),- i G(J_\mu u_\mu^\rho(t))\rho_2 (t) \big )\nonumber\\
					&=  \operatorname{Re} \big ( J_\mu u_\mu^\rho(t),- i G(J_\mu u_\mu^\rho(t)) \rho_2(t)  \big ) =\operatorname{Re} \big ( J_\mu u_\mu^\rho(t),- i G(J_\mu u_\mu^\rho(t)) \sum_{m\in \mathbb{N}}l_m(t) e^2_m  \big ) \nonumber\\
					&  =\sum_{m\in \mathbb{N}}l_m (t)\operatorname{Re} \big (  J_\mu u_\mu^\rho(t),- i G(J_\mu u_\mu^\rho(t)) e^2_m  \big )=\sum_{m\in \mathbb{N}}l_m (t)\operatorname{Re} \big (  J_\mu u_\mu^\rho(t),- i \tilde{G}(\|J_\mu u_\mu^\rho(t)\|^2_{H})J_\mu u_\mu^\rho(t) e^2_m  \big )\\
					&=i\sum_{m\in \mathbb{N}}  l_m (t) \tilde{G}(\|J_\mu u_\mu^\rho(t)\|^2_{H})\operatorname{Im} \big (  J_\mu u_\mu^\rho(t), J_\mu u_\mu^\rho(t) e^2_m  \big )=0,
				\end{align*}
				Therefore 
				\begin{align}\label{energyG}
					\operatorname{Re} _{H^1}\langle u_\mu^\rho,- i G_\mu(u_\mu^\rho)\rho_2 \rangle{_{H^{-1}}}=0.
				\end{align}
				Combining \eqref{energyA}-\eqref{energyG}, for a.e. $t \in [0,T]$, we have
				\begin{align*}
					\frac{1}{2}\frac{d}{dt}\|u_\mu^\rho(t)\|_H^2=& -\beta \|u_\mu^\rho(t)\|^2_H -	\operatorname{Re} _{H^1}\langle u_\mu^\rho(t),- i G_\mu(u_\mu^\rho(t))\rho_2(t) \rangle{_{H^{-1}}}\leq0.
				\end{align*} 
				Integrating over time from $ 0$ to $t$, we infer
				\begin{align*}
					\|u_\mu^\rho(t)\|^2_H \leq \|u_0\|^2_H.
				\end{align*}
				Taking supremum over time, we get
				\begin{align*}
					\sup_{t \in [0, T]}\|u_\mu^\rho(t)\|^2_H\leq \|u_0\|^2_H .
				\end{align*}
				This implies, 
				\begin{equation}\label{energyl2}
					\sup_{\mu>0} \sup_{t \in [0, T]}\|u_\mu^\rho(t)\|^2_H\leq  \|u_0\|^2_H<\infty.
				\end{equation}
				Since the $L^2$-energy is decaying, we can extend the local solution to global in time.
				For $u\in V_{M,\tilde{T}}$, we have the results holds with $\mathfrak{T}$ replaced by $\mathfrak{T}_\mu$ such that,
				\begin{align*}
					\mathfrak{T}_\mu(u_\mu^\rho, u_0)(t)&=S(t)J_\mu u_0	- \int_{0}^{t} S(t-s)\big[i F_\mu(u_\mu^\rho(s))+\beta u_\mu^\rho(s)\big]ds \\
					& \quad- i\int_{0}^{t} S(t-s)\big[B(u_\mu^\rho(s))\rho_1(s)+ G_\mu(u_\mu^\rho(t))\rho_2(t)\big]ds, \quad t\geq0. 
				\end{align*}
				We can choose $T_0$ (similarly as before we choose to get \eqref{bm}, \eqref{lipm}) for which we have the existence of unique solution $u_\mu^1 \in V_{M,T_0}$ to equation (\ref{approxg}) and
				\begin{equation}\label{energyl2mu1}
					\sup_{t \in [0, T_0]}\|u_\mu^1(t)\|^2_H\leq  \|u_0\|^2_H<\infty.
				\end{equation} 		 
				In particular, 
				\begin{equation*}
					\|u_\mu^1(T_0)\|^2_H\leq  \|u_0\|^2_H.
				\end{equation*} 		 
				Now, we consider the equation,
				\begin{align}\label{mu2}
					\left\{
					\begin{aligned}
						du_\mu^2(t)&=-\big[i Au_\mu^2(t) +i \mathcal{N}_\mu(u_\mu^2(t))+ \beta u_\mu^2(t)\big]dt - i B( u_\mu^2(t))\rho_1(t) dt- i   G_\mu(u_\mu^2(t))\rho_2(t)dt , \quad t>0,	\\
						u_\mu^2(0)&=u_\mu^1(T_0).
					\end{aligned}
					\right.
				\end{align}
				Since the $L^2$-energy is decaying, following the similar approach as before to get $u_\mu^1$, there exists a unique solution $u_\mu^2 \in V_{M,T_0} \cap C([0,T_0];H) $ to \eqref{mu2} and
				\begin{equation}\label{energyl2mu2}
					\sup_{t \in [0, T_0]}\|u_\mu^2(t)\|^2_H\leq  \|u_\mu^1(T_0)\|^2_H\leq  \|u_0\|^2_H<\infty.
				\end{equation}
 Infact, we get the local existence of $u_\mu^2$ upto a time $\tilde{T_0}$, the lifespan of local existence of \eqref{mu2}. We assert that $\tilde{T_0} \geq T_0$. 
From \eqref{T1} and \eqref{T21}, one can easily observe that  $\tilde{T_0}$ depends only on $\tilde{M}=5C\|u_\mu^2(0) \|_{H}$. Also, as the $L^2$-energy is decaying, that is, 
$$\|u_\mu^2(0) \|_{H}  =  \|u_\mu^1(T_0) \|_{H} \leq  \|u_0\|_H.$$
This essentially yields $\tilde{M}\leq M$. Furthermore, $T_0$ increases as $M$ decreases, 
which proves our assertion. 
 For the ease of calculation, we restrict our attention to the existence of the solution on the interval $[0,T_0]$ in each iterative step.  We refer the reader \cite[Theorem 4.6.1 (iv)]{MR2002047} for the $L^2$-conservative case.
 
				\noindent Repeating the same argument, we can see that, there exists a unique solution $u_\mu^k \in V_{M,T_0} \cap C([0,T_0];H) $ to the following equation,
				\begin{align}\label{muk}
					\left\{
					\begin{aligned}
						du_\mu^k(t)&=-\big[i Au_\mu^k(t) +i \mathcal{N}_\mu(u_\mu^k(t))+ \beta u_\mu^k(t)\big]dt - i B( u_\mu^k(t))\rho_1(t) dt- i   G_\mu(u_\mu^k(t))\rho_2(t)dt , \quad t>0,	\\
						u_\mu^k(0)&=u_\mu^{k-1}(T_0).
					\end{aligned}
					\right.
				\end{align}
				We define,
				\begin{align}\label{glueumu}
					u^\rho_\mu(t)=u_\mu^k(t-(k-1)T_0), \quad t\in [(k-1)T_0, kT_0].
				\end{align}
				Since at each iteration step, we obtain $u_\mu^k \in V_{M,T_0}$, from the definition of the set in \eqref{localset}, it follows that $\|u_\mu^k \|_{L^p(0,T_0;L^r(\mathbb{R}^d))} \leq M$. Therefore, we have \(	u^\rho_\mu=	(u^\rho_\mu(t))_{t \in [0,T]} \in \mathbb{L}^{\infty, p}_{2, r}(T) \cap C([0,T];H)  \) as the unique solution of $\mu$-approximated equation (\ref{approxg}) on $[0,T]$, and 
				\begin{align}\label{boundforumuglobal}
					\sup_{\rho \in \mathbb{D}^N}\sup_{\mu>0} 	\|u_\mu^\rho\|_{\mathbb{L}^{\infty, p}_{2, r}(T)} \leq \Big(\Big[\frac{T}{T_0}\Big]+1\Big)M.
			\end{align}
			We already have the existence of unique solution $u^*\in C([0,T_0];H)$ of (\ref{Skeleton}). Using the similar arguments (gluing technique as in \eqref{glueumu}), we define a function $u^\rho=(u^\rho(t))_{t\in [0,T]}$ which is the unique solution to \eqref{Skeleton} on $[0,T]$.
			Our aim to show, 
			\begin{align}\label{mutoor}
				\limsup_{\mu \to\infty} \|u_\mu^\rho-u^\rho\|_{\mathbb{L}^{\infty, p}_{2, r}(T)}=0.
			\end{align}
			To prove this, we take $\tilde{T} \in [0,T]$ and apply the Strichartz estimates (\Cref{Strichartz}) to get,
			\begin{align}\label{mua1}
				&\bigg\| 	\int_{0}^{\cdot} S(\cdot-s) ( \mathcal{N}(u^\rho(s))- \mathcal{N}_\mu(u_\mu^\rho(s)))ds \bigg\|_{\mathbb{L}^{\infty, p}_{2, r}(\tilde{T})}\nonumber\\
				&\leq C \| \mathcal{N}(u^\rho)- J_\mu \mathcal{N}(J_\mu u_\mu^\rho)\|_{L^{p'}(0, \tilde{T}; L^{r'}(\mathbb{R}^d))}\nonumber \\
				&\leq C \| J_\mu \mathcal{N}(J_\mu u_\mu^\rho)-J_\mu \mathcal{N}( u^\rho)\|_{L^{p'}(0, \tilde{T}; L^{r'}(\mathbb{R}^d))}+  C \| J_\mu \mathcal{N}( u^\rho)-\mathcal{N}(u^\rho)\|_{L^{p'}(0, \tilde{T}; L^{r'}(\mathbb{R}^d))}.
			\end{align}
			Using the properties of $J_\mu$ mentioned in \Cref{yosidaoperator}, and previous estimate (\ref{lipF}), we deduce
			\begin{align}\label{mua2}
				&\| J_\mu \mathcal{N}(J_\mu u_\mu^\rho)-J_\mu \mathcal{N}( u^\rho)\|_{L^{p'}(0, \tilde{T}; L^{r'}(\mathbb{R}^d))}\nonumber\\
				&\leq \| \mathcal{N}(J_\mu u_\mu^\rho)-\mathcal{N}( u^\rho)\|_{L^{p'}(0, \tilde{T}; L^{r'}(\mathbb{R}^d))}\nonumber\\
				&\leq C	\Big(\|J_\mu u_\mu^\rho\|_{L^{p}(0, \tilde{T}; L^{r}(\mathbb{R}^d))}+ \|u^\rho\|_{L^{p}(0, \tilde{T}; L^{r}(\mathbb{R}^d))}\Big)^{\alpha-1} \| J_\mu u_\mu^\rho - u^\rho \|_{L^{p}(0, \tilde{T}; L^{r}(\mathbb{R}^d))}\textstyle(\tilde{T})^{\frac{p-r}{p}}\nonumber\\
				&\leq C	 \textstyle(\tilde{T})^{\frac{p-r}{p}} \Big(\|u_\mu^\rho\|_{L^{p}(0, \tilde{T}; L^{r}(\mathbb{R}^d))}+ \|u^\rho\|_{L^{p}(0, \tilde{T}; L^{r}(\mathbb{R}^d))}\Big)^{\alpha-1} \big\{\|J_\mu u_\mu^\rho -J_\mu u^\rho \|_{L^{p}(0, \tilde{T}; L^{r}(\mathbb{R}^d))}\nonumber\\
				&\qquad+ \|J_\mu u^\rho - u^\rho \|_{L^{p}(0, \tilde{T}; L^{r}(\mathbb{R}^d))}\big\}\nonumber\\
				&\leq C	 \textstyle(\tilde{T})^{\frac{p-r}{p}} \Big(\|u_\mu^\rho\|_{L^{p}(0, \tilde{T}; L^{r}(\mathbb{R}^d))}+ \|u^\rho\|_{L^{p}(0, \tilde{T}; L^{r}(\mathbb{R}^d))}\Big)^{\alpha-1} \big\{\| u_\mu^\rho -u^\rho \|_{L^{p}(0, \tilde{T}; L^{r}(\mathbb{R}^d))}\nonumber\\
				&\qquad+ \|J_\mu u^\rho - u^\rho \|_{L^{p}(0, \tilde{T}; L^{r}(\mathbb{R}^d))}\big\}.
			\end{align}	
			We apply the Strichartz estimates (\Cref{Strichartz}) for $(p,r)=(\infty,2)$ to estimate,
			\begin{align}\label{mua3}
				&\bigg\| 	\int_{0}^{\cdot} S(\cdot-s) \big(  B( u_\mu^\rho(s))\rho_1(s)-  B( u^\rho(s))\rho_1(s)\big)ds \bigg\|_{\mathbb{L}^{\infty, p}_{2, r}(\tilde{T})}\nonumber\\
				&\leq2 C \| B( u_\mu^\rho)\rho_1-  B( u^\rho)\rho_1\|_{L^{p'}(0, \tilde{T}; L^{r'}(\mathbb{R}^d))}\nonumber \\
				&\leq2 C \| B\big( u_\mu^\rho-u^\rho\big)\rho_1\|_{L^{1}(0, \tilde{T}; L^{2}(\mathbb{R}^d))}\nonumber \\
				&	\leq 2C \int_{0}^{\tilde{T}}\big\{ \|B\|_{\mathscr{L}(H,\mathcal{L}_2 (Y_1,H))}\|u_\mu^\rho(s)-u^\rho(s)\|_H \|\rho_1(s)\|_{Y_1} \big\}ds\nonumber\\
				& 	\leq 2C \int_{0}^{\tilde{T}}\big\{ \|B\|_{\mathscr{L}(H,\mathcal{L}_2 (Y_1,H))}\|u_\mu^\rho-u^\rho\|_{\mathbb{L}^{\infty, p}_{2, r}(s)} \|\rho_1(s)\|_{Y_1} \big\}ds.
			\end{align}
			Again, we use the Strichartz estimates (\Cref{Strichartz}) for $(p,r)=(\infty,2)$ and derive
			\begin{align}\label{mua4}
				&\bigg\| 	\int_{0}^{\cdot} S(\cdot-s) \big(  G_\mu(u_\mu^\rho(s))\rho_2(s)-  G( u^\rho(s))\rho_2(s)\big)ds \bigg\|_{\mathbb{L}^{\infty, p}_{2, r}(\tilde{T})}\nonumber\\
				&\leq  \bigg\| 	\int_{0}^{\cdot} S(\cdot-s) \big\{ J_\mu G(J_\mu u_\mu^\rho(s))-  G( u^\rho(s))\big\} \rho_2(s)ds \bigg\|_{\mathbb{L}^{\infty, p}_{2, r}(\tilde{T})}\nonumber\\
				&\leq  C \big\| \big\{ J_\mu G(J_\mu u_\mu^\rho)-  G( u^\rho)\big\} \rho_2\big\|_{L^{p'}(0, \tilde{T}; L^{r'}(\mathbb{R}^d))}\nonumber\\
				&\leq  C \int_{0}^{\tilde{T}}\| J_\mu G(J_\mu u_\mu^\rho(s))-  G( u^\rho(s))\|_{\mathcal{L}_2 (Y_2,H)}  \|\rho_2(s)\|_{Y_2}ds.
			\end{align}
			Now, using the properties of $J_\mu$ from \Cref{yosidaoperator}, Lipschitz continuity of G \eqref{LipG}, we obtain
			\begin{align}\label{mua5}
				&\| J_\mu G(J_\mu u_\mu^\rho(s))-  G( u^\rho(s))\|_{\mathcal{L}_2 (Y_2,H)}\nonumber \\
				&\leq \| J_\mu G(J_\mu u_\mu^\rho(s))-  J_\mu G( u^\rho(s))\|_{\mathcal{L}_2 (Y_2,H)} + \| J_\mu G( u^\rho(s))-G( u^\rho(s))\|_{\mathcal{L}_2 (Y_2,H)}\nonumber \\
				&\leq  C\| G(J_\mu u_\mu^\rho(s))- G( u^\rho(s))\|_{\mathcal{L}_2 (Y_2,H)}+ \| (J_\mu - I)G( u^\rho(s))\|_{\mathcal{L}_2 (Y_2,H)}\nonumber \\
				&\leq  C L_G\| J_\mu u_\mu^\rho(s)-  u^\rho(s)\|_H+ \| (J_\mu - I)G( u^\rho(s))\|_{\mathcal{L}_2 (Y_2,H)}\nonumber \\
				&\leq  C L_G\big\{\| J_\mu u_\mu^\rho(s)- J_\mu u^\rho(s)\|_H+\|  J_\mu u^\rho(s)-u^\rho(s)\|_H\big\}+ \| (J_\mu - I)G( u^\rho(s))\|_{\mathcal{L}_2 (Y_2,H)}\nonumber \\
				&\leq  C L_G\big\{\| u_\mu^\rho(s)-  u^\rho(s)\|_H+\| ( J_\mu -I)u^\rho(s)\|_H\big\}+ \| (J_\mu - I)G( u^\rho(s))\|_{\mathcal{L}_2 (Y_2,H)}.
			\end{align}
			Therefore, the estimates \eqref{mua1}-\eqref{mua5} together imply
			\begin{align}\label{mulip}
				&\|u_\mu^\rho-u^\rho\|_{\mathbb{L}^{\infty, p}_{2, r}(\tilde{T})}\nonumber\\
				&\leq  \|S(\cdot)J_\mu u_0-S(\cdot)u_0\|_{\mathbb{L}^{\infty, p}_{2, r}(\tilde{T})} +\bigg\| \int_{0}^\cdot S(\cdot-s) \Big[ \mathcal{N}(u^\rho(s))- \mathcal{N}_\mu(u_\mu^\rho(s))\Big]ds \bigg\|_{\mathbb{L}^{\infty, p}_{2, r}(\tilde{T})}\nonumber\\
				&\quad+ \bigg\| \int_{0}^\cdot S(\cdot-s) \Big[  \beta u^\rho(s)- \beta u_\mu^\rho(s)\Big]ds \bigg\|_{\mathbb{L}^{\infty, p}_{2, r}(\tilde{T})}+ \bigg\| 	\int_{0}^\cdot S(\cdot-s) \Big[ B( u_\mu^\rho(s))\rho_1(s)-  B( u^\rho(s))\rho_1(s)\Big]ds \bigg\|_{\mathbb{L}^{\infty, p}_{2, r}(\tilde{T})}\nonumber\\
				&\quad+ \bigg\| 	\int_{0}^\cdot S(\cdot-s) \Big[ G_\mu( u_\mu^\rho(s))\rho_2(s)-  G( u^\rho(s))\rho_2(s)\Big]ds \bigg\|_{\mathbb{L}^{\infty, p}_{2, r}(\tilde{T})}\nonumber\\
				&\leq C \|J_\mu u_0-u_0\|_H + C	 \textstyle(\tilde{T})^{\frac{p-r}{p}}\Big(\|u_\mu^\rho\|_{L^{p}(0, \tilde{T}; L^{r}(\mathbb{R}^d))}+ \|u^\rho\|_{L^{p}(0, \tilde{T}; L^{r}(\mathbb{R}^d))}\Big)^{\alpha-1}\nonumber\\
				&\qquad  \left\{\| u_\mu^\rho -u^\rho \|_{L^{p}(0, \tilde{T}; L^{r}(\mathbb{R}^d))}+\|J_\mu u^\rho - u^\rho \|_{L^{p}(0, \tilde{T}; L^{r}(\mathbb{R}^d))} \right\}+C \| J_\mu \mathcal{N}( u^\rho)-\mathcal{N}(u^\rho)\|_{L^{p'}(0, \tilde{T}; L^{r'}(\mathbb{R}^d))}\nonumber\\
				&\quad+\beta\int_{0}^{\tilde{T}} \|u_\mu^\rho-u^\rho\|_{\mathbb{L}^{\infty, p}_{2, r}(s)}ds+  2C \int_{0}^{\tilde{T}}\big\{ \|B\|_{\mathscr{L}(H,\mathcal{L}_2 (Y_1,H))}\|u_\mu^\rho-u^\rho\|_{\mathbb{L}^{\infty, p}_{2, r}(s)} \|\rho_1(s)\|_{Y_1} \big\}ds\nonumber\\
				&\quad+C \int_{0}^{\tilde{T}}\Big[C L_G\{\| u_\mu^\rho(s)-  u^\rho(s)\|_{H}+\| ( J_\mu -I)u^\rho(s)\|_H\}+ \| (J_\mu - I)G( u^\rho(s))\|_{\mathcal{L}_2 (Y_2,H)}\Big]\|\rho_2(s)\|_{Y_2}ds\nonumber\\
				&\leq C \|J_\mu u_0-u_0\|_H + C	 \textstyle(\tilde{T})^{\frac{p-r}{p}}\Big(\displaystyle\sup_\mu\|u_\mu^\rho\|_{L^{p}(0, \tilde{T}; L^{r}(\mathbb{R}^d))}+ \|u^\rho\|_{L^{p}(0, \tilde{T}; L^{r}(\mathbb{R}^d))}\Big)^{\alpha-1}  \nonumber\\
				&\qquad\left\{\| u_\mu^\rho -u^\rho \|_{L^{p}(0, \tilde{T}; L^{r}(\mathbb{R}^d))}+\|J_\mu u^\rho - u^\rho \|_{L^{p}(0, \tilde{T}; L^{r}(\mathbb{R}^d))} \right\}+C \| J_\mu \mathcal{N}( u^\rho)-\mathcal{N}(u^\rho)\|_{L^{p'}(0, \tilde{T}; L^{r'}(\mathbb{R}^d))}\nonumber\\
				&\quad+ \beta\int_{0}^{\tilde{T}} \|u_\mu^\rho-u^\rho\|_{\mathbb{L}^{\infty, p}_{2, r}(s)}ds+  2C \int_{0}^{\tilde{T}}\big\{ \|B\|_{\mathscr{L}(H,\mathcal{L}_2 (Y_1,H))}\|u_\mu^\rho-u^\rho\|_{\mathbb{L}^{\infty, p}_{2, r}(s)} \|\rho_1(s)\|_{Y_1} \big\}ds\nonumber\\
				&\quad+C \int_{0}^{\tilde{T}}\Big[C L_G\{\| ( J_\mu -I)u^\rho(s)\|_H\}+ \| (J_\mu - I)G( u^\rho(s))\|_{\mathcal{L}_2 (Y_2,H)}\Big]\|\rho_2(s)\|_{Y_2}ds\nonumber\\
				&\quad+C \int_{0}^{\tilde{T}}C L_G\| u_\mu^\rho-  u^\rho\|_{\mathbb{L}^{\infty, p}_{2, r}(s)}\|\rho_2(s)\|_{Y_2}ds.
			\end{align}
			Taking $\limsup \mu\to \infty$ in \eqref{mulip}, using Fatou's lemma and the convergence of $J_\mu$ to $I$ (\Cref{yosidaoperator}), we conclude
			\begin{align}\label{mulip1}
				&\limsup_{\mu \to\infty} \|u_\mu^\rho-u^\rho\|_{\mathbb{L}^{\infty, p}_{2, r}(\tilde{T})}\nonumber\\
				&\leq  C	 \textstyle(\tilde{T})^{\frac{p-r}{p}}\Big(\displaystyle\sup_\mu\|u_\mu^\rho\|_{L^{p}(0, \tilde{T}; L^{r}(\mathbb{R}^d))}+ \|u^\rho\|_{L^{p}(0, \tilde{T}; L^{r}(\mathbb{R}^d))}\Big)^{\alpha-1} \limsup_{\mu \to\infty} \|u_\mu^\rho-u^\rho\|_{\mathbb{L}^{\infty, p}_{2, r}(\tilde{T})} \nonumber\\
				&\quad+ 2C \Big(\int_{0}^{\tilde{T}} \limsup_{\mu \to\infty} \|u_\mu^\rho-u^\rho\|_{\mathbb{L}^{\infty, p}_{2, r}(s)} \|B\|_{\mathscr{L}(H,\mathcal{L}_2 (Y_1,H))} \|\rho_1(s)\|_{Y_1} ds\Big) \nonumber\\
				&\quad+\beta\int_{0}^{\tilde{T}}\limsup_{\mu \to\infty} \|u_\mu^\rho-u^\rho\|_{\mathbb{L}^{\infty, p}_{2, r}(s)}ds+CL_G\int_{0}^{\tilde{T}}\limsup_{\mu \to\infty} \|u_\mu^\rho-u^\rho\|_{\mathbb{L}^{\infty, p}_{2, r}(s)}\|\rho_2(s)\|_{Y_2}ds.
			\end{align}
			The estimate \eqref{boundforumuglobal} assures the existence of a small time $\tilde{T}=T_1$, such that, 
			\begin{align*}
				C	 \textstyle(T_1)^{\frac{p-r}{p}}\Big(\displaystyle\sup_\mu\|u_\mu^\rho\|_{L^{p}(0, T_1; L^{r}(\mathbb{R}^d))}+ \|u^\rho\|_{L^{p}(0,T_1; L^{r}(\mathbb{R}^d))}\Big)^{\alpha-1}\leq \frac{1}{2}.
			\end{align*}
			Therefore, from\eqref{mulip1}, we get
			\begin{align}\label{mulip2}
				&\limsup_{\mu \to\infty} \|u_\mu^\rho-u^\rho\|_{\mathbb{L}^{\infty, p}_{2, r}(T_1)}\nonumber\\
				& \leq 2 \int_{0}^{\tilde{T}}\limsup_{\mu \to\infty} \|u_\mu^\rho-u^\rho\|_{\mathbb{L}^{\infty, p}_{2, r}(s)}\big(\beta+2C\|B\|_{\mathscr{L}(H,\mathcal{L}_2 (Y_1,H))} \|\rho_1(s)\|_{Y_1}+CL_G \|\rho_2(s)\|_{Y_2}\big)ds.
			\end{align}
			Using Gr\"onwall's inequality in \eqref{mulip2}, we conclude
			\begin{align}\label{mulip3}
				\limsup_{\mu \to\infty} \|u_\mu^\rho-u^\rho\|_{\mathbb{L}^{\infty, p}_{2, r}(T_1)}=0.
			\end{align}
			Following similar approach we did above to get \eqref{mulip3}, we can extend the time domain and get, 
			\begin{align*}
				\limsup_{\mu \to\infty} \|u_\mu^\rho-u^\rho\|_{\mathbb{L}^{\infty, p}_{2, r}(T_1, 2T_1)}=0.
			\end{align*}
			Continuing this procedure, we can prove, 
			\begin{align}\label{muconv}
				\limsup_{\mu \to\infty} \|u_\mu^\rho-u^\rho\|_{\mathbb{L}^{\infty, p}_{2, r}(T)}=0.
			\end{align}
			Therefore, using the weak lower semicontinuity of norm, the solution  to \eqref{Skeleton}, $u^\rho=(u^\rho(t))_{t\in [0,T]} \in  \mathbb{L}^{\infty, p}_{2, r}(T)\cap C([0,T];H)$.
			Moreover, 
			\begin{align}\label{nindependentbound}
				\sup_{\rho \in \mathbb{D}^N}	\|u^\rho\|_{\mathbb{L}^{\infty, p}_{2, r}(T)} \leq \Big(\Big[\frac{T}{T_0}\Big]+1\Big)M.
			\end{align}
			Since the bound in \eqref{nindependentbound} is independent of $N$, for any $\rho \in \mathbb{D}$, \eqref{nindependentbound} holds true.
			Hence, we have the existence of global solution to \eqref{Skeleton}. This completes the proof of \Cref{wellskl}.
		\end{steps}
	\end{proof}
	\subsection{Conservation law in special case}\label{sub4conservation}
	In this subsection, we consider the particular case  $\beta=0, G\equiv0$ and show the conservation law holds.
	To show that, first we recall the Yosida approximated form \eqref{approxg} of the equation,	
	\begin{equation*}
		\left\{
		\begin{aligned}
			du_\mu^\rho(t)&=-\big[i Au_\mu^\rho(t) +i \mathcal{N}_\mu(u_\mu^\rho(t))+ \beta u_\mu^\rho(t)\big]dt - i B( u_\mu^\rho(t))\rho_1(t) dt- i   G_\mu(u_\mu^\rho(t))\rho_2(t)dt , \quad t>0,	\\
			u_\mu^\rho(0)&=J_\mu u_0.
		\end{aligned}
		\right.
	\end{equation*}
	We choose $\beta=0, G\equiv0$ and get the reduced equation,
	\begin{align}\label{approxreduced}
		\left\{
		\begin{aligned}
			du_\mu^\rho(t)&=-\big[i Au_\mu^\rho(t) +i \mathcal{N}_\mu(u_\mu^\rho(t))\big] - i B( u_\mu^\rho(t))\rho_1(t) dt , \quad t>0,	\\
			u_\mu^\rho(0)&=J_\mu u_0.
		\end{aligned}
		\right.
	\end{align}  
	We can choose $T_0$ (similarly as before we choose to get \eqref{bm}, \eqref{lipm}) such that there exists a unique solution $u_\mu^1 \in V_{M,T_0}$ of equation (\ref{approxreduced}).	
	Now, taking the $H^{-1}-H^1$ duality product of equation (\ref{approxreduced}) by $u_\mu^1$, we deduce, 
	\begin{align*}
		\frac{1}{2}\frac{d}{dt}\|u_\mu^1(t)\|_H^2&= \operatorname{Re} {_{H^1}}\langle u_\mu^1(t),\partial_t u_\mu^1(t)\rangle{_{H^{-1}}}\\
		&= \operatorname{Re} _{H^1}\langle u_\mu^1(t),-i A u_\mu^1(t) -i  \mathcal{N}_\mu(u_\mu^1(t))- i B(u_\mu^1(t))\rho_1(t)  \rangle{_{H^{-1}}}.
	\end{align*}
	Repeating the estimates \eqref{energyA}-\eqref{energyB} for $u_\mu^1$, we get
	\begin{align*}
		\frac{d}{dt}\|u_\mu^1(t)\|_H^2=0, \quad \text{for a.e. } t\in [0,T].
	\end{align*}
	Integrating over time from $0$ to $t$, we conclude
	\begin{align}\label{conservation}
		\|u_\mu^1(t)\|_H^2=\|u_\mu^1(0)\|_H^2=\|J_\mu u_0\|_H^2, \, t\in [0,T_0].
	\end{align}
	Now, we consider the equation,
	\begin{align}\label{mu2c}
		\left\{
		\begin{aligned}
			du_\mu^2(t)&=-\big[i Au_\mu^2(t) +i \mathcal{N}_\mu(u_\mu^2(t))\big]- i B( u_\mu^2(t))\rho_1(t) dt , \quad t>0,	\\
			u_\mu^2(0)&=u_\mu^1(T_0).
		\end{aligned}
		\right.
	\end{align}
	Following the similar approach as before, there exists a unique solution $u_\mu^2 \in V_{M,T_0} \cap C([0,T_0];H) $ to \eqref{mu2c} and the following is satisfied: 
	\begin{align*}
		\|u_\mu^2(t)\|_H^2=\|u_\mu^2(0)\|_H^2=\|u_\mu^1(T_0)\|_H^2=\|J_\mu u_0\|_H^2, \, t\in [0,T_0].
	\end{align*}
	Repeating the same arguments, we can see that, there exists a unique solution $u_\mu^k \in V_{M,T_0} \cap C([0,T_0];H) $ to the following equation:
	\begin{align}\label{mukc}
		\left\{
		\begin{aligned}
			du_\mu^k(t)&=-\big[i Au_\mu^k(t) +i \mathcal{N}_\mu(u_\mu^k(t))\big] - i B( u_\mu^k(t))\rho_1(t) dt , \quad t>0,	\\
			u_\mu^k(0)&=u_\mu^{k-1}(T_0),
		\end{aligned}
		\right.
	\end{align}
	and 
	\begin{align*}
		\|u_\mu^k(t)\|_H^2=\|u_\mu^k(0)\|_H^2=\|u_\mu^{k-1}(T_0)\|_H^2=\|J_\mu u_0\|_H^2, \, t\in [0,T_0].
	\end{align*}
	Now, we define
	\begin{align*}
		u^\rho_\mu(t)=u_\mu^k(t-(k-1)T_0), \quad t\in [(k-1)T_0, kT_0].
	\end{align*}
	Therefore, we have 	\(	u^\rho_\mu=	(u^\rho_\mu(t))_{t \in [0,T]} \in \mathbb{L}^{\infty, p}_{2, r}(T) \cap C([0,T];H)  \) as the unique solution of $\mu$- approximated equation (\ref{approxreduced}) on $[0,T]$, and 
	\begin{align*}
		\|u_\mu^\rho(t)\|_H^2=\|J_\mu u_0\|_H^2, \, t\in [0,T].
	\end{align*}
	Moreover, we get
	\begin{align}\label{nindepdent}
		\sup_{\rho \in \mathbb{D}^N}\sup_{\mu>0} 	\|u_\mu^\rho\|_{\mathbb{L}^{\infty, p}_{2, r}(T)} \leq \Big(\Big[\frac{T}{T_0}\Big]+1\Big)M.
	\end{align}
	Assume \(	u^\rho=	(u^\rho(t))_{t \in [0,T]} \in \mathbb{L}^{\infty, p}_{2, r}(T) \cap C([0,T];H)  \) is the unique solution to the mild form on $[0,T]$. Repeating the calculations in the derivation of \eqref{muconv}, with $\beta=0, G\equiv0$, we conclude
	\begin{align*}
		\lim_{\mu \to \infty} \sup_{t \in [0,{T}]} 	\|u_\mu^\rho(t)-u^\rho(t)\|_H=0.
	\end{align*}
	We already have the existence of unique solution $u^\rho_1\in C([0,T_0];H)$ of (\ref{approxreduced}) from the well-posedness of the skeleton equation. By $L^2$- conservation law, we have $\|u^\rho_1\|_{H}$ is constant for all $t\in [0,T_0]$. Using the similar arguments $u^\rho(t)=u^k(t-(k-1)T_0), \quad t\in [(k-1)T_0, kT_0]$ (and gluing technique), we define a function $u^\rho=(u^\rho(t))_{t\in [0,T]} \in  \mathbb{L}^{\infty, p}_{2, r}(T)\cap C([0,T];H)$ which is the unique solution to (\ref{approxreduced}) on $[0,T]$ and 
	\begin{align*}
		\|u^\rho(t)\|_{H}^2=\|u_0(t)\|_{H}^2, \  t\in [0,T].
	\end{align*}
	Moreover, since the bound \eqref{nindepdent} is independent of $N$, we have
	\begin{align*}
		\sup_{\rho \in \mathbb{D}} 	\|u^\rho\|_{\mathbb{L}^{\infty, p}_{2, r}(T)} \leq \Big(\Big[\frac{T}{T_0}\Big]+1\Big)M.
	\end{align*}
	\section{Well-posedness of stochastic controlled equation}\label{sectionsce} 
	In this section, we will prove the existence and uniqueness of the global mild solution of the following stochastic controlled equation \eqref{Sce}. First, we introduce the stochastic controlled equation  on a filtered probability space \((\Omega, \mathcal{F}, \mathbb{F}, \mathbb{P})\) as 
	\begin{align}\label{Sce}
		\left\{
		\begin{aligned}
			du^\rho(t)&=-\big[i Au^\rho(t) +i \mathcal{N}(u^\rho(t))+ \beta u^\rho(t) + \varepsilon b(u^\rho(t))\big]dt - i B(u^\rho(t))\rho_1(t)dt\\
			&\quad- i  G(u^\rho(t))\rho_2(t)dt - i\sqrt{\varepsilon}  B(u^\rho(t)) d\mathcal{W}_1(t)- i\sqrt{\varepsilon}  G(u^\rho(t)) d\mathcal{W}_2(t) , \quad t>0, \\
			u^\rho(0)&=u_0.
		\end{aligned}
		\right.
	\end{align}
	Here, we define the notion of a mild solution for \eqref{Sce}:
	\begin{defi}[Mild Solution]\label{mildsce}
		Fix any $q\geq 2$ and let $(p,r)$ be admissible pair with $r=\alpha+1$. Given any $u_0 \in H$, a mild solution to the stochastic controlled equation (\ref{Sce}) is an $\{\mathcal{F}_t\}_{t\geq0}$-adapted stochastic process $u^\rho = \{u^\rho(t): t\in [0,T]\} \in L^q(\Omega; C([0, T]; H)\cap L^p(0, T; L^r(\mathbb{R}^d)))$ satisfying 
		\begin{align}\label{mildformsce}
			u^\rho(t)&=S(t)u_0	- \int_{0}^{t} S(t-s) \big[i \mathcal{N}(u^\rho(s))+ \beta u^\rho(s) + \varepsilon b(u^\rho(t))\big]ds - i\int_{0}^{t} S(t-s)B(u^\rho(s))\rho_1(s)ds\nonumber\\ 
			& \quad- i\int_{0}^{t} S(t-s)G(u^\rho(s))\rho_2(s)ds- i \sqrt{\varepsilon}\int_{0}^{t} S(t-s)\big[B(u^\rho(s))d\mathcal{W}_1(s)+G(u^\rho(s))d\mathcal{W}_2(s) \big],
		\end{align}
		$\mathbb{P}$\,-a.s., for all $t\in[0,T]$. 
	\end{defi}
	\begin{thm}\label{wellsce}
		Let $q\geq 2, r=\alpha+1$ and $1< \alpha< \frac{4}{d}+1$, choose a number p such that (p,r) is an admissible pair and $u_0\in H$.Under the assumptions taken in Subsection \ref{Ass}, for any $\rho\in \mathbb{S}$, there exists a unique global mild solution $u^\rho \in L^q(\Omega; C([0, T]; H)  \cap L^p(0, T; L^r(\mathbb{R}^d)))$ to the stochastic controlled equation \eqref{Sce} satisfying the mild form \eqref{mildformsce}.
	\end{thm}
	To prove this theorem, we first consider an equation where the nonlinear term is multiplied by a suitable truncating function, and establish the existence of a unique solution to the truncated equation. Then, we show that, for sufficiently large support of the truncating function, the solutions of the truncated equation are the solutions of the original equation. Since for any $\rho\in \mathbb{S}$, there exists $N \in \mathbb{N}$ such that $\rho\in \mathbb{S}_N$, without loss of generality, we work with $\rho\in \mathbb{S}_N$.
	\subsection{A truncated equation}\label{subtruncated }
	Let $\theta \in C_0^\infty(\mathbb{R})$ with $\mathrm{supp} (\theta) \in (-2,2)$ such that 
	\begin{align*}
		\theta(x):=\begin{cases*}
			1, \qquad \qquad\,\, x \in [-1,1],\\
			\in [0,1], \qquad x \in [-2,2],\\
			0, \qquad\qquad\,\,\,\text{otherwise}.
		\end{cases*}
	\end{align*}
	Take $R>0$ and define $\theta_R(x)=\theta(x/R)$. Therefore $\theta_R(x)= 1$ for $x \in [-R,R]$.
	We consider the following truncated equation:
	\begin{align}\label{truncated}
		\left\{
		\begin{aligned}
			du^\rho(t)&=-\big[i Au^\rho(t) +i \theta_R(	\|u^\rho\|_{\mathbb{L}^{\infty, p}_{2, r}(t)})\mathcal{N}(u^\rho(t))+ \beta u^\rho(t) + \varepsilon b(u^\rho(t))\big]dt - i B(u^\rho(t))\rho_1(t)dt\\
			&\quad -i  G(u^\rho(t))\rho_2(t)dt - i\sqrt{\varepsilon}  B(u^\rho(t)) d\mathcal{W}_1(t)- i\sqrt{\varepsilon}  G(u^\rho(t)) d\mathcal{W}_2(t) , \quad t>0, \\
			u^\rho(0)&=u_0,
		\end{aligned}
		\right.
	\end{align} 
	and the corresponding mild form, 
	\begin{align}
		u^\rho(t)&=S(t)u_0	- i\int_{0}^{t} S(t-s)\theta_R(	\|u^\rho\|_{\mathbb{L}^{\infty, p}_{2, r}(s)}) \mathcal{N}(u^\rho(s))ds- \int_{0}^{t} S(t-s) \big[\beta u^\rho(s) + \varepsilon b(u^\rho(t))\big]ds\nonumber\\ 
		&\quad - i\int_{0}^{t} S(t-s)\big[B(u^\rho(s))\rho_1(s)+G(u^\rho(s))\rho_2(s) \big]ds\nonumber\\
		&\quad- i \sqrt{\varepsilon}\int_{0}^{t} S(t-s)\big[B(u^\rho(s))d\mathcal{W}_1(s)+G(u^\rho(s))d\mathcal{W}_2(s) \big].
	\end{align}
	To establish the well-posedness of the truncated equation, we need the following theorem:
	\begin{thm}\label{liptrunc}
		Assume $1< \alpha< \frac{4}{d}+1, r=\alpha+1$. Denote 
		\begin{align*}
			I_R(u^\rho(t)):=\int_{0}^{t} S(t-s)\theta_R(	\|u^\rho(s)\|_{\mathbb{L}^{\infty, p}_{2, r}(s)}) \mathcal{N}(u^\rho(s))ds.
		\end{align*}
		Then for every $T>0$, the function $I_R$ maps $\mathbb{L}^{\infty, p}_{2, r}(T)$ to itself and for $u_1, u_2 \in \mathbb{L}^{\infty, p}_{2, r}(T)$, we have 
		\begin{align*}
			\|I_R(u_1)-I_R(u_2)\|_{\mathbb{L}^{\infty, p}_{2, r}(T)} \leq C T^{\frac{p-r}{p}}\|u_1-u_2\|_{\mathbb{L}^{\infty, p}_{2, r}(T)}.
		\end{align*}
	\end{thm}		
	\begin{proof}
		Let us fix $T>0$ and take $u\in \mathbb{L}^{\infty, p}_{2, r}(T)$.  Let us define the stopping time,
		$$\tau:=\inf  \{t\geq 0: 	\|u\|_{\mathbb{L}^{\infty, p}_{2, r}(t)}> 2R\} \land T.$$
		We have $\theta_R(	\|u\|_{\mathbb{L}^{\infty, p}_{2, r}(t)})=0$ for  $\|u\|_{\mathbb{L}^{\infty, p}_{2, r}(t)}>2R$. Since the map, $t \mapsto \|u\|_{\mathbb{L}^{\infty, p}_{2, r}(t)}$ is non-decreasing on $[0,T],$ we have $\theta_R(	\|u\|_{\mathbb{L}^{\infty, p}_{2, r}(t)})=0$ for  $t\geq \tau$.	Now, by applying the Strichartz estimates (\Cref{Strichartz}), we have
		\begin{align*}
			&\|I_R(u)\|_{L^\infty(0,T; H)}\leq C \|\theta_R(	\|u\|_{\mathbb{L}^{\infty, p}_{2, r}(\cdot)}) |u|^{\alpha-1}u\|_{L^{p'}(0,T; L^{r'}(\mathbb{R}^d))}, \\
			&	\|I_R(u)\|_{L^p(0,T; L^r(\mathbb{R}^d))}\leq C \|\theta_R(	\|u\|_{\mathbb{L}^{\infty, p}_{2, r}(\cdot)}) |u|^{\alpha-1}u\|_{L^{p'}(0,T; L^{r'}(\mathbb{R}^d))}.
		\end{align*}
		Therefore, using H\"older's inequality with conjugate exponents $s=\frac{r'p}{p'r}, t=\frac{p-1}{p-r}$ and the boundedness of the truncating function, we get
		\begin{align}\label{boundftrunc}
			\|I_R(u)\|_{\mathbb{L}^{\infty, p}_{2, r}(t)}&\leq  C \|\theta_R(	\|u\|_{\mathbb{L}^{\infty, p}_{2, r}(\cdot)}) |u|^{\alpha-1}u\|_{L^{p'}(0,T; L^{r'}(\mathbb{R}^d))}\nonumber\\
			&\leq  C \|\theta_R(	\|u\|_{\mathbb{L}^{\infty, p}_{2, r}(\cdot)})\|_{L^\infty(0,\tau)}  \||u|^{\alpha-1}u\|_{L^{p'}(0,\tau; L^{r'}(\mathbb{R}^d))}\nonumber\\
			&\leq  C \tau^{\frac{p-r}{p}} \|u\|^\alpha_{L^{p}(0,\tau; L^{r}(\mathbb{R}^d))}\leq  C_R T^{\frac{p-r}{p}}\|u\|^\alpha_{\mathbb{L}^{\infty, p}_{2, r}(T)}.
		\end{align}
		We take $u_1, u_2 \in \mathbb{L}^{\infty, p}_{2, r}(T)$ and define the stopping times, 
		$$\tau_i:=\inf \big\{t\geq 0: 	\|u_i\|_{\mathbb{L}^{\infty, p}_{2, r}(t)}> 2R\big\} \land T \quad \text{ for } i=1,2.$$ 
		Without loss of generality, we can assume $\tau_1 \leq \tau_2.$ So, we can partition the interval as, $[0,T]=[0,\tau_1] \cup [\tau_1, \tau_2] \cup [\tau_2, T].$
		Using the Strichartz estimates (\Cref{Strichartz}) again, we have
		\begin{align}
			&\|I_R(u_1)-I_R(u_2)\|_{L^\infty(0,T; H)}\leq C \|\theta_R(	\|u_1\|_{\mathbb{L}^{\infty, p}_{2, r}(\cdot)}) |u_1|^{\alpha-1}u_1-\theta_R(\|u_2\|_{\mathbb{L}^{\infty, p}_{2, r}(\cdot)}) |u_2|^{\alpha-1}u_2\|_{L^{p'}(0,T; L^{r'}(\mathbb{R}^d))}, \label{l1}\\
			&\|I_R(u_1)-I_R(u_2)\|_{L^p(0,T; L^r(\mathbb{R}^d))}\nonumber\\
			&\leq C \|\theta_R(	\|u_1\|_{\mathbb{L}^{\infty, p}_{2, r}(\cdot)}) |u_1|^{\alpha-1}u_1-\theta_R(\|u_2\|_{\mathbb{L}^{\infty, p}_{2, r}(\cdot)}) |u_2|^{\alpha-1}u_2\|_{L^{p'}(0,T; L^{r'}(\mathbb{R}^d))}.\label{l2}
		\end{align}
		Therefore, using \eqref{l1} and \eqref{l2}, we have
		\begin{align}
			&\|I_R(u_1)-I_R(u_2)\|_{\mathbb{L}^{\infty, p}_{2, r}(T)}\nonumber\\
			&\leq  2C \|\theta_R(	\|u_1\|_{\mathbb{L}^{\infty, p}_{2, r}(\cdot)}) |u_1|^{\alpha-1}u_1-\theta_R(\|u_2\|_{\mathbb{L}^{\infty, p}_{2, r}(\cdot)}) |u_2|^{\alpha-1}u_2\|_{L^{p'}(0,T; L^{r'}(\mathbb{R}^d))}\nonumber\\
			&\leq   2C \|\theta_R(	\|u_1\|_{\mathbb{L}^{\infty, p}_{2, r}(\cdot)}) |u_1|^{\alpha-1}u_1-\theta_R(\|u_2\|_{\mathbb{L}^{\infty, p}_{2, r}(\cdot)}) |u_1|^{\alpha-1}u_1\|_{L^{p'}(0,\tau_1; L^{r'}(\mathbb{R}^d))}\nonumber\\
			&\quad+ 2C \|\theta_R(\|u_2\|_{\mathbb{L}^{\infty, p}_{2, r}(\cdot)}) \{|u_1|^{\alpha-1}u_1- |u_2|^{\alpha-1}u_2\}\|_{L^{p'}(0,\tau_1; L^{r'}(\mathbb{R}^d))}\nonumber\\
			&\quad+ 2C \|\theta_R(\|u_2\|_{\mathbb{L}^{\infty, p}_{2, r}(\cdot)}) |u_2|^{\alpha-1}u_2\|_{L^{p'}(\tau_1,\tau_2; L^{r'}(\mathbb{R}^d))}\nonumber\\
			&= I^R_1+I^R_2+I^R_3. \label{l}
		\end{align}
		To estimate $I^R_1$ in \Cref{liptrunc}, we establish the following lemma first.
		\begin{lem}\label{lemI1}
			Let  $u_1, u_2 \in \mathbb{L}^{\infty, p}_{2, r}(T)$ and $u_3 \in L^{p'}(0,T; L^{r'}(\mathbb{R}^d))$, then we have
			\begin{align*}
				A &\equiv \,\|\{\theta_R(	\|u_1\|_{\mathbb{L}^{\infty, p}_{2, r}(\cdot)})-\theta_R(\|u_2\|_{\mathbb{L}^{\infty, p}_{2, r}(\cdot)})\} u_3\|_{L^{p'}(0,T; L^{r'}(\mathbb{R}^d))}\\
				&\leq  C_R \|u_1-u_2\|_{\mathbb{L}^{\infty, p}_{2, r}(T)}\|u_3\|_{L^{p'}(0,T; L^{r'}(\mathbb{R}^d))}.
			\end{align*}
			\begin{proof} For the sake of simplifying the calculations, we consider $A^{p'}$ and use Lagrange's mean value theorem on $\theta_R$ to estimate,
				\begin{align*}
					A^{p'}&=\|\{\theta_R(	\|u_1\|_{\mathbb{L}^{\infty, p}_{2, r}(\cdot)})-\theta_R(\|u_2\|_{\mathbb{L}^{\infty, p}_{2, r}(\cdot)})\} u_3\|^{p'}_{L^{p'}(0,T; L^{r'}(\mathbb{R}^d))}\\
					&= \int_{0}^{T}\|\{\theta_R(	\|u_1\|_{\mathbb{L}^{\infty, p}_{2, r}(t)})-\theta_R(\|u_2\|_{\mathbb{L}^{\infty, p}_{2, r}(t)})\} u_3(t)\|^{p'}_{ L^{r'}(\mathbb{R}^d))}dt\\
					&= \int_{0}^{T}\|\theta'_R(\xi)\{	\|u_1\|_{\mathbb{L}^{\infty, p}_{2, r}(t)}-\|u_2\|_{\mathbb{L}^{\infty, p}_{2, r}(t)}\} u_3\|^{p'}_{ L^{r'}(\mathbb{R}^d)}dt\\
					&\leq  \|\theta'_R(\xi)\|^{p'}_{L^\infty(0,T)} \int_{0}^{T}\|u_1-u_2\|^{p'}_{\mathbb{L}^{\infty, p}_{2, r}(t)}\| u_3(t)\|^{p'}_{ L^{r'}(\mathbb{R}^d)}dt\\
					&\leq  \|\theta'_R(\xi)\|^{p'}_{L^\infty(0,T)} \|u_1-u_2\|^{p'}_{\mathbb{L}^{\infty, p}_{2, r}(T)}\int_{0}^{T}\| u_3(t)\|^{p'}_{ L^{r'}(\mathbb{R}^d)}dt\\
					&\leq  \|\theta'_R(\xi)\|^{p'}_{L^\infty(0,T)} \|u_1-u_2\|^{p'}_{\mathbb{L}^{\infty, p}_{2, r}(T)}\|u_3\|^{p'}_{L^{p'}(0,T; L^{r'}(\mathbb{R}^d))}\\
					&\leq  C_R \|u_1-u_2\|^{p'}_{\mathbb{L}^{\infty, p}_{2, r}(T)}\|u_3\|^{p'}_{L^{p'}(0,T; L^{r'}(\mathbb{R}^d))}.
				\end{align*}
				Hence, \Cref{lemI1} is straight forward.
			\end{proof}
		\end{lem}
		We use \Cref{lemI1} with $T=\tau_1$ and H\"older's inequality with conjugate exponents $s=\frac{r'p}{p'r}, t=\frac{p-1}{p-r}$, to estimate $I^R_1$ as 
		\begin{align}
			I^R_1&=2C \|\theta_R(	\|u_1\|_{\mathbb{L}^{\infty, p}_{2, r}(\cdot)}) |u_1|^{\alpha-1}u_1-\theta_R(\|u_2\|_{\mathbb{L}^{\infty, p}_{2, r}(\cdot)}) |u_1|^{\alpha-1}u_1\|_{L^{p'}(0,\tau_1; L^{r'}(\mathbb{R}^d))}\nonumber\\
			&\leq  2C C_R \|u_1-u_2\|_{\mathbb{L}^{\infty, p}_{2, r}(\tau_1)}\||u_1|^{\alpha-1}u_1\|_{L^{p'}(0,\tau_1; L^{r'}(\mathbb{R}^d))}\nonumber\\
			&\leq  2C C_R \|u_1-u_2\|_{\mathbb{L}^{\infty, p}_{2, r}(\tau_1)}\||u_1|^{\alpha}\|_{L^{p'}(0,\tau_1; L^{r'}(\mathbb{R}^d))}\nonumber\\
			&\leq C_R \tau_1^{\frac{p-r}{p}} \|u_1-u_2\|_{\mathbb{L}^{\infty, p}_{2, r}(\tau_1)}\|u_1\|^{\alpha}_{L^{p}(0,\tau_1; L^{r}(\mathbb{R}^d))}\nonumber\\
			&\leq C_R T^{\frac{p-r}{p}}R^\alpha \|u_1-u_2\|_{\mathbb{L}^{\infty, p}_{2, r}(T)}.\label{b1}
		\end{align}
		The boundedness of the truncating function and \eqref{liptypeF} give the estimate of $I^R_2$ as
		\begin{align}
			I^R_2&= 2C \|\theta_R(\|u_2\|_{\mathbb{L}^{\infty, p}_{2, r}(\cdot)}) \{|u_1|^{\alpha-1}u_1- |u_2|^{\alpha-1}u_2\}\|_{L^{p'}(0,\tau_1; L^{r'}(\mathbb{R}^d))}\nonumber\\
			&\leq  2C C' \| |u_1|^{\alpha-1}u_1- |u_2|^{\alpha-1}u_2\|_{L^{p'}(0,\tau_1; L^{r'}(\mathbb{R}^d))}\nonumber\\
			&\leq  2C C' \tau_1^{\frac{p-r}{p}} (\|u_1\|_{\mathbb{L}^{\infty, p}_{2, r}(\tau_1)}+\|u_2\|_{\mathbb{L}^{\infty, p}_{2, r}(\tau_1)})^{\alpha-1}\|u_1-u_2\|_{\mathbb{L}^{\infty, p}_{2, r}(\tau_1)}\nonumber\\
			&\leq  2C C' T^{\frac{p-r}{p}} (4R)^{\alpha-1} \|u_1-u_2\|_{\mathbb{L}^{\infty, p}_{2, r}(\tau_1)}\nonumber\\
			&\leq  C_R T^{\frac{p-r}{p}} R^{\alpha-1} \|u_1-u_2\|_{\mathbb{L}^{\infty, p}_{2, r}(T)}.\label{b2}
		\end{align}
		Since  $\theta_R(	\|u_1\|_{\mathbb{L}^{\infty, p}_{2, r}(t)})=0$ for  $t\in (\tau_1, \tau_2)$, we can use the \Cref{lemI1} again to get, 
		\begin{align}
			I^R_3&=2C \|\theta_R(\|u_2\|_{\mathbb{L}^{\infty, p}_{2, r}(\cdot)}) |u_2|^{\alpha-1}u_2\|_{L^{p'}(\tau_1,\tau_2; L^{r'}(\mathbb{R}^d))}\nonumber\\
			&=2C \|\{\theta_R(\|u_1\|_{\mathbb{L}^{\infty, p}_{2, r}(\cdot)})-\theta_R(\|u_2\|_{\mathbb{L}^{\infty, p}_{2, r}(\cdot)})\} |u_2|^{\alpha-1}u_2\|_{L^{p'}(\tau_1,\tau_2; L^{r'}(\mathbb{R}^d))}\nonumber\\
			&\leq  C_R T^{\frac{p-r}{p}}\|u_1-u_2\|_{\mathbb{L}^{\infty, p}_{2, r}(T)}\|u_2\|^{\alpha}_{L^{p}(\tau_1, \tau_2; L^{r}(\mathbb{R}^d))}\nonumber\\
			&\leq  C_R T^{\frac{p-r}{p}}R^\alpha\|u_1-u_2\|_{\mathbb{L}^{\infty, p}_{2, r}(T)}.\label{b3}
		\end{align}
		Therefore, using the estimates \eqref{b1}-\eqref{b3} in \eqref{l}, we conclude
		\begin{align*}
			\|I_R(u_1)-I_R(u_2)\|_{\mathbb{L}^{\infty, p}_{2, r}(T)} &\leq  I^R_1+I^R_2+I^R_3\leq  C_R T^{\frac{p-r}{p}}\|u_1-u_2\|_{\mathbb{L}^{\infty, p}_{2, r}(T)}.
		\end{align*}
		This completes the proof of \Cref{liptrunc}.
	\end{proof}
	Let us now prove the existence of a unique solution of the truncated equation \eqref{truncated} using \Cref{liptrunc}. 
	\vspace{-.3cm}
	\begin{thm}\label{wellposetrunc}
		Let $r=\alpha+1$ and $1< \alpha< \frac{4}{d}+1$, choose p such that (p,r) is an admissible pair and $u_0\in H$.Under the assumptions taken in Subsection \ref{Ass}, there exists a unique golobal mild solution $u^\rho \in L^q(\Omega; C([0, T]; H) \cap L^p(0, T; L^r(\mathbb{R}^d)))$ of the truncated equation satisfying the mild form
		\begin{align}\label{mildtrunc}
			u^\rho(t)&=S(t)u_0	- i\int_{0}^{t} S(t-s)\theta_R(	\|u^\rho\|_{\mathbb{L}^{\infty, p}_{2, r}(s)}) F(u^\rho(s))ds- \int_{0}^{t} S(t-s) \big[\beta u^\rho(s) + \varepsilon b(u^\rho(t))\big]ds\nonumber\\
			& \quad- i\int_{0}^{t} S(t-s)\big[B(u^\rho(s))\rho_1(s)+G(u^\rho(s))\rho_2(s) \big]ds \nonumber\\
			&\quad- i \sqrt{\varepsilon}\int_{0}^{t} S(t-s)\big[B(u^\rho(s))d\mathcal{W}_1(s)+G(u^\rho(s))d\mathcal{W}_2(s) \big], \quad\mathbb{P}\,\text{-a.s.}
		\end{align}
		for all $t\in[0,T]$.
	\end{thm}
	\begin{proof}
		First we use the fixed point argument in the Banach space $\mathbb{X}=L^q(\Omega; L^\infty(0,T; H) \cap L^p(0, T; L^r(\mathbb{R}^d)))$ for some sufficiently small time $T_0$ depending on $\beta, \alpha, M, R,  \|u_0\|_{H}$ to prove the local existence of a solution. Then using mathematical induction, we prove the global existence of a unique mild solution.
		\subsubsection{Local existence of the solution of truncated equation:}\label{loctrun}
		To prove the local existence of a solution, consider the level set
		\begin{equation*}
			V^{\mathbb{X}}_{M,\tilde{T}} :=
			\Big\{
			u \in \mathbb{X} : \|u\|_{\mathbb{X}}
			\leq M
			\Big\}, \quad \tilde{T} \in (0,T] \text{ and } M>0.
		\end{equation*}
		Define the mapping, $\mathfrak{L}: V^{\mathbb{X}}_{M,\tilde{T}} \to  V^{\mathbb{X}}_{M,\tilde{T}}$ by,
		\begin{align}\label{Iomildtrunc}
			\mathfrak{L}(u, u_0, \rho)(t)&:=S(t)u_0	- i\int_{0}^{t} S(t-s)\theta_R(	\|u^\rho\|_{\mathbb{L}^{\infty, p}_{2, r}(s)}) \mathcal{N}(u^\rho(s))ds- \int_{0}^{t} S(t-s) \big[\beta u^\rho(s) + \varepsilon b(u^\rho(t))\big]ds\nonumber\\ 
			& \quad- i\int_{0}^{t} S(t-s)\big[B(u^\rho(s))\rho_1(s)+G(u^\rho(s))\rho_2(s) \big]ds\nonumber\\ 
			&\quad- i \sqrt{\varepsilon}\int_{0}^{t} S(t-s)\big[B(u^\rho(s))d\mathcal{W}_1(s)+G(u^\rho(s))d\mathcal{W}_2(s) \big],
		\end{align}
		for all $t\in[0,T]$.
		First, we will prove the energy estimate. 
		To do this, we need some estimates for the following stochastic integrals
		\begin{align}\label{kbg}
			K_B(t)=K_B(u)(t):=\int_{0}^{t} S(t-s)B(u^\rho(s))d\mathcal{W}_1(s) \,\text{ and }	K_G(t)=K_G(u)(t):=\int_{0}^{t} S(t-s)G(u^\rho(s))d\mathcal{W}_2(s).
		\end{align}
		Consider $\mathbb{X}_1=L^q(\Omega; L^\infty(0,T; H))$ and $\mathbb{X}_2=L^q(\Omega;L^p(0, T; L^r(\mathbb{R}^d)))$, then  $\mathbb{X}=\mathbb{X}_1 \cap\mathbb{X}_2$.
		\vskip 0.1cm\noindent
		\textbf{Estimates for $K_B(t)$:}
		First we estimate $	K_B(\cdot)$ in $\mathbb{X}_1$. To do so we use the Burkholder-Davis-Gundy inequality (\Cref{BDG}), the boundedness of semigroup $\{S(t)\}_{t\in\mathbb{R}}$ to get,
		\begin{align}\label{kbx1}
			&\|	K_B(\cdot)\|^q_{\mathbb{X}_1}=	\bigg\|\int_{0}^{\cdot} S(\cdot-s)B(u^\rho(s))d\mathcal{W}_1(s)\bigg\|^q_{\mathbb{X}_1}=\mathbb{E}\Big\{\Big(\Big\|\int_{0}^{\cdot} S(\cdot-s)B(u^\rho(s))d\mathcal{W}_1(s)\Big\|^q_{L^\infty(0,T; H)}   \Big)\Bigr\} \nonumber\\
			&=\mathbb{E}\Big\{\Big(  \sup_{t \in [0,T]}	 \Big\|\int_{0}^{t} S(t-s)B(u^\rho(s))d\mathcal{W}_1(s) \Big\|^q_{H}   \Big)\Bigr\}\leq C \mathbb{E}\Big\{\Big( \int_{0}^{t}\| B(u^\rho(s))\|^2_{\mathcal{L}_2 (Y_1,H)}ds   \Big)^{q/2}\Big\}\nonumber\\
			&\leq C \mathbb{E}\Big\{\Big( \int_{0}^{t} \|B\|^2_{\mathscr{L}(H,\mathcal{L}_2 (Y_1,H))}\|u^\rho(s)\|^2_Hds   \Big)^{q/2}\Big\}\leq C \|B\|^q_{\mathscr{L}(H,\mathcal{L}_2 (Y_1,H))}\mathbb{E}\Big\{\Big( \int_{0}^{t} \|u^\rho(s)\|^2_Hds   \Big)^{q/2}\Big\}\nonumber\\
			&\leq C \|B\|^q_{\mathscr{L}(H,\mathcal{L}_2 (Y_1,H))}\mathbb{E}\Big\{\Big( T \|u^\rho(s)\|^2_{L^\infty(0,T; H)}  \Big)^{q/2}\Big\}\leq C \|B\|^q_{\mathscr{L}(H,\mathcal{L}_2 (Y_1,H))} T^{q/2}	\|u^\rho\|^q_{\mathbb{X}_1}.
		\end{align}
		To estimate $K_B(\cdot)$ in $\mathbb{X}_2$, we fix $t_0$ and consider
		\begin{align}\label{kbt0}
			K_B(t_0,t)=\int_{0}^{t_0} S(t-s)B(u^\rho(s))d\mathcal{W}_1(s),
		\end{align}
		and establish the following lemma for $K_B(t_0,t)$\,:
		\begin{lem}\label{lemkbt0}
			For $u \in L^q(\Omega;L^\infty(0, T; L^r(\mathbb{R}^d)))$ and $K_B(t_0,t)$ defined as above in \eqref{kbt0}, for any $t\in [0,T]$, we have the estimate
			\begin{align*}
				\mathbb{E}\Big\{  \sup_{t_0 \in [0,T]}	\|K_B(t_0,t)\|^q_{L^r(\mathbb{R}^d)}\Big\}
				\leq C \|B\|^q_{\mathscr{L}(H,\mathcal{L}_2 (Y_1,H))} T^{q/2}	\|u^\rho\|^q_{\mathbb{X}_1}.
			\end{align*}
		\end{lem}
		\begin{proof}
			We have, \begin{align*}
				&\mathbb{E}\Big\{ \sup_{t_0 \in [0,T]} \|K_B(t_0,t)\|^q_{L^r(\mathbb{R}^d)}\Big\}=\mathbb{E}\Big\{  \sup_{t \in [0,T]}	\bigg\|\int_{0}^{t_0} S(t-s)B(u^\rho(s))d\mathcal{W}_1(s)\bigg\|^q_{L^r(\mathbb{R}^d)}\Big\}.
			\end{align*}
			Using the Burkholder-Davis-Gundy inequality (\Cref{BDG}) and boundedness of the semigroup, we deduce
			\begin{align*}
				\mathbb{E}\Big\{  \sup_{t_0 \in [0,T]}	\|K_B(t_0,t)\|^q_{L^r(\mathbb{R}^d)}\Big\}	&\leq C \mathbb{E}\Big\{\Big( \int_{0}^{t_0}\| B(u^\rho(s))\|^2_{\mathcal{L}_2 (Y_1,H)}ds   \Big)^{q/2}\Big\}\\
				&\leq C \|B\|^q_{\mathscr{L}(H,\mathcal{L}_2 (Y_1,H))} T^{q/2}	\|u^\rho\|^q_{\mathbb{X}_1}.
			\end{align*}
			This completes the proof of \Cref{lemkbt0}.
		\end{proof}
		For the estimation of $K_B(\cdot)$ in $\mathbb{X}_2$, we use $K_B(t)=K_B(t,t)$ and obtain
		\begin{align*}
			\|	K_B(\cdot)\|^q_{\mathbb{X}_2}&=	\|K_B(\cdot,\cdot)\|^q_{\mathbb{X}_2}=\mathbb{E}\Big\{\|K_B(\cdot,\cdot)\|^q_{L^p(0, T; L^r(\mathbb{R}^d))}\Big\} \leq\mathbb{E}\Big\{\Big( \int_{0}^{T} 	\|K_B(t,t)\|^p_{L^r(\mathbb{R}^d)}dt   \Big)^{q/p}\Big\}.
		\end{align*}
		Using H\"older's inequality in time, we get
		\begin{align}\label{kbx2}
			\|	K_B(\cdot)\|^q_{\mathbb{X}_2}&\leq \mathbb{E}\Big\{\Big( (T)^{\frac{q-p}{q}}\Big(\int_{0}^{T} 	\|K_B(t,t)\|^{p\cdot\frac{q}{p}}_{L^r(\mathbb{R}^d)}dt\Big) ^{p/q}  \Big)^{q/p}\Big\}\leq (T)^{\frac{q-p}{p}}\mathbb{E}\Big\{\int_{0}^{T} 	\|K_B(t,t)\|^{q}_{L^r(\mathbb{R}^d)}dt\Big\}\nonumber\\
			&\leq (T)^{\frac{q-p}{p}}\int_{0}^{T} \mathbb{E} \Big\{	\sup_{t_0 \in [0,T]}\|K_B(t_0,t)\|^{q}_{L^r(\mathbb{R}^d)}\Big\}dt\nonumber\\
			& \leq (T)^{\frac{q-p}{p}}\int_{0}^{T}\big\{ C \|B\|^q_{\mathscr{L}(H,\mathcal{L}_2 (Y_1,H))} T^{q/2}	\|u^\rho\|^q_{\mathbb{X}_1}\big\}dt \nonumber\\
			&\leq T^{\frac{q}{p}+\frac{q}{2}}\big\{ C \|B\|^q_{\mathscr{L}(H,\mathcal{L}_2 (Y_1,H))}	\|u^\rho\|^q_{\mathbb{X}_1}\big\}.
		\end{align}
		\vskip 0.1cm\noindent 
		\textbf{Estimate for $	K_G(\cdot)$:}
		We use Burkholder-Davis-Gundy inequality (\Cref{BDG}), boundedness of semigroup and estimate $K_G(\cdot)$ in $\mathbb{X}_1$ as
		\begin{align}\label{kgx1}
			\|	K_G(\cdot)\|^q_{\mathbb{X}_1}&=	\bigg\|\int_{0}^{\cdot} S(\cdot-s)G(u^\rho(s))d\mathcal{W}_2(s)\bigg\|^q_{\mathbb{X}_1}=\mathbb{E}\bigg\{ \bigg\|\int_{0}^{\cdot} S(\cdot-s)G(u^\rho(s))d\mathcal{W}_2(s)\bigg\|^q_{L^\infty(0,T; H)}   \bigg\}\nonumber\\
			&=\mathbb{E}\bigg\{  \sup_{t \in [0,T]}	\bigg\|\int_{0}^{t} S(t-s)G(u^\rho(s))d\mathcal{W}_2(s)\bigg\|^q_{H}   \bigg\}
			\leq C \mathbb{E}\Big\{\Big( \int_{0}^{T}\| G(u^\rho(s))\|^2_{\mathcal{L}_2 (Y_2,H)}ds   \Big)^{q/2}\Big\}\nonumber\\
			&\leq C \mathbb{E}\Big\{\Big( \int_{0}^{T} ( C_1+C_2\|u^\rho(s)\|_H)^2ds   \Big)^{q/2}\Big\}\leq C \mathbb{E}\Big\{\Big( \int_{0}^{T} ( 1+\|u^\rho(s)\|_H^2)ds   \Big)^{q/2}\Big\}\nonumber\\
			&\leq CT^{q/2}+C \mathbb{E}\Big\{\Big( \int_{0}^{T} \|u^\rho(s)\|_H^2ds   \Big)^{q/2}\Big\}\leq CT^{q/2} +CT^{q/2}	\|	u^\rho\|^q_{\mathbb{X}_1}.
		\end{align}
		To estimate $	K_G(\cdot)$ in $\mathbb{X}_2$, we fix $t_0$ and consider, 
		\begin{align}\label{kgt0}
			K_G(t_0,t)=\int_{0}^{t_0} S(t-s)G(u^\rho(s))d\mathcal{W}_2(s).
		\end{align}
		\begin{lem}\label{lemkgt0}
			For $u \in L^q(\Omega;L^\infty(0, T; L^r(\mathbb{R}^d)))$ and $K_G(t_0,t)$ defined as above in \eqref{kgt0}, for any $t\in [0,T]$, we have the estimate
			\begin{align*}
				\mathbb{E}\Big\{ \sup_{t_0 \in [0,T]}	\|K_G(t_0,t)\|^q_{L^r(\mathbb{R}^d)}\Big\} \leq CT^{q/2} +CT^{q/2}	\|	u^\rho\|^q_{\mathbb{X}_1}.
			\end{align*}
		\end{lem}
		\begin{proof}
			Approaching the same way as above in \Cref{lemkbt0} and \eqref{kgx1}, using the  Burkholder-Davis-Gundy inequality (\Cref{BDG}) and boundedness of the semigroup, we easily get
			\begin{align*}
				\mathbb{E}\Big\{  \sup_{t_0 \in [0,T]}	\|K_G(t_0,t)\|^q_{L^r(\mathbb{R}^d)}\Big\}	&\leq C \mathbb{E}\Big\{\Big( \int_{0}^{t_0}\| G(u^\rho(s))\|^2_{\mathcal{L}_2 (Y_2,H)}ds   \Big)^{q/2}\Big\}\\
				&\leq CT^{q/2} +CT^{q/2}	\|	u^\rho\|^q_{\mathbb{X}_1}.
			\end{align*}
			This completes the proof of \Cref{lemkgt0}.
		\end{proof}
		\noindent Now, to estimate $	K_G(\cdot)$ in $\mathbb{X}_2$, we use $	K_G(t)=K_G(t,t)$ and find 
		\begin{align*}
			\|	K_G(\cdot)\|^q_{\mathbb{X}_2}&=	\|K_G(\cdot,\cdot)\|^q_{\mathbb{X}_2}=\mathbb{E}\Big\{	\|K_G(\cdot,\cdot)\|^q_{L^p(0, T; L^r(\mathbb{R}^d))}   \Big\}\leq\mathbb{E}\Big\{\Big( \int_{0}^{T} 	\|K_G(t,t)\|^p_{L^r(\mathbb{R}^d)}dt   \Big)^{q/p}\Big\}.
		\end{align*}
		H\"older's inequality in time and \Cref{lemkgt0} together imply,
		\begin{align}\label{kgx2}
			\|	K_G(\cdot)\|^q_{\mathbb{X}_2}&\leq \mathbb{E}\Big\{\Big( (T)^{\frac{q-p}{q}}\Big(\int_{0}^{T} 	\|K_G(t,t)\|^{p\cdot\frac{q}{p}}_{L^r(\mathbb{R}^d)}dt\Big) ^{p/q}  \Big)^{q/p}\Big\}\leq (T)^{\frac{q-p}{p}}\mathbb{E} \Big\{\int_{0}^{T} 	\|K_G(t,t)\|^{q}_{L^r(\mathbb{R}^d)}dt\Big\}\nonumber\\
			&\leq (T)^{\frac{q-p}{p}}\int_{0}^{T} \mathbb{E} \Big\{\sup_{t_0 \in [0,T]}\|K_G(t_0,t)\|^{q}_{L^r(\mathbb{R}^d)}\Big\}dt\leq (T)^{\frac{q-p}{p}}\int_{0}^{T}\big\{ CT^{q/2} +CT^{q/2}	\|	u^\rho\|^q_{\mathbb{X}_1}\big\}dt \nonumber \\
			&\leq C (T)^{\frac{q}{p}+\frac{q}{2}}\big\{ 1 +	\|	u^\rho\|^q_{\mathbb{X}_1}\big\}.
		\end{align}
		Using the bounds \eqref{kbx1}, \eqref{kbx2}, \eqref{kgx1}, \eqref{kgx2}, we deduce the following estimates:
		\begin{align}
			\|	K_B(u_1)(\cdot)-K_B(u_2)(\cdot)\|_{\mathbb{X}_1}&=	\bigg\|\int_{0}^{\cdot} S(\cdot-s)B(u_1^\rho(s))d\mathcal{W}_1(s)-\int_{0}^{\cdot} S(\cdot-s)B(u_2^\rho(s))d\mathcal{W}_1(s)\bigg\|_{\mathbb{X}_1}\nonumber\\
			&\leq C  T^{\frac{1}{2}}\|B\|_{\mathscr{L}(H,\mathcal{L}_2 (Y_1,H))}	\|u_1^\rho-u_2^\rho\|_{\mathbb{X}_1},\label{dif1}\\
			\|	K_B(u_1)(\cdot)-K_B(u_2)(\cdot)\|_{\mathbb{X}_2}&=	\bigg\|\int_{0}^{\cdot} S(\cdot-s)B(u_1^\rho(s))d\mathcal{W}_1(s)-\int_{0}^{\cdot} S(\cdot-s)B(u_2^\rho(s))d\mathcal{W}_1(s)\bigg\|_{\mathbb{X}_2}\nonumber\\
			&\leq C  T^{\frac{1}{2}+\frac{1}{p}}\|B\|_{\mathscr{L}(H,\mathcal{L}_2 (Y_1,H))}	\|u_1^\rho-u_2^\rho\|_{\mathbb{X}_1},\label{dif2}\\
			\|	K_G(u_1)(\cdot)-K_G(u_2)(\cdot)\|_{\mathbb{X}_1}&=	\bigg\|\int_{0}^{\cdot} S(\cdot-s)G(u_1^\rho(s))d\mathcal{W}_2(s)-\int_{0}^{\cdot} S(\cdot-s)G(u_2^\rho(s))d\mathcal{W}_2(s)\bigg\|_{\mathbb{X}_1}\nonumber\\
			&\leq C  T^{\frac{1}{2}}L_G	\|u_1^\rho-u_2^\rho\|_{\mathbb{X}_1},\label{dif3}
		\end{align}
		and
		\begin{align}\label{dif4}
			\bigg\|	K_G(u_1)(\cdot)-K_G(u_2)(\cdot)\|_{\mathbb{X}_2}&=	\bigg\|\int_{0}^{\cdot} S(\cdot-s)G(u_1^\rho(s))d\mathcal{W}_2(s)-\int_{0}^{\cdot} S(\cdot-s)G(u_2^\rho(s))d\mathcal{W}_2(s)\bigg\|_{\mathbb{X}_2}\nonumber\\
			&\leq C  T^{\frac{1}{2}+\frac{1}{p}}L_G	\|u_1^\rho-u_2^\rho\|_{\mathbb{X}_1}.
		\end{align}
		Now, we estimate the remaining terms (the term involving $b$ and the control terms corresponding to $B, G$) of \eqref{Iomildtrunc} here.
		We have the following estimates for the term related to $b $:
		\begin{align}\label{blinf}
			&\bigg\| \int_{0}^{\cdot} S(\cdot-s) \, b(u^\rho(s)) \, ds\bigg \|_{L^\infty(0,T; H)} =\bigg \| \int_{0}^{\cdot} S(\cdot-s)\frac{1}{2} \sum_{m=1}^{\infty} B_m^2 (u^\rho(s)) \, ds \bigg\|_{L^\infty(0,T; H)}\nonumber\\
			&\leq \sup_{0 \leq t \leq T}\bigg\| \int_{0}^{t} S(t-s)\frac{1}{2} \sum_{m=1}^{\infty}  B_m^2 (u^\rho(s)) \bigg\|_H \, ds\leq \sup_{0 \leq t \leq T} \int_{0}^{t}\Big\| S(t-s) \frac{1}{2}\sum_{m=1}^{\infty} B_m^2 (u^\rho(s)) \Big\|_H \, ds\nonumber\\
			&	\leq  \sup_{0 \leq t \leq T} \sum_{m=1}^{\infty} C\| B_m^2 \|_{\mathscr{L}(H)} \int_{0}^{t} \| u^\rho(s) \|_{H} \leq  \sup_{0 \leq t \leq T} \sum_{m=1}^{\infty} CT\| B_m\|_{\mathscr{L}(H)}^2 \| u^\rho \|_{L^\infty(0,T; H)} \nonumber\\
			&\leq CT \| u^\rho \|_{L^\infty(0,T; H)}.
		\end{align}
		The boundedness of the semigroup and \eqref{blinf} gives,
		\begin{align}\label{bx1}
			&\bigg\| \int_{0}^{\cdot} S(\cdot-s)\, b(u^\rho(s))\, ds \bigg\|^q_{\mathbb{X}_1}=	\bigg\| \int_{0}^{\cdot} S(\cdot-s)\, b(u^\rho(s))\,ds \bigg\|^q_{L^q(\Omega;L^\infty(0,T; H))}\nonumber\\
			& \leq \mathbb{E} \bigg\{	\bigg\| \int_{0}^{\cdot} S(\cdot-s) \,b(u^\rho(s))\,ds \bigg\|^q_{L^\infty(0,T; H)}  \bigg\}\leq \mathbb{E}\Big\{	CT^q \| u^\rho\|^q_{L^\infty(0,T; H)} \Big\}\leq CT^q \|  u^\rho \|^q_{\mathbb{X}_1}.
		\end{align}
		Applying  the Strichartz estimates (\Cref{Strichartz}) with $(p, r)=(\infty,2)$, we get 
		\begin{align*}
			&\bigg\| \int_{0}^{\cdot} S(\cdot-s) \, b(u^\rho(s)) \, ds\bigg \|_{L^p(0,T;L^r(\mathbb{R}^d))} 
			= \bigg\| \int_{0}^{\cdot} S(\cdot-s)\frac{1}{2} \sum_{m=1}^{\infty} B_m^2 (u^\rho(s)) \, ds \bigg\|_{L^p(0,T;L^r(\mathbb{R}^d))}\\
			&\leq C\Big\| \frac{1}{2}\sum_{m=1}^{\infty} B_m^2 (u^\rho)\Big \|_{L^1(0,T;H)} \leq C \int_{0}^{T}\Big\|\frac{1}{2} \sum_{m=1}^{\infty} B_m^2 (u^\rho(s)) \Big\|_H \, ds\\
			&	\leq   \sum_{m=1}^{\infty} C\|B_m^2\|_{\mathscr{L}(H)} \int_{0}^{T} \| u^\rho(s) \|_{H}ds \leq   \sum_{m=1}^{\infty}  CT\| B_m \|_{\mathscr{L}(H)}^2 \| u^\rho \|_{L^\infty(0,T; H)} \\
			&\leq CT \| u^\rho \|_{L^\infty(0,T; H)}.
		\end{align*}
		This implies,
		\begin{align}\label{bx2}
			&\bigg\| \int_{0}^{\cdot} S(\cdot-s) \, b(u^\rho(s)) \, ds \bigg\|^q_{\mathbb{X}_2}\leq \frac{1}{2}CT^q \|  u^\rho \|^q_{\mathbb{X}_1}.
		\end{align}
		Presently, we estimate the control terms involving $B$ and $G$ as	
		\begin{align}\label{boundbgx}
			&\bigg\|	\int_{0}^{\cdot} S(\cdot-s)\big[B(u^\rho(s))\rho_1(s)+ G(u^\rho(s))\rho_2(s)  \big]ds\bigg\|_{\mathbb{X}}\nonumber\\
			& \leq  C \|u^\rho\|_{\mathbb{X}} \int_{0}^{T} \big\{\|B\|_{\mathscr{L}(H,\mathcal{L}_2 (Y_1,H))} \|\rho_1(s)\|_{Y_1}+ C_2 \|\rho_2(s)\|_{Y_2} \big\}ds+ C C_1 \int_{0}^{T}\|\rho_2(s)\|_{Y_2}ds.
		\end{align}
		Combining all the estimates \eqref{kbx1}, \eqref{kbx2}, \eqref{kgx1}, \eqref{kgx2}-\eqref{dif4}, \eqref{bx1}-\eqref{boundbgx} and \Cref{liptrunc}, from \eqref{Iomildtrunc}, we conclude ($\varepsilon<1$ is taken for our convenience),
		\begin{align}\label{Ioxb}
			&\|	\mathfrak{L}(u, u_0, \rho)(t)\|_{\mathbb{X}}\nonumber\\
			&\leq  C\|u_0\|_{H} + (1+C)\beta T \|  u^\rho\|_{\mathbb{X}}+C_R (2R)^\alpha T^{\frac{p-r}{p}} + C C_1 \int_{0}^{T}\|\rho_2(s)\|_{Y_2}ds\nonumber\\
			&\quad+CT \|  u^\rho\|_{\mathbb{X}}+ C \|u^\rho\|_{\mathbb{X}} \int_{0}^{T} \big\{\|B\|_{\mathscr{L}(H,\mathcal{L}_2 (Y_1,H))} \|\rho_1(s)\|_{Y_1}+ C_2 \|\rho_2(s)\|_{Y_2} \big\}ds\nonumber\\ 
			&\quad+(T)^{\frac{1}{p}+\frac{1}{2}}\{ C \|B\|_{\mathscr{L}(H,\mathcal{L}_2 (Y_1,H))}	\|u^\rho\|_{\mathbb{X}}\}  +C T^{\frac{1}{p}+\frac{1}{2}}\{ 1 +	\|	u^\rho\|_{\mathbb{X}}\},
		\end{align}
		and 
		\begin{align}\label{Iodifx}
			&\|	\mathfrak{L}(u_1, u_0, \rho)-	\mathfrak{L}(u_2, u_0, \rho)\|_{\mathbb{X}}\nonumber\\
			&\leq   \{(1+C)\beta T + C (T)^{\frac{p-r}{p}} (2M)^{\alpha -1}\}\|u_1^\rho-u^\rho_2\|_{\mathbb{X}}+ CT \|u_1^\rho-u^\rho_2\|_{\mathbb{X}} \nonumber\\
			&\quad + C \|u_1^\rho-u^\rho_2\|_{\mathbb{X}}\int_{0}^{T}\Big[\|B\|_{\mathscr{L}(H,\mathcal{L}_2 (Y_1,H))} \|\rho_1(s)\|_{Y_1}+ L_G \|\rho_2(s)\|_{Y_2}\Big]ds \nonumber \\
			&\quad+C ( T^{\frac{1}{2}+\frac{1}{p}}+ T^\frac{1}{2})\|B\|_{\mathscr{L}(H,\mathcal{L}_2 (Y_1,H))}	\|u_1^\rho-u_2^\rho\|_{\mathbb{X}}+  C  (T^{\frac{1}{2}+\frac{1}{p}}+ T^\frac{1}{2})L_G	\|u_1^\rho-u_2^\rho\|_{\mathbb{X}}.
		\end{align}
		Now, by the absolute continuity of Lebesgue integration (\Cref{abscts}) and similar arguments we used in case of the skeleton equation to get \eqref{bm} and \eqref{lipm}, we have the existence of a small time $T_0$ such that for all $u \in \mathbb{X}_0=L^q(\Omega; L^\infty(0,T_0; H) \cap L^p(0, T_0; L^r(\mathbb{R}^d)))$,
		\begin{align*}
			\|	\mathfrak{L}(u, u_0, \rho)(t)\|_{\mathbb{X}_0}
			&\leq  M,
		\end{align*}
		and for any $u_1, u_2 \in \mathbb{X}_0$,
		\begin{align}\label{contractionmathfracI}
			\|	\mathfrak{L}(u_1, u_0, \rho)-	\mathfrak{L}(u_2, u_0, \rho)\|_{\mathbb{X}_0}	\leq \frac{1}{2} \|u_1^\rho-u^\rho_2\|_{\mathbb{X}_0}.
		\end{align}
		Therefore $\mathfrak{L}(\cdot, u_0, \rho)$ is an $\frac{1}{2}$-contraction on $\mathbb{X}_0$, where $T_0$ depends on $\alpha, \beta, M, \|u_0\|_{H}$.
		Thus by the Banach fixed point theorem, $\mathfrak{L}(\cdot, u_0, \rho)$ has a unique fixed point $\overline{u}\in \mathbb{X}_0$. So $\overline{u}$ is the unique mild solution to the stochastic controlled equation in $[0,T_0]$.
		\subsubsection{Global existence of the solution of truncated equation:}\label{truncatedglobalexistence}
		Here, we extend the local existence of solution to global solution in $[0,T]$ by mathematical induction.\\
		Choose $j=[\frac{T}{T_0}]+1$ and assume for some $k \in \{1,2,\ldots,j\}$ there exists $u_k^R \in \mathbb{X}_{kT_0}$ such that
		\begin{align*}
			u^R_k=	\mathfrak{L}(u^R_k) \quad\text{on}\quad [0,kT_0].
		\end{align*}
		We define a new cut off function by,
		$$\Theta^R_k(u)(t):=\theta_R(\phi(u)(t)),$$
		where, 
		\begin{align}\label{definedphi}
			\phi(u)(t)&:= \Big(\|u^R_k\|^p_{L^p(0,kT_0;L^r(\mathbb{R}^d))}+\|u\|^p_{L^p(0,t;L^r(\mathbb{R}^d))}\Big)^{\frac{1}{p}}\nonumber\\
			&\quad + \max\Big\{\sup_{0 \leq t \leq kT_0}\|u^R_k(t)\|_{H}, \sup_{0 \leq s \leq t}\|u(s)\|_{H}\Big\}.
		\end{align}
		Before proceeding, we establish that $\phi$ satisfies a Lipschitz condition, which will be used later.
		\begin{prop}\label{philip}
			For $u_1, u_2 \in \mathbb{L}^{\infty, p}_{2, r}(T)$ we have $|\phi(u_1)(t)-\phi(u_2)(t)| \leq C \|u_1 - u_2\|_{\mathbb{L}^{\infty, p}_{2, r}(T)},$ for all $t \in [0,T].$
		\end{prop}
		\begin{proof}
			Consider $\phi(u)$ as a combination of $\phi_1(u)$ and $\phi_2(u)$ as follows:
			\begin{align}\label{defphi}
				\phi(u)(t)= \phi_1(u)(t)+\phi_2(u)(t), \quad\text{ for all } t \in [0,T],
			\end{align}
			where
			\begin{align*}
				\phi_1(u)(t)&:= \Big(\|u^R_k\|^p_{L^p(0,kT_0;L^r(\mathbb{R}^d))}+\|u\|^p_{L^p(0,t;L^r(\mathbb{R}^d))}\Big)^{\frac{1}{p}},\nonumber\\
				\phi_2(u)(t)&:= \max\Big\{\sup_{0 \leq t \leq kT_0}\|(u)^R_k(t)\|_{H}, \sup_{0 \leq s \leq t}\|(u)(s)\|_{H}\Big\}.
			\end{align*}
			Since, any norm on a normed vector space is uniformly Lipschitz continuous and $u_k^R$ is the unique solution of \eqref{truncated} in $\mathbb{X}_{kT_0}$, we can easily check that $\phi_1(u)$ is Lipschitz in $L^p(0,T;L^r(\mathbb{R}^d))$ and $\phi_2(u)$ is Lipschitz in $C([0,T];H)$.
			Therefore, construction of $\phi$ in \eqref{defphi} and definition of norm in $\mathbb{L}^{\infty, p}_{2, r}(T)$ \eqref{normassum} lead to the conclusion that $\phi$ is Lipschitz in $\mathbb{L}^{\infty, p}_{2, r}(T)$ that is, 
			\begin{align*}
				|\phi(u_1)(t)-\phi(u_2)(t)| \leq C \|u_1 - u_2\|_{\mathbb{L}^{\infty, p}_{2, r}(T)},\quad\text{ for all } t \in [0,T].
			\end{align*}
		\end{proof}
		\noindent Now, we define the following operator $	\mathfrak{L}_k$ for $t \in[0,T_0]$ and $u \in \mathbb{X}_{kT_0,(k+1)T_0}$ as, 
		\begin{align*}
			\mathfrak{L}_k(u)(t)&:=S(t)u^R_k(kT_0)	- i\int_{0}^{t} S(t-s)	\Theta^R_k(u)(s) \mathcal{N}(u^\rho(s))ds\nonumber\\
			&\quad- \int_{0}^{t} S(t-s) \big[\beta u^\rho(s) + \varepsilon b(u^\rho(t))\big]ds\nonumber\\ 
			& \quad- i\int_{0}^{t} S(t-s)\big[B(u^\rho(s))\rho_1(s)+G(u^\rho(s))\rho_2(s)\big ]ds\nonumber\\ 
			&\quad- i \sqrt{\varepsilon}\int_{0}^{t} S(t-s)\big[B(u^\rho(s))d\mathcal{W}_1(s)+G(u^\rho(s))d\mathcal{W}_2(s) \big].
		\end{align*}
		Our aim is to prove that, $\mathfrak{L}_k$ is a contraction map and has a unique fixed point. Before proving that, we establish the following proposition for the term involving the nonlinearity which will be used later.
		\begin{prop}\label{lipjkr}
			Assume $1< \alpha< \frac{4}{d}+1, r=\alpha+1$. Denote 
			$$\mathfrak{L}_k^R(u^\rho(t)):=i\int_{0}^{t} S(t-s)\Theta^R_k(u)(s) \mathcal{N}(u^\rho(s))ds,$$
			Then, for every $T>0$, the function $\mathfrak{L}_k^R$ maps $\mathbb{L}^{\infty, p}_{2, r}(T)$ to itself and for $u_1, u_2 \in \mathbb{L}^{\infty, p}_{2, r}(T)$, we have 
			\begin{align*}
				\|\mathfrak{L}_k^R(u_1)-\mathfrak{L}_k^R(u_2)\|_{\mathbb{L}^{\infty, p}_{2, r}(T)} \leq C T^{\frac{p-r}{p}}\|u_1-u_2\|_{\mathbb{L}^{\infty, p}_{2, r}(T)}.
			\end{align*}
		\end{prop}
		\begin{proof}
			The proof is straightforward consequence of \Cref{liptrunc} and Lipschitz continuity of $\phi(u)(\cdot)$, \Cref{philip}. 
		\end{proof}
		Repeating all the arguments, we employed to get \eqref{contractionmathfracI} and using \Cref{lipjkr}, we have
		\begin{align}\label{contractionmathfracIkto}
			\|	\mathfrak{L}(u_1, u_0, \rho)-	\mathfrak{L}(u_2, u_0, \rho)\|_{\mathbb{X}_{kT_0,(k+1)T_0}}	\leq \frac{1}{2} \|u_1^\rho-u^\rho_2\|_{\mathbb{X}_{kT_0,(k+1)T_0}}.
		\end{align}
		Therefore, $\mathfrak{L}_k$ is a $\frac{1}{2}$- contraction in the space $\mathbb{X}_{kT_0,(k+1)T_0}$. Let ${\overline{u}}^R_{k+1}$ be the unique fixed point satisfying ${\overline{u}}^R_{k+1}=\mathfrak{L}_k({\overline{u}}^R_{k+1} )$. Then we construct a solution as follows:
		\begin{align*}
			u^R_{k+1}(t)=\left\{
			\begin{aligned}
				&u^R_k(t), \,\quad &&\text{for}\,t\in [0,kT_0],\\
				& {\overline{u}}^R_{k+1}(t-kT_0) , \quad &&\text{for}\,t \in [kT_0,(k+1)T_0].
			\end{aligned}
			\right. 
		\end{align*}
		It is easy to check, that, $	u^R_{k+1} \in \mathbb{X}_{(k+1)T_0}$.
		Now, we will show that $	u^R_{k+1}$ is a fixed point of $\mathfrak{L}_k$ in $\mathbb{X}_{(k+1)T_0}$.
		Let $t\in [kT_0,(k+1)T_0]$ and define $\hat t=t-kT_0$. Then, we have
		\begin{align*}
			u^R_{k+1}(t)&= {\overline{u}}^R_{k+1}(\hat t)=\mathfrak{L}_k({\overline{u}}^R_{k+1})(\hat t)\\
			&= S(\hat t)u^R_k(kT_0)	- i\int_{0}^{\hat t} S(\hat t-s)	\Theta^R_k({\overline{u}}^R_{k+1})(s) \mathcal{N}({\overline{u}}^R_{k+1}(s))ds\nonumber\\
			&\quad- \int_{0}^{\hat t} S(\hat t-s) \Big[\beta {\overline{u}}^R_{k+1}(s) + \varepsilon b({\overline{u}}^R_{k+1}(s))\Big]ds - i\int_{0}^{\hat t} S(\hat t-s)\Big[B({\overline{u}}^R_{k+1}(s))\rho_1(s)+G({\overline{u}}^R_{k+1}(s))\rho_2(s) \Big]ds\nonumber\\ 
			&\quad- i \sqrt{\varepsilon}\int_{0}^{\hat t} S(\hat t-s)\Big[B({\overline{u}}^R_{k+1}(s))d\mathcal{W}_1(s)+G({\overline{u}}^R_{k+1}(s))d\mathcal{W}_2(s) \Big]\\
			&= S(\hat t) \bigg\{S(kT_0)u_0	- i\int_{0}^{kT_0} S(kT_0-s)\theta_R(	\|u^R_k\|_{\mathbb{L}^{\infty, p}_{2, r}(s)}) \mathcal{N}(u^R_k(s))ds\nonumber\\
			&\quad- \int_{0}^{kT_0} S(kT_0-s)\Big[\beta u^R_k(s) + \varepsilon b(u^R_k(s))\Big]ds\nonumber\\ 
			&\quad - i\int_{0}^{kT_0} S(kT_0-s)\Big[B(u^R_k(s))\rho_1(s)+G(u^R_k(s))\rho_2(s) \Big]ds\nonumber\\ 
			&\quad- i \sqrt{\varepsilon}\int_{0}^{kT_0} S(kT_0-s)\Big[B(u^R_k(s))d\mathcal{W}_1(s)+G(u^R_k(s))d\mathcal{W}_2(s) \Big]\bigg\}\\
			&\quad- i\int_{0}^{\hat t} S(\hat t-s)	\Theta^R_k({\overline{u}}^R_{k+1})(s) \mathcal{N}({\overline{u}}^R_{k+1}(s))ds- \int_{0}^{\hat t} S(\hat t-s) \Big[\beta {\overline{u}}^R_{k+1}(s) + \varepsilon b({\overline{u}}^R_{k+1}(s))\Big]ds\nonumber\\ 
			& \quad- i\int_{0}^{\hat t} S(\hat t-s)\Big[B({\overline{u}}^R_{k+1}(s))\rho_1(s)+G({\overline{u}}^R_{k+1}(s))\rho_2(s) \Big]ds\nonumber\\ 
			&\quad- i \sqrt{\varepsilon}\int_{0}^{\hat t} S(\hat t-s)\Big[B({\overline{u}}^R_{k+1}(s))d\mathcal{W}_1(s)+G({\overline{u}}^R_{k+1}(s))d\mathcal{W}_2(s)\Big].
		\end{align*}
		We can easily check that $\theta_R(	\|u^R_k\|_{\mathbb{L}^{\infty, p}_{2, r}(s)})=\theta_R(	\|u^R_{k+1}\|_{\mathbb{L}^{\infty, p}_{2, r}(s)})$ for $s \in [0,kT_0]$ and \\ $\Theta^R_k({\overline{u}}^R_{k+1})(s)=\theta_R(	\|u^R_{k+1}\|_{\mathbb{L}^{\infty, p}_{2, r}(kT_0+s)})$ for $s \in [0,T_0]$. This implies that,
		\begin{align*}
			u^R_{k+1}(t)&= S(t)u_0	- i\int_{0}^{kT_0} S(t-s)\theta_R(	\|u^R_k\|_{\mathbb{L}^{\infty, p}_{2, r}(s)}) \mathcal{N}(u^R_k(s))ds\nonumber\\
			&\quad- \int_{0}^{kT_0} S(t-s) \Big[\beta u^R_k(s) + \varepsilon b(u^R_k(s))\Big]ds- i\int_{0}^{kT_0} S(t-s)\Big[B(u^R_k(s))\rho_1(s)+G(u^R_k(s))\rho_2(s) \Big]ds\nonumber\\ 
			&\quad- i \sqrt{\varepsilon}\int_{0}^{kT_0} S(t-s)\Big[B(u^R_k(s))d\mathcal{W}_1(s)+G(u^R_k(s))d\mathcal{W}_2(s) \Big]\\
			&\quad- i\int_{0}^{\hat t} S(\hat t-s)\theta_R(	\|u^R_{k+1}\|_{\mathbb{L}^{\infty, p}_{2, r}(kT_0+s)}) \mathcal{N}(u^R_{k+1}(kT_0+s))ds\nonumber\\
			&\quad- \int_{0}^{\hat t} S(\hat t-s) \Big[\beta u^R_{k+1}(kT_0+s) + \varepsilon b(u^R_{k+1}(kT_0+s))\Big]ds\nonumber\\ 
			&\quad - i\int_{0}^{\hat t} S(\hat t-s)\Big[Bu^R_{k+1}(kT_0+s)\rho_1(s)+G(u^R_{k+1}(kT_0+s))\rho_2(s) \Big]ds\nonumber\\ 
			&\quad- i \sqrt{\varepsilon}\int_{0}^{\hat t} S(\hat t-s)\Big[Bu^R_{k+1}(kT_0+s)d\mathcal{W}_1(s)+G(u^R_{k+1}(kT_0+s))d\mathcal{W}_2(s) \Big].
		\end{align*}
		We use the transformation $s=kT_0+s$ and observe that
		\begin{align*}
			u^R_{k+1}(t)&= S(t)u_0	- i\int_{0}^{kT_0} S(t-s)\theta_R(	\|u^R_k(s)\|_{\mathbb{L}^{\infty, p}_{2, r}(s)}) \mathcal{N}(u^R_k(s))ds\nonumber\\
			&\quad- \int_{0}^{kT_0} S(t-s) \Big[\beta u^R_k(s) + \varepsilon b(u^R_k(s))\Big]ds\nonumber\\ 
			&\quad - i\int_{0}^{kT_0} S(t-s)\Big[B(u^R_k(s))\rho_1(s)+G(u^R_k(s))\rho_2(s) \Big]ds\nonumber\\ 
			&\quad- i \sqrt{\varepsilon}\int_{0}^{kT_0} S(t-s)\Big[B(u^R_k(s))d\mathcal{W}_1(s)+G(u^R_k(s))d\mathcal{W}_2(s) \Big]\\
			&\quad- i\int_{kT_0}^{t} S(t-s)\theta_R(	\|u^R_{k+1}\|_{\mathbb{L}^{\infty, p}_{2, r}(s)}) \mathcal{N}(u^R_{k+1}(s))ds\nonumber\\
			&\quad- \int_{kT_0}^{t} S(t-s) \Big[\beta u^R_{k+1}(s) + \varepsilon b(u^R_{k+1}(s))\Big]ds\nonumber\\ 
			&\quad - i\int_{kT_0}^{t}S(t-s)\Big[B(u^R_{k+1}(s))\rho_1(s)+G(u^R_{k+1}(s))\rho_2(s) \Big]ds\nonumber\\ 
			&\quad- i \sqrt{\varepsilon}\int_{kT_0}^{t} S(t-s)\Big[B(u^R_{k+1}(s))d\mathcal{W}_1(s)+G(u^R_{k+1}(s))d\mathcal{W}_2(s) \Big].
		\end{align*}
		Consequently, we obtain
		\begin{align*}
			u^R_{k+1}(t)&=S(t)u_0- i\int_{0}^{t} S(t-s)\theta_R(	\|u^R_{k+1}\|_{\mathbb{L}^{\infty, p}_{2, r}(s)}) \mathcal{N}(u^R_{k+1}(s))ds\nonumber\\
			&\quad- \int_{0}^{t} S(t-s) \Big[\beta u^R_{k+1}(s) + \varepsilon b(u^R_{k+1}(s))\Big]ds\nonumber\\ 
			&\quad - i\int_{0}^{t}S(t-s)\Big[B(u^R_{k+1}(s))\rho_1(s)+G(u^R_{k+1}(s))\rho_2(s) \Big]ds\nonumber\\ 
			&\quad- i \sqrt{\varepsilon}\int_{0}^{t} S(t-s)\Big[B(u^R_{k+1}(s))d\mathcal{W}_1(s)+G(u^R_{k+1}(s))d\mathcal{W}_2(s) \Big].
		\end{align*}
		which shows that $u^R_{k+1}(\cdot)$ is a fixed point of $\mathfrak{L}_k$ in $\mathbb{X}_{(k+1)T_0}$.
		Therefore $u^R=u_j^R$ is the unique solution to the truncated equation \eqref{truncated} on $[0,T]$.
	\end{proof}
	\subsection{Local solution for the original equation}\label{locoriginal}
	In this subsection, we shall prove the existence of the unique local solution for the original equation \eqref{Sce}. To do that, we use the existence of the unique global solution of the truncated equation established in \Cref{truncatedglobalexistence}.
	\begin{prop}\label{localoriginal}
		For each $k$, let $u_k \in \mathbb{X}_T$ be the solution of the truncated equation \eqref{truncated} with $R$ replaced by $k$. Define a stopping time $\tau_k$ by 
		\begin{align}\label{stoppingtime}
			\tau_k:=\inf \{t\geq 0: 	\|u_k\|_{\mathbb{L}^{\infty, p}_{2, r}(t)}> k\} \land T,
		\end{align}
		with the initial convention $\inf \varnothing=T.$ Then the following statements hold:
		\begin{itemize}
			\item[(1)] For $k\leq n$, we have $0\leq \tau_k \leq \tau_n, \mathbb{P}$- a.s. and $u_k(t)=u_n(t), \mathbb{P}$- a.s. for $t\in [0,\tau_k]$.
			\item[(2)]Define $u_k(t)=u(t), $ for $t\in [0,\tau_k]$ and $\tau_\infty= \lim_{k \to \infty} \tau_k$. Then $(u(t))_{t\in [0, \tau_{\infty})}$ is a maximal local mild solution of the original equation \eqref{Sce}.
			\item[(3)] The solution is unique.
		\end{itemize}
	\end{prop}
	\begin{proof}
		For any $n \in \mathbb{N}$ and $t \in [0,T]$, the global solution $u_n \in \mathbb{X}_T$ to \eqref{truncated}, satisfies
		\begin{align}\label{ntruncated}
			u_n(t)&=S(t)u_0	- i\int_{0}^{t} S(t-s)\theta_R(	\|u_n\|_{\mathbb{L}^{\infty, p}_{2, r}(s)}) \mathcal{N}(u_n(s))ds- \int_{0}^{t} S(t-s) \big[\beta u_n(s) + \varepsilon b(u_n(s))\big]ds\nonumber\\ 
			&\quad - i\int_{0}^{t} S(t-s)\big[B(u_n(s))\rho_1(s)+G(u_n(s))\rho_2(s) \big]ds\nonumber\\ 
			&\quad- i \sqrt{\varepsilon}\int_{0}^{t} S(t-s)\big[B(u_n(s))d\mathcal{W}_1(s)+G(u_n(s))d\mathcal{W}_2(s) \big], \quad \mathbb{P}\,\text{-a.s. for } t\in [0,T].
		\end{align}
		For $k\leq n$, denote $\tau_k \land \tau_n=\tau_{k,n}$. The definition of $\theta_n$ implies $$\theta_n(\|u_n\|_{\mathbb{L}^{\infty, p}_{2, r}(t)})=1 \text{ and } \theta_k(\|u_k\|_{\mathbb{L}^{\infty, p}_{2, r}(t)})=1 \text{ for } t\in [0,\tau_{k,n}).$$ Therefore, for $l=k,n$, we have
		\begin{align}\label{ltruncated}
			u_l(t)&=S(t)u_0	- i\int_{0}^{t} S(t-s)\theta_R(	\|u_l\|_{\mathbb{L}^{\infty, p}_{2, r}(s)}) \mathcal{N}(u_l(s))ds- \int_{0}^{t} S(t-s) \big[\beta u_l(s) + \varepsilon b(u_l(s))\big]ds\nonumber\\ 
			&\quad - i\int_{0}^{t} S(t-s)\big[B(u_l(s))\rho_1(s)+G(u_l(s))\rho_2(s) \big]ds\nonumber\\ 
			&\quad- i \sqrt{\varepsilon}\int_{0}^{t} S(t-s)\big[B(u_l(s))d\mathcal{W}_1(s)+G(u_l(s))d\mathcal{W}_2(s) \big], \quad \mathbb{P}\,\text{-a.s. for } t\in [0,\tau_{k,n}).
		\end{align}
		The uniqueness of the solution follows from the uniqueness of the solution of \eqref{truncated} i.e. $u_k(t)=u_n(t), \mathbb{P}$- a.s. for $t\in [0,\tau_{k,n})$. Denote the right hand side of \eqref{ltruncated} by $\Phi(u_l)$. Notice that the value of $\Phi(u_l)$ at $\tau_{k,n}$ depends only on the values of $u_l$ on $[0,\tau_{k,n}) $. Therefore, the process $u_l$ can be extended from  $[0,\tau_{k,n}) $ to  $[0,\tau_{k,n}] $ by taking,
		\begin{align}\label{taukn}
			u_l({\tau_{k,n}})&=S({\tau_{k,n}})u_0	- i\int_{0}^{{\tau_{k,n}}} S({\tau_{k,n}}-s)\theta_R(	\|u_l\|_{\mathbb{L}^{\infty, p}_{2, r}(s)}) \mathcal{N}(u_l(s))ds\nonumber\\
			&\quad- \int_{0}^{{\tau_{k,n}}} S({\tau_{k,n}}-s) \big[\beta u_l(s) + \varepsilon b(u_l(s))\big]ds\nonumber\\ 
			& \quad- i\int_{0}^{{\tau_{k,n}}} S({\tau_{k,n}}-s)\big[B(u_l(s))\rho_1(s)+G(u_l(s))\rho_2(s) \big]ds\nonumber\\ 
			&\quad- i \sqrt{\varepsilon}\int_{0}^{{\tau_{k,n}}} S({\tau_{k,n}}-s)\big[B(u_l(s))d\mathcal{W}_1(s)+G(u_l(s))d\mathcal{W}_2(s) \big],\quad \mathbb{P}\,\text{-a.s.}
		\end{align}
		By combining the above two equations \eqref{ltruncated} and \eqref{taukn}, we conclude the stopped process $u_l(t \land \tau_{n})$ satisfies the following equation,
		\begin{align}
			u_l({t \land \tau_{k,n}})&=S({t \land \tau_{k,n}})u_0	- i\int_{0}^{{t \land \tau_{k,n}}} S({t \land \tau_{k,n}}-s)\theta_R(	\|u_l\|_{\mathbb{L}^{\infty, p}_{2, r}(s)}) \mathcal{N}(u_l(s))ds\nonumber\\
			&\quad- \int_{0}^{{t \land \tau_{k,n}}} S({t \land \tau_{k,n}}-s) \big[\beta u_l(s) + \varepsilon b(u_l(s))\big]ds\nonumber\\ 
			& \quad- i\int_{0}^{{t \land \tau_{k,n}}} S({t \land \tau_{k,n}}-s)\big[B(u_l(s))\rho_1(s)+G(u_l(s))\rho_2(s) \big]ds\nonumber\\ 
			&\quad- i \sqrt{\varepsilon}\int_{0}^{{t \land \tau_{k,n}}} S({t \land \tau_{k,n}}-s)\big[B(u_l(s))d\mathcal{W}_1(s)+G(u_l(s))d\mathcal{W}_2(s) \big],\quad \mathbb{P}\,\text{-a.s. for } t\in [0,T].
		\end{align}
		This asserts that, 
		\begin{align*}
			u_k(t)=u_n(t) \quad \text{for } t\in [0,\tau_{k,n}].
		\end{align*}
		From the definition of stopping time \eqref{stoppingtime},
		we have $0\leq \tau_k \leq \tau_n, \mathbb{P} $\,-a.s. for $k\leq n$, and $u_k(t)=u_n(t), \mathbb{P}$\,- a.s. for $t\in [0,\tau_k].$ This proves the first part of the \Cref{localoriginal}.\\
		\indent 
		Since sequence $(\tau_n)_{n\in \mathbb{N}}$ is increasing and bounded by $T$, the limit $\lim_{n \to \infty} \tau_n:=\tau_{\infty}$ exists $\mathbb{P}$\,-a.s. Let us denote $\Omega_0=\{\omega: \lim_{n \to \infty}\tau_n = \tau_{\infty}\}$ and note that $\mathbb{P}(\Omega_0)=1$. Now, we define a local process $(u(t))_{0\leq t <\tau_{\infty}}$ as follows: if $\omega \notin \Omega_0$ set $u(t,\omega)=0$ for $0\leq t <\tau_{\infty}$. And for $\omega \in \Omega_0$ for every $ t <\tau_{\infty}$, there exists a natural number $n\in \mathbb{N}$ such that $t \leq\tau_n(\omega)$ and we set $u(t,\omega)=u_n(t,\omega)$. We already have, $(u(t))_{[0,\tau_n]} \in \mathbb{X}_{\tau_n}$ and it satisfies
		\begin{align*}
			u({t \land \tau_{n}})&=S({t \land \tau_{n}})u_0	- i\int_{0}^{{t \land \tau_{n}}} S({t \land \tau_{n}}-s)\theta_R(	\|u\|_{\mathbb{L}^{\infty, p}_{2, r}(s)}) \mathcal{N}(u(s))ds\nonumber\\
			&\quad- \int_{0}^{{t \land \tau_{n}}} S({t \land \tau_{n}}-s) \big[\beta u(s) + \varepsilon b(u(s))\big]ds\nonumber\\ 
			& \quad- i\int_{0}^{{t \land \tau_{n}}} S({t \land \tau_{n}}-s)\big[B(u(s))\rho_1(s)+G(u(s))\rho_2(s) \big]ds\nonumber\\ 
			&\quad- i \sqrt{\varepsilon}\int_{0}^{{t \land \tau_{n}}} S({t \land \tau_{n}}-s)\big[B(u(s))d\mathcal{W}_1(s)+G(u(s))d\mathcal{W}_2(s) \big],\quad \mathbb{P}\,\text{-a.s. for } t\in [0,T].
		\end{align*}
		Furthermore, by the definition of $(\tau_n)_{n\in \mathbb{N}}$, we deduce that, on the set $\{\tau_{\infty} < \infty\}$,
		\begin{align*}
			\lim_{t\to \tau_{\infty}} \|u\|_{\mathbb{L}^{\infty, p}_{2, r}(t)}= \lim_{n \to \infty}\|u\|_{\mathbb{L}^{\infty, p}_{2, r}(\tau_n)} \geq \lim_{n \to \infty} n =\infty,\quad \mathbb{P}\,\text{-a.s.}
		\end{align*}
		Therefore $(u(t))_{0\leq t <\tau_{\infty}}$ is a maximal local mild solution of the original equation \eqref{Sce}.\\
		\indent The construction of the solution and the uniqueness of the solution to \eqref{truncated} implies the uniqueness of this maximal local mild solution of \eqref{Sce}. This completes the proof of \Cref{localoriginal}.
	\end{proof}
	\subsection{Global existence of solution of the original equation}\label{globoriginal}
	In this subsection, we prove the global existence of a unique mild solution of the original equation \eqref{Sce}. Before moving to the proof of \Cref{wellsce}, we establish the following lemma to get the uniform bounds for the solutions of \eqref{truncated} in $L^q(\Omega; L^\infty(0,T; H))$.
	\begin{lem}\label{lemuklinfbound}
		Let $r=\alpha+1$ and $1< \alpha< \frac{4}{d}+1, r=\alpha+1$, choose a p such that (p,r) is an admissible pair and $u_0\in H$. Then the solutions of \eqref{truncated} with $R$ replaced by $k$ is uniformly (uniform in $k$) bounded in $L^q(\Omega; L^\infty(0,T; H))$, that is, there exists a constant (independent of $k$) such that
		$$	\|u^k\|_{L^q(\Omega; L^\infty(0,T; H))}\leq C.$$
	\end{lem}
	\begin{proof}
		To prove the lemma, we make use of the Yosida approximation operator. Consider the following Yosida approximated form of the truncated equation \eqref{truncated}:	
		\begin{equation}\label{approxgs}
			\left\{
			\begin{aligned}
				du_\mu^k(t)&=-\big[i Au_\mu^k(t) +i \mathcal{N}_\mu(u_\mu^k(t))+ \beta u_\mu^k(t)+\varepsilon b(u_\mu^k(t))\big]dt - i B( u_\mu^k(t))\rho_1(t) dt \\
				&\quad - i   G_\mu(u_\mu^k(t))\rho_2(t)dt- i\sqrt{\varepsilon}  B(u_\mu^k(t)) d\mathcal{W}_1(t)- i\sqrt{\varepsilon}  G_\mu(u_\mu^k(t)) d\mathcal{W}_2(t), \quad t>0,	\\
				u_\mu^k(0)&=J_\mu u_0,
			\end{aligned}
			\right.
		\end{equation}
		where $\mathcal{N}_\mu(u_\mu^k)=J_\mu \theta_k(\|u_\mu^k\|_{\mathbb{L}^{\infty, p}_{2, r}(t)})\mathcal{N}( u_\mu^k),\, G_\mu( u_\mu^k)= J_\mu G( u_\mu^k).$\\
		Since the truncated equation is globally well-posed, for $u\in \mathbb{X}$, if we replace $\mathfrak{L}$ by $\mathfrak{L}_\mu$ in \eqref{Iomildtrunc}, by the Yosida approximation method, we have the existence of unique global solution $u^k_\mu$ to \eqref{approxgs} satisfying
		\begin{align*}
			\mathfrak{L}_\mu(u, u_0, k)(t)&=S(t)J_\mu u_0	- \int_{0}^{t} S(t-s)\big[i \mathcal{N}_\mu(u_\mu^k(s))+\beta u_\mu^k(s)+\varepsilon b(u_\mu^k(t))\big]ds \\
			&\quad- i\int_{0}^{t} S(t-s)\big[B(u_\mu^k(s))\rho_1(s)+ G_\mu(u_\mu^k(t))\rho_2(t)\big]ds \\
			&\quad- i\sqrt{\varepsilon} \int_{0}^{t} S(t-s)\big[ B(u_\mu^k(s)) d\mathcal{W}_1(s)+G_\mu(u_\mu^k(s)) d\mathcal{W}_2(s)\big]\quad \text{ for } t\in[0,T].
		\end{align*}
		Assume \(	u^k=	(u^k(t))_{t \in [0,T]} \in \mathbb{X}  \) is the unique solution to the truncated stochastic controlled equation on $[0,T]$. Therefore, for all $t\in [0,T]$, it satisfies the mild form,
		\begin{align*}
			u^k(t)&=S(t)u_0	- i\int_{0}^{t} S(t-s)\theta_k(	\|u^k\|_{\mathbb{L}^{\infty, p}_{2, r}(s)}) \mathcal{N}(u^k(s))ds\nonumber\\
			&\quad- \int_{0}^{t} S(t-s) \big[\beta u^k(s) + \varepsilon b(u^k(t))\big]ds\nonumber\\ 
			&\quad - i\int_{0}^{t} S(t-s)\big[B(u^k(s))\rho_1(s)+G(u^k(s))\rho_2(s) \big]ds\nonumber\\ 
			&\quad- i \sqrt{\varepsilon}\int_{0}^{t} S(t-s)\big[B(u^k(s))d\mathcal{W}_1(s)+G(u^k(s))d\mathcal{W}_2(s) \big]. 
		\end{align*} 
		By the properties of the Yosida approximation, we can easily check that, 
		\begin{align}\label{limitpasslim}
			\lim_{\mu \to \infty} \mathbb{E} \Big\{\sup_{t \in [0,{T}]}\|u_\mu^k(t)-u^k(t)\|^2_{H}\Big\}=0.
		\end{align}
		Consider $\psi:H \to \mathbb{R}$ by $\psi(u)=\|u\|^2_{H}$. Since we have more regularity on spatial variable, we apply the infinite dimensional It\^o formula (Page-177, \cite{MR1706888}) to the function $\psi$ and $u^k_\mu$, to obtain
		\begin{align}\label{realpart}
			\frac{1}{2}\|u^k_\mu(t)\|^2_{H} &=\frac{1}{2}\|J_\mu u_0\|^2_{H} -\int_{0}^{t} \operatorname{Re}\big\langle i Au_\mu^k(s) +i \mathcal{N}_\mu(u_\mu^k(s))+ \beta u_\mu^k(s)+\varepsilon b(u_\mu^k(s)), u_\mu^k(s)\big\rangle ds \nonumber\\
			& \quad- i\int_{0}^{t} \operatorname{Re}\big\langle B(u_\mu^k(s))\rho_1(s)+ G_\mu(u_\mu^k(s))\rho_2(s), u_\mu^k(s)\big\rangle ds\nonumber\\
			& \quad - i\sqrt{\varepsilon}\int_{0}^{t} \operatorname{Re}\big\langle B(u_\mu^k(s))d\mathcal{W}_1(s)+ G_\mu(u_\mu^k(s))d\mathcal{W}_2(s), u_\mu^k(s)\big\rangle.
		\end{align}
		From the estimates \eqref{energyA}-\eqref{energyG}, we have 
		\begin{align*}
			&\operatorname{Re} \big\langle i Au_\mu^k(s), u_\mu^k(s)\big\rangle=0,\\
			&	\operatorname{Re} \big\langle i \mathcal{N}_\mu(u_\mu^k(s)), u_\mu^k(s)\big\rangle=0,\\
			&	\operatorname{Re} \big\langle \beta u_\mu^k(s), u_\mu^k(s)\big\rangle=\beta \|u_\mu^k\|^2_{H},\\
			&	\operatorname{Re} \big\langle \varepsilon b(u_\mu^k(s)), u_\mu^k(s)\big\rangle=\frac{\varepsilon}{2} \sum_{m=1}^{\infty}\|B_m(u_\mu^k)\|^2_{H},\\
			&	\operatorname{Re} \big\langle i  B(u_\mu^k(s))\rho_1(s), u_\mu^k(s)\big\rangle=0,\\
			&	\operatorname{Re} \big\langle i B(u_\mu^k(s))d\mathcal{W}_1(s), u_\mu^k(s)\big\rangle=0,\\
			&\operatorname{Re} \big\langle -i G_\mu(u_\mu^k(s))\rho_2(s), u_\mu^k(s)\big\rangle = 0.	
		\end{align*}
		From \eqref{realpart}, we infer
		\begin{align*}
			\frac{1}{2}\|u^k_\mu(t)\|^2_{H} &=\frac{1}{2}\|J_\mu u_0\|^2_{H} -\int_{0}^{t} \operatorname{Re}\big\langle i Au_\mu^k(s) +i \mathcal{N}_\mu(u_\mu^k(s))+ \beta u_\mu^k(s)+\varepsilon b(u_\mu^k(s)), u_\mu^k(s)\big\rangle ds \\
			&\quad - i\int_{0}^{t} \operatorname{Re}\big\langle B(u_\mu^k(s))\rho_1(s)+ G_\mu(u_\mu^k(s))\rho_2(s), u_\mu^k(s)\big\rangle ds\\
			& \quad - i\sqrt{\varepsilon}\int_{0}^{t} \operatorname{Re}\big\langle B(u_\mu^k(s))d\mathcal{W}_1(s)+ G_\mu(u_\mu^k(s))d\mathcal{W}_2(s), u_\mu^k(s)\big\rangle \\
			&\leq  \frac{1}{2}\|J_\mu u_0\|^2_{H} - \int_{0}^{t} \beta \|u^k_\mu(s)\|^2_{H} ds -\frac{\varepsilon}{2} \int_{0}^{t} \sum_{m=1}^{\infty}\|B_m(u_\mu^k(s))\|^2_{H} ds\\
			&\quad  +\sqrt{\varepsilon}\int_{0}^{t}\operatorname{Re} \big\langle -i G_\mu(u_\mu^k(s))d\mathcal{W}_2(s), u_\mu^k(s)\big\rangle\\
			&\leq  \frac{1}{2}\|J_\mu u_0\|^2_{H} +\int_{0}^{t}\operatorname{Re} \big\langle -i G_\mu(u_\mu^k(s))d\mathcal{W}_2(s), u_\mu^k(s)\big\rangle\\
			&\leq  \frac{1}{2}\| u_0\|^2_{H} +\sqrt{\varepsilon}\int_{0}^{t}\operatorname{Re} \big\langle -i G_\mu(u_\mu^k(s))d\mathcal{W}_2(s), u_\mu^k(s)\big\rangle.	
		\end{align*}
		Since we are studying the behaviour of solution as $\varepsilon\to 0$, we choose $\varepsilon<1$ for our convenience. This implies
		\begin{align*}
			\|u^k_\mu(t)\|^2_{H}&\leq  \|u_0\|^2_{H} +C_N(1+T)+2\Big|\int_{0}^{t}\operatorname{Re} \big\langle -i G_\mu(u_\mu^k(s))d\mathcal{W}_2(s), u_\mu^k(s)\big\rangle\Big|.		
		\end{align*}
		Now, taking power $\frac{q}{2}$ both sides and using Jensen's inequality, we derive
		\begin{align*}
			\|u^k_\mu(t)\|^q_{H}&\leq  C\| u_0\|^p_{H} +2^{\frac{q}{2}}\Big(\Big|\int_{0}^{t}\operatorname{Re} \big\langle -i G_\mu(u_\mu^k(s))d\mathcal{W}_2(s), u_\mu^k(s)\big\rangle\Big|\Big)^{\frac{q}{2}}.
		\end{align*}
		Taking the supremum over time, followed by the expectation, leads to the conclusion that
		\begin{align*}
			&\mathbb{E} \Big\{\sup_{0 \leq t \leq T}\|u^k_\mu(t)\|^q_{H}\Big\}\leq  	C\| u_0\|^p_{H} +2^{\frac{q}{2}}	\mathbb{E} \Big\{\sup_{0 \leq t \leq T}\Big(\Big|\int_{0}^{t}\operatorname{Re} \big\langle -i G_\mu(u_\mu^k(s))d\mathcal{W}_2(s), u_\mu^k(s)\big\rangle\Big|\Big)^{\frac{q}{2}}\Big\}.	
		\end{align*}
		For the term involving noise, we use the Burkholder-Davis-Gundy inequality (\Cref{BDG}) and then H\"older's inequality, Young's inequality and Jensen's inequality to get, 
		\begin{align*}
			&		\mathbb{E} \Big\{\sup_{0 \leq t \leq T}\Big|\int_{0}^{t}\operatorname{Re} \big\langle -i G_\mu(u_\mu^k(s))d\mathcal{W}_2(s), u_\mu^k(s)\big\rangle\Big|^{\frac{q}{2}}\Big\}\leq C \mathbb{E}\Big\{ \Big[\int_{0}^{T}\Big\{(C_1+C_2\|u_\mu^k(s)\|_{H})\|u_\mu^k(s)\|_{H}\Big\}^2 ds \Big]^\frac{q}{4}\Big\}\\
			& \leq  C \mathbb{E} \Big\{\Big[\int_{0}^{T}\Big\{1+\|u_\mu^k(s)\|^2_{H}\Big\}^2 ds \Big]^\frac{q}{4} \Big\}\leq  C \mathbb{E} \Big\{\|1+\|u_\mu^k(s)\|^2_{H}\|^{{\frac{q}{2}}}_{L^2(0,T)}\Big\}\\
			&\leq C(1+T)^\frac{q}{2} + 	\mathbb{E} \Big\{\Big(\int_{0}^{T}\|u_\mu^k(t)\|^4_{H} dt\Big)^{\frac{q}{4}} \Big\}\leq C(1+T)^\frac{q}{2} + 	\mathbb{E} \Big\{\sup_{0 \leq t \leq T}\|u_\mu^k(s)\|^{\frac{q}{2}}_{H}   \Big(\int_{0}^{T}\|u_\mu^k(t)\|^2_{H} dt\Big)^{\frac{q}{4}} \Big\}\\
			&\leq C(1+T)^\frac{q}{2} +\delta	\mathbb{E} \Big\{\sup_{0 \leq t \leq T}\|u^k_\mu(t)\|^q_{H}\Big\}+ \frac{1}{4 \delta} 	\mathbb{E} \Big\{\Big(\int_{0}^{T}\|u_\mu^k(t)\|^2_{H} dt\Big)^{\frac{q}{2}} \Big\}	\\
			&\leq C(1+T)^\frac{q}{2} +\delta	\mathbb{E} \Big\{\sup_{0 \leq t \leq T}\|u^k_\mu(t)\|^q_{H}\Big\}+ \frac{1}{4 \delta} (T)^{\frac{q}{2}-1} 	\mathbb{E} \Big\{\int_{0}^{T}\|u_\mu^k(t)\|^q_{H} dt \Big\}.
		\end{align*}
		Consequently, we obtain
		\begin{align*}
			\mathbb{E} \Big\{\sup_{0 \leq t \leq T}\|u^k_\mu(t)\|^q_{H}\Big\}&\leq  	C\| u_0\|^p_{H} +2^{\frac{q}{2}}	\Big[ C(1+T)^\frac{q}{2} +\delta	\mathbb{E} \Big\{\sup_{0 \leq t \leq T}\|u^k_\mu(t)\|^q_{H}\Big\}\\
			&\quad+ \frac{1}{4 \delta} (T)^{\frac{q}{2}-1} 	\mathbb{E} \Big\{\int_{0}^{T}\|u_\mu^k(t)\|^q_{H} dt \Big\}\Big].
		\end{align*}
		This leads us to conclude,
		\begin{align*}
			&\big(1- 2^{\frac{q}{2}}\delta\big)\mathbb{E} \Big\{\sup_{0 \leq t \leq T}\|u^k_\mu(t)\|^q_{H}\Big\}\\
			&\leq  	C\| u_0\|^p_{H} +2^{\frac{q}{2}}C(1+T)^\frac{q}{2}+ 2^{\frac{q}{2}}\frac{1}{4 \delta} (T)^{\frac{q}{2}-1} 	\mathbb{E} \Big\{\int_{0}^{T}\|u_\mu^k(s)\|^q_{H} ds\Big\},
		\end{align*}
		which implies,
		\begin{align}\label{forgron}
			C_\delta \mathbb{E} \Big\{\sup_{0 \leq t \leq T}\|u^k_\mu(t)\|^q_{H}\Big\}&\leq C_{\| u_0\|_{H}, N,T,q}+ C_{\delta,q,T,N} \mathbb{E} \Big\{\int_{0}^{T}\|u_\mu^k(s)\|^q_{H} ds\Big\}\nonumber\\	
			&\leq C_{\| u_0\|_{H},N, T,q}+ C_{\delta,q,T,N} \int_{0}^{T} \mathbb{E} \Big\{\sup_{0 \leq s \leq t}\|u_\mu^k(s)\|^q_{H} \Big\}dt.	
		\end{align}
		An application of Gr\"onwall's inequality in \eqref{forgron} yields,
		\begin{align*}
			\mathbb{E} \Big\{\sup_{0 \leq t \leq T}\|u^k_\mu(t)\|^q_{H}\Big\}\leq C_{\| u_0\|_{H}, \delta,q,T,N}.
		\end{align*}
		Taking $\mu\to \infty$, using \eqref{limitpasslim} and weak lower semicontinuity of norm, we conclude
		\begin{align}\label{uklinfbound}
			\mathbb{E} \Big\{\sup_{0 \leq t \leq T}\|u^k(t)\|^q_{H}\Big\}\leq C_{\| u_0\|_{H}, \delta,q,T,N}< \infty.
		\end{align}
		This completes the proof of \Cref{lemuklinfbound}.
	\end{proof}
	Finally we are ready to prove \Cref{wellsce}. To do that, first we prove the uniform boundedness of solution $u^k$ of the truncated equation in Step (A) and then we show that the stopping time $\tau_\infty$ is $\mathbb{P}$\,-a.s. equal to $T$ in Step (B).
	\begin{proof}[Proof of \Cref{wellsce}]
		Let $(u^k)_{k\in \mathbb{N}}$ be a sequence of solutions of \eqref{truncated} provided by the existence of unique global solution $u^R$ of the truncated equation with $R$ replaced by $k$.
		Then from the existence of local solution (with help of the stopping time $\tau_k$) of the original equation, we have $u(t)=u^k(t), t \in [0,\tau_k]$ is local mild solution of \eqref{Sce} upto time $\tau_\infty$, i.e., in $[0,\tau_\infty).$ The solution $u^k$ satisfies the following mild form:
		\begin{align}
			u^k(t)&=S(t)u_0	- i\int_{0}^{t} S(t-s)\theta_k(	\|u^k\|_{\mathbb{L}^{\infty, p}_{2, r}(s)}) \mathcal{N}(u^k(s))ds\nonumber\\
			&\quad- \int_{0}^{t} S(t-s) \big[\beta u^k(s) + \varepsilon b(u^k(t))\big]ds\nonumber\\ 
			&\quad - i\int_{0}^{t} S(t-s)\big[B(u^k(s))\rho_1(s)+G(u^k(s))\rho_2(s) \big]ds\nonumber\\ 
			&\quad- i \sqrt{\varepsilon}\int_{0}^{t} S(t-s)\big[B(u^k(s))d\mathcal{W}_1(s)+G(u^k(s))d\mathcal{W}_2(s) \big], \quad \text{for } t\in [0,T]. 
		\end{align}
		\begin{steps}
			\item\label{uniformukstep}\textbf{(Uniform bound for $u^k$ in $\mathbb{X}$):} First we will prove that $u^k$ is uniformly bounded in $\mathbb{X}_2=L^q(\Omega;L^p(0, T; L^r(\mathbb{R}^d)))$ i.e. we will show the existence of a constant $C>0$ such that 
			\begin{align*}
				\sup_{k} \mathbb{E} \big\{\|u^k\|^q_{L^p(0, T; L^r(\mathbb{R}^d))}\big \}\leq C.
			\end{align*}
			Note that, from \eqref{kbx2} and \eqref{kgx2}, using \Cref{lemuklinfbound}, we have the estimates for the noise terms, 
			\begin{align}\label{conseq417}
				\mathbb{E}\big\{ \|K_B(\cdot)\|^q_{L^p(0, T; L^r(\mathbb{R}^d))}\big\}&\leq (T)^{\frac{q}{p}+\frac{q}{2}}\big[ C \|B\|^q_{\mathscr{L}(H,\mathcal{L}_2 (Y_1,H))}	\mathbb{E} \big\{\|u^k\|^q_{L^\infty(0,T; H)}\big\}\big]\nonumber\\
				&\leq C_{\| u_0\|_{H}, \delta,q,T}.
			\end{align}
			and 
			\begin{align}\label{conseq420}
				\mathbb{E}\{ \|K_G(\cdot)\|^q_{L^p(0, T; L^r(\mathbb{R}^d))}\}&\leq  C (T)^{\frac{q}{p}+\frac{q}{2}}\big[ 1 +\mathbb{E} \{\|u^k\|^q_{L^\infty(0,T; H)}\}\big] \nonumber\\
				&\leq C_{\| u_0\|_{H}, \delta,q,T}.
			\end{align}
			%
			%
			%
			%
			%
			%
			%
			%
			Let us fix $\omega \in \Omega$ and take $T_k(\omega) \in (0,T]$ (which will be determined later). Then, using $\varepsilon<1$, we have the estimate
			\begin{align}\label{1uniformbound}
				&\|	u^k\|^q_{L^p(0, T_k; L^r(\mathbb{R}^d))}\nonumber\\
				&\leq  C\|u_0\|^q_{H} +C_k  T_k^{\frac{q(p-r)}{p}} \|u^k\|_{L^p(0, T_k; L^r(\mathbb{R}^d))}^{q\alpha}+ (\beta^q+ (C\varepsilon)^q )T_k^q \|  u^k\|^q_{L^\infty(0,T; H)}\nonumber\\
				&\quad+C\|  u^k\|^q_{L^\infty(0,T; H)}\Big\{\int_{0}^{T_k} \big\{\|B\|_{\mathscr{L}(H,\mathcal{L}_2 (Y_1,H))} \|\rho_1(s)\|_{Y_1}+ C_2 \|\rho_2(s)\|_{Y_2} \big\}ds\Big\}^q\nonumber \\ 
				&\quad+ \Big(C C_1 \int_{0}^{T_k}\|\rho_2(s)\|_{Y_2}ds\Big)^q +\|K_B(\cdot)\|^q_{L^p(0, T_k; L^r(\mathbb{R}^d))}+\|K_G(\cdot)\|^q_{L^p(0, T_k; L^r(\mathbb{R}^d))}\nonumber\\
				&\leq   C\|u_0\|^q_{H}+C_k  T_k^{\frac{q(p-r)}{p}} \|u^k(t)\|_{L^p(0, T_k; L^r(\mathbb{R}^d))}^{q\alpha}+ C(1+T_k^q)\|  u^k\|^q_{L^\infty(0,T_k; H)}\nonumber\\
				&\quad+\|K_B(\cdot)\|^q_{L^p(0, T_k; L^r(\mathbb{R}^d))}+\|K_G(\cdot)\|^q_{L^p(0, T_k; L^r(\mathbb{R}^d))}.
			\end{align}
			%
			%
			We consider,
			\begin{align}\label{defmk}
				M_k^{T_k}(\omega):= C\|u_0\|^q_{H}+ C(1+T_k^q)\|  u^k\|^q_{L^\infty(0,T_k; H)}+\|K_B(\cdot)\|^q_{L^p(0, T_k; L^r(\mathbb{R}^d))}+\|K_G(\cdot)\|^q_{L^p(0, T_k; L^r(\mathbb{R}^d))},
			\end{align}
			and denote, 
			\begin{align*}
				M_k(\omega):= C\|u_0\|^q_{H}+ C(1+T^q)\|  u^k\|^q_{L^\infty(0,T; H)}+\|K_B(\cdot)\|^q_{L^p(0, T; L^r(\mathbb{R}^d))}+\|K_G(\cdot)\|^q_{L^p(0, T; L^r(\mathbb{R}^d))}.
			\end{align*}
			Then, from \eqref{1uniformbound}, we have
			\begin{align}\label{boundbymk}
				\|	u^k\|^q_{L^p(0, T_k; L^r(\mathbb{R}^d))} \leq 	M_k^{T_k}(\omega)+C_k  T_k^{\frac{q(p-r)}{p}} \|u^k\|_{L^p(0, T_k; L^r(\mathbb{R}^d))}^{q\alpha}.
			\end{align}
			Consider a function $ g, $ defined by
			\begin{align*}
				g(x) := M + \frac{1}{4(2M)^{\alpha-1}} x^{\alpha} - x, \quad x \geq 0.
			\end{align*}	
			Here $M > 0 $. Since $ g''(x) \geq 0 $, for all $ x \in [0, \infty) $, it is convex on $[0, \infty) $. 
			
			Also note that
			\begin{align*}
				& g(0) = M > 0,\\
				& g(2M) = M + \frac{1}{4(2M)^{\alpha-1}} (2M)^{\alpha} - 2M = -\frac{M}{2} < 0,\\
				&g\Big( 4^{\frac{1}{\alpha-1}} \cdot 2M \Big) = M + \frac{1}{4(2M)^{\alpha-1}} \Big( 4^{\frac{1}{\alpha-1}} \cdot 2a \Big)^{\alpha} - 4^{\frac{1}{\alpha-1}} \cdot 2M = M > 0.
			\end{align*}	
			By the intermediate value theorem, there exists two points $ C_1, C_2 $ with $ 0 < C_1 < 2M < C_2 < 4^{\frac{1}{2\sigma}} \cdot 2M < \infty $ at which $ g(C_1) = g(C_2) = 0 $. Therefore, we can conclude that
			\begin{align*}
				g(x) \geq 0 \quad \text{if and only if} \quad 0 \leq x \leq C_1 \text{ or } x \geq C_2.
			\end{align*}
			Keeping the assumption $ \frac{p-r}{p} > 0 $ in mind, we choose the required value $ T_k(\omega)$, for $ \omega \in \Omega $, to be
			\begin{align}\label{tkomega}
				T_k(\omega) = T \wedge (4C  (2M_k(\omega))^{\alpha-1})^{-\frac{1}{\frac{q(p-r)}{p}}}.
			\end{align}
			This implies,
			\begin{align*}
				(T_k(\omega))^{\frac{q(p-r)}{p}}\leq  (4C  (2M_k(\omega))^{\alpha-1})^{-1},
			\end{align*}
			that is, $$C(M_k(\omega))^{\alpha-1}(T_k(\omega))^{\frac{q(p-r)}{p}}\leq  \frac{1}{4\cdot 2^{\alpha-1}}.$$ 
			Replacing $M_k$ by $M_k^{T_k}$, we have
			\begin{align*} 
				&T_k^{\frac{q(p-r)}{p}} \Big[C\|u_0\|^q_{H}+ C(1+T^q)\|  u^k\|^q_{L^\infty(0,T; H)}\\
				&+\|K_B(\cdot)\|^q_{L^p(0, T_k; L^r(\mathbb{R}^d))}+\|K_G(\cdot)\|^q_{L^p(0, T_k; L^r(\mathbb{R}^d))}\Big]^{\alpha-1}	\leq \frac{1}{4 \cdot 2^{\alpha-1}}. 
			\end{align*} 
			Define the sequence of stopping times,
			\begin{align*} 
				\sigma_k :&=\inf \Big\{ t \in [0,T] :  t^{\frac{q(p-r)}{p}} 
				\Big[C\|u_0\|^q_{H}+ C(1+t^q)\|  u^k\|^q_{L^\infty(0,T; H)}\\
				&\quad+\|K_B(\cdot)\|^q_{L^p(0,  t; L^r(\mathbb{R}^d))}+\|K_G(\cdot)\|^q_{L^p(0,  t; L^r(\mathbb{R}^d))}\Big]^{\alpha-1}	> \frac{1}{4 \cdot 2^{\alpha-1}} \Big\}. 
			\end{align*} 
			It is easy to check that $T_k \leq \sigma_k$ and \eqref{boundbymk} holds with $T_k$ replaced by $\sigma_k$, i.e., 
			\begin{align*}
				\|	u^k\|^q_{L^p(0, \sigma_k; L^r(\mathbb{R}^d))}&\leq  M_k^{\sigma_k}(\omega) +C_k  \sigma_k^{\frac{q(p-r)}{p}} \|u^k\|_{L^p(0,\sigma_k; L^r(\mathbb{R}^d))}^{q\alpha}\\
				&\leq  	M_k^{\sigma_k}(\omega) +\frac{ 1}{4(2M_k^{\sigma_k})^{\alpha-1}} \|u^k\|_{L^p(0,\sigma_k; L^r(\mathbb{R}^d))}^{q\alpha}.
			\end{align*}
			Since $\|u^k\|_{L^p(0,t; L^r(\mathbb{R}^d))}^{q}$ is continuous in $t\in [0,\infty)$, and is $0$ at $t=0$, we infer that 
			\begin{align*}
				\|u^k\|_{L^p(0,\sigma_k; L^r(\mathbb{R}^d))}^{q} \leq C_1\leq 2M_k^{\sigma_k}(\omega).
			\end{align*}
			Substituting the value of $M_k^{\sigma_k}(\omega)$ from \eqref{defmk}, we have
			\begin{align}\label{sigmakbound}
				&\|u^k\|_{L^p(0,\sigma_k; L^r(\mathbb{R}^d))}^{q}\nonumber\\
				&\leq 2 \Big(C\|u_0\|^q_{H}+ C(1+\sigma_k^q)\|  u^k(s)\|^q_{L^\infty(0, \sigma_k; H)}+\|K_B(\cdot)\|^q_{L^p(0, \sigma_k; L^r(\mathbb{R}^d))}+\|K_G(\cdot)\|^q_{L^p(0, \sigma_k; L^r(\mathbb{R}^d))}\Big).
			\end{align}
			We define a sequence $\sigma_k^j$ for $j=1,2,\ldots$ as follows. For $j = 1$, we put $\sigma_k^1 = \sigma_k$. For $j = 2$, we define $\sigma_k^2$ as the infimum of times after $\sigma_k^1$ such that the above condition holds, but the initial time 0 is replaced by the initial time $\sigma_k^1$. To be more precise, for $j = 1,2, \dots$, we define  
			\begin{align*} 
				\sigma_k^{j+1} := \inf \Big\{& t \in [\sigma_k^j, T] :   (t - \sigma_k^j)^{\frac{q(p-r)}{p}} \Big[C\|u_0\|^q_{H}+ C(1+(t - \sigma_k^j)^q)\|  u^k(s)\|^q_{L^\infty(0, \sigma_k; H)}\\
				&+ \|K_B(\cdot)\|^q_{L^p(\sigma_k^j,  t; L^r(\mathbb{R}^d))}+\|K_G(\cdot)\|^q_{L^p(\sigma_k^j,  t; L^r(\mathbb{R}^d))}\Big]^{\alpha-1}	> \frac{1}{4 \cdot 2^{\alpha-1}}\Big\}.  
			\end{align*}  
			By the definition of $\sigma_k^j$, we can easily check that $\sigma_k^{j+1}-\sigma_k^j\geq T_k$, for each $j$. Let $N'=[\frac{T}{T_k}]$. Then, we can see that $\sigma_k^j=T$ for $j=N+1, N+2,\ldots$.  Since $\sigma_k^j$ is a stopping time, for $t \geq\sigma_k^j$
			\begin{align*}
				u^k(t)&=S(t-\sigma_k^j)u^k(\sigma_k^j)	- i\int_{\sigma_k^j}^{t} S(t-s)\theta_k(	\|u^k\|_{\mathbb{L}^{\infty, p}_{2, r}(s)}) \mathcal{N}(u^k(s))ds\nonumber\\
				&\quad- \int_{\sigma_k^j}^{t} S(t-s) \big[\beta u^k(s) + \varepsilon b(u^k(t))\big]ds\nonumber\\ 
				&\quad - i\int_{\sigma_k^j}^{t} S(t-s)\big[B(u^k(s))\rho_1(s)+G(u^k(s))\rho_2(s) \big]ds\nonumber\\ 
				&\quad- i \sqrt{\varepsilon}\int_{\sigma_k^j}^{t} S(t-s)\big[B(u^k(s))d\mathcal{W}_1(s)+G(u^k(s))d\mathcal{W}_2(s) \big]. 
			\end{align*}
			By applying a similar argument as in \eqref{sigmakbound} inductively, and using the previous estimates, we obtain
			\begin{align*}
				\|u^k\|_{L^p(\sigma_k^j,\sigma_k^{j+1}; L^r(\mathbb{R}^d))}^{q}
				&\leq  2 \Big(C\|u_0\|^q_{H}+ C(1+(\sigma_k^{j+1}-\sigma_k^j)^q)\|  u^k\|^q_{L^\infty(\sigma_k^j,\sigma_k^{j+1}; H)}\\
				&\qquad+\|K_B(\cdot)\|^q_{L^p(\sigma_k^j,\sigma_k^{j+1}; L^r(\mathbb{R}^d))}+\|K_G(\cdot)\|^q_{L^p(\sigma_k^j,\sigma_k^{j+1}; L^r(\mathbb{R}^d))}\Big)\\
				&\leq  2 M_k(\omega).
			\end{align*}
			Using the fact $N'=[\frac{T}{T_k}]\leq\frac{T}{T_k}$ and \eqref{tkomega}, we have
			\begin{align}\label{ublp}
				\|u^k\|_{L^p(0,T; L^r(\mathbb{R}^d))}^{q} &\leq  \|u^k\|_{L^p(0,\sigma_k^{1}; L^r(\mathbb{R}^d))}^{q}+ \sum_{j=1}^{N'}\|u^k\|_{L^p(\sigma_k^j,\sigma_k^{j+1}; L^r(\mathbb{R}^d))}^{q}\nonumber\\
				&\leq 2M_k + 2N' M_k\nonumber\\
				&\leq  2M_k + 2T \Big(4C (2M_k)^{\alpha-1}\Big)^{\frac{1}{\frac{q(p-r)}{p}}}M_k\nonumber\\
				&\leq  2M_k + C_{d,q,\alpha,T} (M_k)^{1+ \frac{p(\alpha+1)}{q(p-r)}}.
			\end{align}
			Denote $\lambda=1+ \frac{p(\alpha+1)}{q(p-r)}$. Taking expectation on the both sides of \eqref{ublp}, we get
			\begin{align*}
				&\mathbb{E} \big\{\|u^k\|^q_{L^p(0, T; L^r(\mathbb{R}^d))}\big\}\leq 2\mathbb{E}\{M_k \}+ C_{d,q,\alpha,T} \mathbb{E}\{(M_k)^{\lambda}\}.
			\end{align*}
			Using the estimates \eqref{conseq417} and \eqref{conseq420} for noise terms with $q=q \lambda, q$, we conclude
			\begin{align*}
				\mathbb{E}\{M_k \}&= C\|u_0\|^q_{H}+ C(1+T^q) \mathbb{E}\Big\{\|  u^k\|^q_{L^\infty(0,T; H)}+\|K_B(\cdot)\|^q_{L^p(0, T; L^r(\mathbb{R}^d))}+\|K_G(\cdot)\|^q_{L^p(0, T; L^r(\mathbb{R}^d))}\Big\}\\
				&\leq   C\|u_0\|^q_{H}+ C_{\| u_0\|_{H}, \delta,q,T,N}+ C_{\| u_0\|_{H}, \delta,q,T}\\
				&\leq   C_{\| u_0\|_{H}, \delta,q,T,N}.		 	
			\end{align*}
			Similarly, we will have,
			\begin{align*}
				\mathbb{E}\{(M_k)^{\lambda} \}&\leq   C_{\| u_0\|_{H},\lambda, \delta,q,T,N}.		 	
			\end{align*}
			Therefore, we deduce
			\begin{align*}
				\mathbb{E} \big\{\|u^k\|^q_{L^p(0, T; L^r(\mathbb{R}^d))}\big\}&\leq  2 C_{\| u_0\|_{H}, \delta,q,T}+C_{d,q,\alpha,T,N} C_{\| u_0\|_{H},\lambda, \delta,q,T,N}\\
				&= C_{\| u_0\|_{H},\alpha,d,\lambda, \delta,q,T,N}=L.
			\end{align*}
			Here, the constant $L$ is independent of $k$. Hence we have the uniform estimate for $u^k$.
			\item \textbf{:} In this step, we show that $T_\infty=T, \mathbb{P}$\,-a.s. Recall from \eqref{stoppingtime}, 
			\begin{align*}
				\tau_k= \inf \,\big\{\,t\in [0,T]: \|u^k\|_{\mathbb{L}^{\infty, p}_{2, r}(t)}>k\,\big\}\land T.
			\end{align*}
			Now, using the definition of stopping time and Chebyshev inequality, we get
			\begin{align*}
				\mathbb{P}(\tau_k = T) &= \mathbb{P} \Big\{ \sup_{0 \leq t \leq T} \| u^k (t) \|_{L^2 (\mathbb{R}^d)} + \| u^k \|_{L^p (0,T; L^r (\mathbb{R}^d))} \leq k \Big\} \\
				&\geq 1-  \mathbb{P} \Big\{ \sup_{0 \leq t \leq T} \| u^k (t) \|_{L^2 (\mathbb{R}^d)} + \| u^k \|_{L^p (0,T; L^r (\mathbb{R}^d))} > k \Big\} \\
				&\geq 1 - \frac{ 2^q\Big(\mathbb{E}\big\{\sup_{0 \leq t \leq T} \| u^k (t) \|_{L^2 (\mathbb{R}^d)}^q \big\}+  \mathbb{E}\big\{ \| u^k \|_{L^p (0,T; L^r (\mathbb{R}^d))}^q\big\}\Big)}{k^p} \\
				&\geq 1 - \frac{2^q\big\{ \| u_0 \|_{L^2 (\mathbb{R}^d)}^q + L\big\}}{k^q}.
			\end{align*}
			Hence, we have
			\begin{align*}
				\mathbb{P} (\tau_{\infty} = T) \geq \mathbb{P} \Big( \bigcup_k \{ \tau_k = T \} \Big) = \lim_{k \to \infty} \mathbb{P} (\tau_k = T) = 1.
			\end{align*}
			Thus, we infer that, $\tau_{\infty} = T$ $\mathbb{P}$\,-a.s. Since $T$ is arbitrary, this shows $u^\rho(t):=u(t)_{t \in [0, \infty)}$ is a global mild solution of \eqref{Sce} in $ L^q (\Omega; L^\infty(0,T; H)) \cap L^p (0,T; L^r (\mathbb{R}^d)))$. Moreover, from \eqref{mildformsce}, we have $u^\rho(\cdot,\omega)\in C([0, T]; H)$, $\mathbb{P}$\,-a.s. This completes the proof of \Cref{wellsce}.
		\end{steps}
	\end{proof}
	\section{Large deviation principle}\label{sectionldp}
	In this section, we establish the LDP for the solutions of \eqref{perturbed}. Here, we employ the weak convergence approach in a suitable Polish space, developed by Budhiraja and Dupuis in \cite{MR1785237, MR2435853}, using the variational representation for non-negative functionals of Brownian motion.
	\subsection{Condition 1(Compactness criterion):}
	In this subsection, we prove the first equivalent condition of Laplace principle, known as compactness criteria.
	For an admissible pair $(p,r)$ and $u_0\in H$, existence and uniqueness of solutions to (\ref{Skeleton}) allows us to define the following Borel measurable map (see \cite{MR1785237} and  \cite[ Theorem 2.2]{MR1275582})
	\begin{align*}
		\Psi^0: H \times C([0,T]; Y_1)\times C([0,T]; Y_2) \to L^\infty(0,T; H) \cap L^p(0, T; L^r(\mathbb{R}^d)),
	\end{align*}
	such that $\Psi^0\big(u_0,\int_{0}^{\cdot}\rho_1(s)ds,\int_{0}^{\cdot}\rho_2(s)ds\big)=u^\rho,$ where $u^\rho$ is the unique solution of \eqref{Skeleton} satisfying the mild form (\ref{mildformsk}) on $[0,T]$. We denote $	\Psi^0(\rho):=\Psi^0\big(u_0,\int_{0}^{\cdot}\rho_1(s)ds,\int_{0}^{\cdot}\rho_2(s)ds\big).$
	\begin{thm}\label{compactness}
		For any given $N\in \mathbb{N}$, let $\rho^0, \rho^n \in \mathbb{D}_N$ be such that $\rho^n \to \rho^0$ weakly in $\mathbb{D}_N$ as $n \to \infty$. Then
		\begin{align*}
			\lim_{n \to \infty} \|	\Psi^0(\rho^n)-	\Psi^0(\rho^0)\|_{\mathbb{L}^{\infty, p}_{2, r}(T)}=0.
		\end{align*}
	\end{thm}
	\begin{proof}
		For each $n=0,1,2,\ldots$ consider $\Psi^0(\rho^n)=u^{\rho^n}=u^n$ and $\Psi^0(\rho^0)=u^{\rho^0}=u^0$,
		where $\rho^n = (\rho_1^n, \rho_2^n)$ and $u^n$ is the unique solution of the (\ref{Skeleton}) with $\rho=\rho^n$.
		Then, from \eqref{mildformsk}, we have
		\begin{align}\label{c1}
			u^n(t)&=S(t)u^n_0	- \int_{0}^{t} S(t-s) \big[i \mathcal{N}(u^n(s))+ \beta u^n(s)\big]ds - i\int_{0}^{t}S(t-s)\big[B(u^n(s))\rho^n_1(s)+ G(u^n(s))\rho^n_2(s)  \big]ds, 
		\end{align}
		and 
		\begin{align}\label{c2}
			u^0(t)&=S(t)u_0	- \int_{0}^{t} S(t-s) \big[i \mathcal{N}(u^0(s))+ \beta u^0(s)\big]ds - i\int_{0}^{t}S(t-s)\big[B(u^0(s))\rho^0_1(s)+ G(u^0(s))\rho^0_2(s)  \big]ds,
		\end{align}
		where $u^n_0=u_0$. Our aim is to show that, 
		\begin{align}\label{ctarget}
			\lim_{n \to \infty}\|u^n-u^0\|_{\mathbb{L}^{\infty, p}_{2, r}(T)}=0.
		\end{align}
		For any time $t \in [0,T]$, \eqref{c1} and \eqref{c2} imply,
		\begin{align*}
			u^n(t)-	u^0(t)&=S(t)(u^n_0-u_0)	- \int_{0}^{t} S(t-s) \big[i\big\{ \mathcal{N}(u^n(s))-\mathcal{N}(u^0(s))\big\}+ \beta \big\{u^n(s)-u^0(s)\big\}\big]ds\\
			&\quad - i\int_{0}^{t}S(t-s)\big[\big\{B(u^n(s))\rho^n_1(s)-B(u^0(s))\rho^0_1(s)\big\}+\big\{ G(u^n(s))\rho^n_2(s)- G(u^0(s))\rho^0_2(s)\big\} \big]ds.
		\end{align*}
		Using similar arguments to establish the estimates \eqref{liptypeF}, \eqref{liptypebeta}, \eqref{diffI2a} and \eqref{diffI2b}, for any time $\tilde{T}\in [0,T]$ we have the existence of a constant $C$ independent of $n, \tilde{T}$ such that, the following estimates hold:
		\begin{align}
			&\bigg\|\int_{0}^{\cdot} S(\cdot-s)i\big\{ \mathcal{N}(u^n(s))-\mathcal{N}(u^0(s))\big\} ds\bigg\|_{\mathbb{L}^{\infty, p}_{2, r}(\tilde{T})} \leq C (\tilde{T})^{\frac{p-r}{p}} (2M)^{\alpha -1} \|	u^n-	u^0 \|_{\mathbb{L}^{\infty, p}_{2, r}(\tilde{T})},\label{cliptypeF}\\
			&\bigg\|\int_{0}^{\cdot} S(\cdot-s)\beta \big\{u^n(s)-u^0(s)\big\} ds\bigg\|_{\mathbb{L}^{\infty, p}_{2, r}(\tilde{T})}\leq C\beta \int_{0}^{\tilde{T}} \|u^n-u^0\|_{\mathbb{L}^{\infty, p}_{2, r}(s)}ds,\label{cliptypebeta}\\
			& \bigg\|- i\int_{0}^{\cdot}S(\cdot-s)\big\{B(u^n(s))\rho^n_1(s)-B(u^0(s))\rho^0_1(s)\big\}ds\bigg\|_{\mathbb{L}^{\infty, p}_{2, r}(\tilde{T})}\nonumber\\
			&= \bigg\|- i\int_{0}^{\cdot}S(\cdot-s)\big\{B(u^0(s))(\rho^n_1(s)-\rho^0_1(s))+(B(u^n(s))-B(u^0(s)))\rho^n_1(s)\big\}ds\bigg\|_{\mathbb{L}^{\infty, p}_{2, r}(\tilde{T})}\nonumber\\
			&\leq  2C \int_{0}^{\tilde{T}}\|B\|_{\mathscr{L}(H,\mathcal{L}_2 (Y_1,H))} \|\rho^n_1(s)\|_{Y_1} \|u^n-u^0\|_{\mathbb{L}^{\infty, p}_{2, r}(s)}ds \nonumber\\
			&\quad+ 2C \int_{0}^{\tilde{T}}\|B\|_{\mathscr{L}(H,\mathcal{L}_2 (Y_1,H))}\|u^0(s)\|_H \|\rho^n_1(s)-\rho^0_1(s)\|_{Y_1}ds,\label{cdiffB}
		\end{align}
		and
		\begin{align}\label{cdiffg}
			&  \bigg\|- i\int_{0}^{\cdot}S(\cdot-s)\big\{ G(u^n(s))\rho^n_2(s)- G(u^0(s))\rho^0_2(s)\big\} ds \bigg\|_{\mathbb{L}^{\infty, p}_{2, r}(\tilde{T})}\nonumber\\
			& =\bigg\|- i\int_{0}^{\cdot}S(\cdot-s)\big\{G(u^0(s))(\rho^n_2(s)-\rho^0_2(s))+(G(u^n(s))-G(u^0(s)))\rho^n_2(s)\big\}ds\bigg\|_{\mathbb{L}^{\infty, p}_{2, r}(\tilde{T})}\nonumber\\
			&\leq  2C \int_{0}^{\tilde{T}}L_G \|\rho^n_2(s)\|_{Y_2}\|u^n-u^0\|_{\mathbb{L}^{\infty, p}_{2, r}(s)}ds+ 2C \int_{0}^{\tilde{T}}\big(C_1+C_2\|u^0(s)\|_H \big) \|\rho^n_2(s)-\rho^0_2(s)\|_{Y_2}ds.
		\end{align}
		Considering the estimates \eqref{cliptypeF}-\eqref{cdiffg}, we deduce
		\begin{align}\label{caim1}
			&\|u^n-u^0\|_{\mathbb{L}^{\infty, p}_{2, r}(\tilde{T})}\nonumber\\
			&\leq  C\|u^n_0-u_0\|_{H}+ C (\tilde{T})^{\frac{p-r}{p}} (2M)^{\alpha -1} \|	u^n-	u^0 \|_{\mathbb{L}^{\infty, p}_{2, r}(\tilde{T})} +C\beta \int_{0}^{\tilde{T}} \|u^n-u^0\|_{\mathbb{L}^{\infty, p}_{2, r}(s)}ds\nonumber\\
			&\quad+2C \int_{0}^{\tilde{T}}\|B\|_{\mathscr{L}(H,\mathcal{L}_2 (Y_1,H))} \|\rho^n_1(s)\|_{Y_1}\|u^n-u^0\|_{\mathbb{L}^{\infty, p}_{2, r}(s)}ds\nonumber\\
			&\quad+ 2C \int_{0}^{\tilde{T}}[\|B\|_{\mathscr{L}(H,\mathcal{L}_2 (Y_1,H))}\|u^0(s)\|_H \|\rho^n_1(s)-\rho^0_1(s)\|_{Y_1}ds+ 2C \int_{0}^{\tilde{T}}L_G \|\rho^n_2(s)\|_{Y_2}\|u^n-u^0\|_{\mathbb{L}^{\infty, p}_{2, r}(s)}ds\nonumber\\
			&\quad+ 2C \int_{0}^{\tilde{T}}\big(C_1+C_2\|u^0(s)\|_H \big) \|\rho^n_2(s)-\rho^0_2(s)\|_{Y_2}ds.
		\end{align}
		Choose $T_0$ such that $C (T_0)^{\frac{p-r}{p}} (2M)^{\alpha -1}\leq \frac{1}{2}$. Since $u^n_0=u_0$, we have
		\begin{align}\label{caim}
			&\frac{1}{2}\|u^n-u^0\|_{\mathbb{L}^{\infty, p}_{2, r}(T_0)}\nonumber\\
			&\leq C\beta \int_{0}^{T_0} \|u^n-u^0\|_{\mathbb{L}^{\infty, p}_{2, r}(s)}ds+2C \int_{0}^{T_0}\|B\|_{\mathscr{L}(H,\mathcal{L}_2 (Y_1,H))} \big(\sup_{n}\|\rho^n_1(s)\|_{Y_1}\big)\|u^n-u^0\|_{\mathbb{L}^{\infty, p}_{2, r}(s)}ds\nonumber\\
			&\quad+ 2C \int_{0}^{T_0}[\|B\|_{\mathscr{L}(H,\mathcal{L}_2 (Y_1,H))}\|u^0(s)\|_H \|\rho^n_1(s)-\rho^0_1(s)\|_{Y_1}ds\nonumber\\
			&\quad+ 2C \int_{0}^{T_0}L_G \big(\sup_{n}\|\rho^n_2(s)\|_{Y_2}\big)\|u^n-u^0\|_{\mathbb{L}^{\infty, p}_{2, r}(s)}ds+ 2C \int_{0}^{T_0}\big(C_1+C_2\|u^0(s)\|_H \big) \|\rho^n_2(s)-\rho^0_2(s)\|_{Y_2}ds.
		\end{align}
		Since 
		\begin{align*}
			\int_{0}^{T_0}\Big\{ C\beta+\|B\|_{\mathscr{L}(H,\mathcal{L}_2 (Y_1,H))} \big(\sup_{n}\|\rho^n_1(s)\|_{Y_1}\big)+L_G \big(\sup_{n}\|\rho^n_2(s)\|_{Y_2}\big)\Big\}ds <C_{N},
		\end{align*}
		where $C_{N}$ is independent of $n$,
		using Gr\"onwall's inequality in \eqref{caim}, we assert that
		\begin{align}\label{clim}
			\|u^n-u^0\|_{\mathbb{L}^{\infty, p}_{2, r}(T_0)}&\leq \Big[4C \int_{0}^{T_0}[\|B\|_{\mathscr{L}(H,\mathcal{L}_2 (Y_1,H))}\|u^0(s)\|_H \|\rho^n_1(s)-\rho^0_1(s)\|_{Y_1}ds\nonumber\\
			& \quad\quad+4C \int_{0}^{T_0}\big(C_1+C_2\|u^0(s)\|_H \big) \|\rho^n_2(s)-\rho^0_2(s)\|_{Y_2}ds\Big]e^{C_N}.
		\end{align}
		Taking limit supremum ${n \to\infty}$ in \eqref{clim} and using Fatou's lemma, we infer
		\begin{align}
			\limsup_{n \to \infty}\|u^n-u^0\|_{\mathbb{L}^{\infty, p}_{2, r}(T_0)}&\leq \Big[4C \int_{0}^{T_0}[\|B\|_{\mathscr{L}(H,\mathcal{L}_2 (Y_1,H))}\|u^0(s)\|_H \limsup_{n \to \infty}\|\rho^n_1(s)-\rho^0_1(s)\|_{Y_1}ds\nonumber\\
			& \quad\quad+4C \int_{0}^{T_0}\big(C_1+C_2\|u^0(s)\|_H \big) \limsup_{n \to \infty}\|\rho^n_2(s)-\rho^0_2(s)\|_{Y_2}ds\Big]e^{C_N}.
		\end{align}
		Using the weak convergence of $\rho^n$ to $\rho$ in $\mathbb{D}_N$, we conclude
		\begin{align}\label{compactnessext}
			\limsup_{n \to \infty}\|u^n-u^0\|_{\mathbb{L}^{\infty, p}_{2, r}(T_0)}=0.
		\end{align}
		Repeating similar arguments we employed to get \eqref{muconv} from \eqref{mulip3}, we obtain the following consequence of \eqref{compactnessext} as
		\begin{align*}
			\limsup_{n \to \infty}\|u^n-u^0\|_{\mathbb{L}^{\infty, p}_{2, r}(T)}=0.
		\end{align*}
		This completes the proof of \Cref{compactness}.	
	\end{proof} 
	\subsection{Condition 2 (Weak convergence criterion for LDP):}
	Let $(p,r)$ be an admissible pair and assume $q\geq 2, r=\alpha+1$ and $1< \alpha< \frac{4}{d}+1$. Then for any $\varepsilon \in (0,1),\, u_0 \in H$, \Cref{wellsce} and Yamada-Watamabe theorem (\cite{rockner2008yamada}) allow us to define the Borel measurable map (see \cite{MR1785237} and  \cite[ Theorem 2.2]{MR1275582}),
	\begin{align*}
		\Psi^{\varepsilon}: H \times C([0,T]; Y_1)\times C([0,T]; Y_2) \to L^\infty(0,T; H) \cap L^p(0, T; L^r(\mathbb{R}^d))
	\end{align*}
	such that $\Psi^\varepsilon(u_0, \sqrt{\varepsilon}\mathcal{W}_1+\int_{0}^{\cdot}\rho_1(s)ds,\sqrt{\varepsilon}\mathcal{W}_2+\int_{0}^{\cdot}\rho_2(s)ds)=u^{\rho} $, where $ u^{\rho}$ is the unique solution of \eqref{Sce} satisfying the mild form (\ref{mildformsce}) on $[0,T]$. We denote $\Psi^\varepsilon(\rho):=\Psi^\varepsilon(u_0, \sqrt{\varepsilon}\mathcal{W}_1+\int_{0}^{\cdot}\rho_1(s)ds,\sqrt{\varepsilon}\mathcal{W}_2+\int_{0}^{\cdot}\rho_2(s)ds)$.
	\begin{thm}\label{weakconv}
		For any given $\varepsilon, \delta>0$ and $N\in \mathbb{N}$, let $\rho=(\rho_1,\rho_2), \rho^\varepsilon=(\rho_1^\varepsilon,\rho_2^\varepsilon) \in \mathbb{S}_N$ be such that $\rho^\varepsilon \to \rho$ in distribution in $\mathbb{S}_N$ as $\varepsilon \to 0$. Then
		\begin{align*}
			\lim_{\varepsilon \to 0}  \mathbb{P}	\Big(\|\Psi^\varepsilon(\rho^\varepsilon)-	\Psi^0(\rho)\|_{\mathbb{L}^{\infty, p}_{2, r}(T)} \geq \delta\Big)=0,
		\end{align*}
		where $\Psi^0(\rho):=u^\rho$ is the solution of the skeleton equation \eqref{Skeleton} with the control $\rho$.
	\end{thm}
	\begin{proof}
		Since $u^{\rho^\varepsilon}$ is the unique solution of (\ref{Sce}) on $[0,T]$, it satisfies
		\begin{align}\label{mildrhoepsilon}
			u^{\rho^\varepsilon}(t)&=S(t)u_0	- \int_{0}^{t} S(t-s) \big[i \mathcal{N}({u^{\rho^\varepsilon}}(s))+ \beta {u^{\rho^\varepsilon}}(s) + \varepsilon b({u^{\rho^\varepsilon}}(t))\big]ds\\ \nonumber
			&\quad - i\int_{0}^{t} S(t-s)\big[B{u^{\rho^\varepsilon}}(s)\rho^\varepsilon_1(s)+G({u^{\rho^\varepsilon}}(s))\rho^\varepsilon_2(s) \big]ds\\ \nonumber
			&\quad- i \sqrt{\varepsilon}\int_{0}^{t} S(t-s)\big[B{u^{\rho^\varepsilon}}(s)d\mathcal{W}_1(s)+G({u^{\rho^\varepsilon}}(s))d\mathcal{W}_2(s) \big], 
		\end{align}
		and $u^\rho$ is the unique solution of the skeleton equation \eqref{Skeleton} implies, 
		\begin{align}\label{mildrho}
			u^\rho(t)&=S(t)u_0	- \int_{0}^{t} S(t-s) \big[i \mathcal{N}(u^\rho(s))+ \beta u^\rho(s)\big]ds - i\int_{0}^{t} S(t-s)\big[B(u^\rho(s))\rho_1(s)+G(u^\rho(s))\rho_2(s) \big]ds. 
		\end{align}
		Our goal is to prove, for any $\delta>0$,
		\begin{align}\label{wcaim}
			\lim_{\varepsilon \to 0}  \mathbb{P}	\Big(\|{u^{\rho^\varepsilon}}-u^\rho\|_{\mathbb{L}^{\infty, p}_{2, r}(T)} \geq\delta\Big)=0.
		\end{align}
		Here, we prove the following lemma which will be used later to prove \Cref{weakconv}.
		\begin{lem}\label{lemsupesplinf}
			For $\varepsilon\in(0,1)$, the solutions $(u^{\rho^\varepsilon})_{\varepsilon\in(0,1)}$ of \eqref{Sce} with control $\rho^\varepsilon$ satisfy the following estimate:
			\begin{align*}
				\sup_{\varepsilon \in (0,1)} \mathbb{E} \big\{\|u^{\rho^\varepsilon}\|^q_{L^\infty(0,T; H)}\big\}<C_{\| u_0\|_{H},N, \delta,q,T}.
			\end{align*}
			\begin{proof}
				For a fixed $\varepsilon$, $T_\infty=T$ and \Cref{lemuklinfbound} together imply
				\begin{align*}
					\mathbb{E} \big\{\|u^{\rho^\varepsilon}\|^q_{L^\infty(0,T; H)}\big\}<C_{\| u_0\|_{H},N, \delta,q,T},
				\end{align*}
				where the bound is independent of $\varepsilon$.
				Taking supremum over $\varepsilon \in (0,1)$, we get
				\begin{align}\label{supepslinf}
					\sup_{\varepsilon \in (0,1)} \mathbb{E} \big\{\|u^{\rho^\varepsilon}\|^q_{L^\infty(0,T; H)}\big\}<C_{\| u_0\|_{H},N, \delta,q,T}.
				\end{align}
				This completes the proof of \Cref{lemsupesplinf}.
			\end{proof}
		\end{lem}
		\begin{lem}\label{lemmn}
			For any $\eta \in (0,1)$, there exists a constant $M_\eta >0$ such that,
			\begin{align*}
				\sup_{\varepsilon\in (0,1)} \mathbb{P} \big\{\|{u^{\rho^\varepsilon}}\|_{L^p(0,T;L^r(\mathbb{R}^d))}\geq M_\eta\big\} \leq \eta.
			\end{align*}
		\end{lem}
		\begin{proof}
			Fix $\varepsilon \in (0,1)$. Using \eqref{kbx2} and \eqref{kgx2} together with \Cref{lemsupesplinf}, we have the following estimates for the noise terms, 
			\begin{align}\label{wckbb}
				\mathbb{E} \Big\{\|K_B(\cdot)\|^q_{L^p(0, T; L^r(\mathbb{R}^d))}\Big\}&\leq (T)^{\frac{q}{p}+\frac{q}{2}}\big( C \|B\|^q_{\mathscr{L}(H,\mathcal{L}_2 (Y_1,H))}	\mathbb{E} \big\{\|u^{\rho^\varepsilon}\|^q_{L^\infty(0,T; H)}\big\}\big)\leq C_{\| u_0\|_{H},N, \delta,q,T},
			\end{align}
			and 
			\begin{align}\label{wckgb}
				\mathbb{E} \Big\{\|K_G(\cdot)\|^q_{L^p(0, T; L^r(\mathbb{R}^d))}\Big\}&\leq  C (T)^{\frac{q}{p}+\frac{q}{2}}\big( 1 +\mathbb{E} \big\{\|u^{\rho^\varepsilon}\|^q_{L^\infty(0,T; H)}\big\}\big)\leq C_{\| u_0\|_{H},N, \delta,q,T}.
			\end{align}
			Following the similar procedure in the derivation of \eqref{1uniformbound}, we obtain
			\begin{align}\label{wcub1}
				&\|	u^{\rho^\varepsilon}\|_{L^p(0, T; L^r(\mathbb{R}^d))}\nonumber\\
				&\leq  C\|u_0\|_{H} +C_k  T^{\frac{(p-r)}{p}} \|u^{\rho^\varepsilon}\|_{L^p(0, T; L^r(\mathbb{R}^d))}^{\alpha}+ (\beta+ (C\varepsilon) )T \|  u^{\rho^\varepsilon}\|_{L^\infty(0,T; H)}\nonumber\\
				&\quad+C\|  u^{\rho^\varepsilon}\|_{L^\infty(0,T; H)}\Bigl\{\int_{0}^{T} \big\{\|B\|_{\mathscr{L}(H,\mathcal{L}_2 (Y_1,H))} \|\rho^\varepsilon_1(s)\|_{Y_1}+ C_2 \|\rho^\varepsilon_2(s)\|_{Y_2} \big\}ds\Bigl\}\nonumber \\ 
				&\quad+ \Big(C C_1 \int_{0}^{T}\|\rho^\varepsilon_2(s)\|_{Y_2}ds\Big) +\sqrt{\varepsilon}\|K_B(\cdot)\|_{L^p(0, T; L^r(\mathbb{R}^d))}+\sqrt{\varepsilon}\|K_G(\cdot)\|_{L^p(0, T; L^r(\mathbb{R}^d))}.
			\end{align}
			For each $\varepsilon$, we have $\rho^{\varepsilon}\in \mathbb{S}^N$ and
			\begin{align*}
				\int_{0}^{T} \big\{\|B\|_{\mathscr{L}(H,\mathcal{L}_2 (Y_1,H))} \|\rho^\varepsilon_1(s)\|_{Y_1}+ C_2 \|\rho^\varepsilon_2(s)\|_{Y_2} \big\}ds < C_{\| u_0\|_{H}, N,T}.
			\end{align*}
			This impllies,
			\begin{align}\label{epsindb}
				\sup_{\varepsilon\in (0,1)} \int_{0}^{T} \big\{\|B\|_{\mathscr{L}(H,\mathcal{L}_2 (Y_1,H))} \|\rho^\varepsilon_1(s)\|_{Y_1}+ C_2 \|\rho^\varepsilon_2(s)\|_{Y_2} \big\}ds < C_{\| u_0\|_{H}, N,T}.
			\end{align}
			Using the similar argument as in \eqref{epsindb}, we get
			\begin{align}\label{epsindrho2}
				\sup_{\varepsilon\in (0,1)} \int_{0}^{T}\|\rho^\varepsilon_2(s)\|_{Y_2}ds < C_{\| u_0\|_{H}, N,T}.
			\end{align}
			Using the estimates \eqref{epsindb} and \eqref{epsindrho2}, from \eqref{wcub1}, we conclude
			\begin{align}\label{wcub}
				\|	u^{\rho^\varepsilon}\|_{L^p(0, T; L^r(\mathbb{R}^d))}
				&\leq   C\|u_0\|_{H}+C_k  T^{\frac{(p-r)}{p}} \|u^{\rho^\varepsilon}\|_{L^p(0, T; L^r(\mathbb{R}^d))}^{\alpha}+ C\big(1+\beta T+ \varepsilon T\big)\|  u^{\rho^\varepsilon}\|_{L^\infty(0,T; H)}\nonumber\\
				&\quad+\sqrt{\varepsilon}\|K_B(\cdot)\|_{L^p(0, T; L^r(\mathbb{R}^d))}+\sqrt{\varepsilon}\|K_G(\cdot)\|_{L^p(0, T; L^r(\mathbb{R}^d))}.
			\end{align}
			We consider,
			\begin{align}\label{wcmtw}
				M_\varepsilon^{T}(\omega)&:= C\|u_0\|_{H}+ C(1+\beta T+ \varepsilon T)\|  u^{\rho^\varepsilon}\|_{L^\infty(0,T; H)}\nonumber\\
				&\quad+\sqrt{\varepsilon}\|K_B(\cdot)\|_{L^p(0, T; L^r(\mathbb{R}^d))}+\sqrt{\varepsilon}\|K_G(\cdot)\|_{L^p(0, T; L^r(\mathbb{R}^d))}.
			\end{align}
			Taking expectation on the both sides of \eqref{wcmtw} and using $\varepsilon<1$, we get
			\begin{align}\label{epsindmt}
				\mathbb{E}\big\{	M_\varepsilon^{T}(\omega)\big\}&\leq C\|u_0\|_{H}+ C(1+\beta T+  T)\mathbb{E}\big\{\|  u^{\rho^\varepsilon}\|_{L^\infty(0,T; H)}\big\}\nonumber\\
				&\quad+\mathbb{E}\big\{\|K_B(\cdot)\|_{L^p(0, T; L^r(\mathbb{R}^d))}\big\}+\mathbb{E}\big\{\|K_G(\cdot)\|_{L^p(0, T; L^r(\mathbb{R}^d))}\big\}\nonumber\\
				& \leq C_{\| u_0\|_{H}, N,T}.
			\end{align}
			Thus, \eqref{wcub} and \eqref{wcmtw} together imply,
			\begin{align}\label{wcubmt}
				\|	u^{\rho^\varepsilon}\|_{L^p(0, T; L^r(\mathbb{R}^d))}\leq  	M_\varepsilon^{T}+C_k  T^{\frac{(p-r)}{p}} \|u^{\rho^\varepsilon}\|_{L^p(0, T; L^r(\mathbb{R}^d))}^{\alpha}.
			\end{align}
			For any $\eta>0$, there exists $\Theta_\eta>0$, such that
			\begin{align}\label{claim}
				\inf_{\varepsilon \in (0,1)} \mathbb{P} \big(M^T_\varepsilon\leq \Theta_\eta\big)\geq 1-\eta.
			\end{align}
			To establish \eqref{claim}, we use Markov's inequality to get,
			\begin{align*}
				\mathbb{P} \big(M^T_\varepsilon> \Theta_\eta\big)\leq \frac{1}{\Theta_\eta} \mathbb{E}\big\{M^T_\varepsilon\big\} \leq \frac{1}{\Theta_\eta}  C_{\| u_0\|_{H}, N,T}.
			\end{align*}
			This implies, 
			\begin{align*}
				\mathbb{P} \big(M^T_\varepsilon< \Theta_\eta\big) =1-\mathbb{P} \big(M^T_\varepsilon> \Theta_\eta\big)\geq 1- \frac{1}{\Theta_\eta}  C_{\| u_0\|_{H}, N,T}.
			\end{align*}
			Therefore, we can choose some $\Theta_\eta \geq \frac{1}{\eta} C_{\| u_0\|_{H}, N,T}$, such that
			\begin{align*}
				\inf_{\varepsilon \in (0,1)} \mathbb{P} \big(M^T_\varepsilon\leq \Theta_\eta\big)\geq 1-\eta.
			\end{align*}
			By repeating the steps in the derivation of \eqref{ublp}, we conclude
			\begin{align}\label{wcublp}
				\|	u^{\rho^\varepsilon}\|_{L^p(0,T; L^r(\mathbb{R}^d))} &\leq  2M^T_\varepsilon + C_{d,q,\alpha,T} (M_\varepsilon^T)^{1+ \frac{p(\alpha+1)}{(p-r)}}.
			\end{align}
			For $\omega \in \{M^T_\varepsilon\leq \Theta_\eta\},$ we have
			\begin{align*}
				\|	u^{\rho^\varepsilon}\|_{L^p(0,T; L^r(\mathbb{R}^d))} &\leq  2\Theta_\eta + C_{d,q,\alpha,T} (\Theta_\eta)^{1+ \frac{p(\alpha+1)}{(p-r)}}.
			\end{align*}
			Now, we choose 
			\begin{align}\label{wcmn}
				M_\eta=2\Theta_\eta + C_{d,q,\alpha,T} (\Theta_\eta)^{1+ \frac{p(\alpha+1)}{(p-r)}}. 
			\end{align}
			This implies,
			\begin{align*}
				\inf_{\varepsilon\in (0,1)} \mathbb{P} \Big\{\|{u^{\rho^\varepsilon}}\|_{L^p(0,T;L^r(\mathbb{R}^d))} \leq M_\eta\Big\} \geq \inf_{\varepsilon\in (0,1)}\mathbb{P} \big\{M^T_\varepsilon\leq \Theta_\eta\big\}\geq 1-\eta.
			\end{align*}
			Therefore, we have the existence of a constant $M_\eta$ such that 
			\begin{align*}
				\sup_{\varepsilon\in (0,1)} \mathbb{P} \Big\{\|{u^{\rho^\varepsilon}}\|_{L^p(0,T;L^r(\mathbb{R}^d))}\geq M_\eta\Big\} \leq \eta.
			\end{align*}
			This completes the proof of \Cref{lemmn}.
		\end{proof}
		From above two equations \eqref{mildrhoepsilon} and \eqref{mildrho}, for $t\in[0,T]$, we have
		\begin{align}\label{diffepsilon}
			{u^{\rho^\varepsilon}}(t)-u^\rho(t)&=-i\int_{0}^{t} S(t-s) \big[ \mathcal{N}({u^{\rho^\varepsilon}}(s))- \mathcal{N}(u^\rho(s))\big]ds-\beta \int_{0}^{t} S(t-s) \big[ {u^{\rho^\varepsilon}}(s)-u^\rho(s)\big]ds \nonumber\\ 
			& \quad- i\int_{0}^{t} S(t-s)\big[B({u^{\rho^\varepsilon}}(s))\rho^\varepsilon_1(s)-B(u^\rho(s))\rho_1(s) \big]ds\nonumber\\
			&\quad- i\int_{0}^{t} S(t-s)\big[G({u^{\rho^\varepsilon}}(s))\rho^\varepsilon_2(s)-G(u^\rho(s))\rho_2(s) \big]ds\nonumber\\
			& \quad-\varepsilon\int_{0}^{t} S(t-s) b({u^{\rho^\varepsilon}}(s))ds- i \sqrt{\varepsilon}\int_{0}^{t} S(t-s)\big[B{u^{\rho^\varepsilon}}(s)d\mathcal{W}_1(s)+G({u^{\rho^\varepsilon}}(s))d\mathcal{W}_2(s) \big].
		\end{align}
		For given $\eta \in (0,1)$, we consider $M_\eta$ as the constant in \eqref{wcmn} and define the following sequence of stopping times,
		\begin{align*}
			\tau_{M_\eta}^\varepsilon := \inf \Big\{t\geq0, \|u^{\rho^\varepsilon}\|_{L^p(0,t;L^r(\mathbb{R}^d))} \geq M_\eta\Big\}\wedge T.
		\end{align*}
		Now, we recall the following  estimates:
		From \eqref{cliptypeF}, we have the following estimate for the term involving nonlinearity,
		\begin{align}\label{epsliptypeF}
			&\bigg\|-i\int_{0}^{\cdot} S(\cdot-s) \big[ \mathcal{N}({u^{\rho^\varepsilon}}(s))- \mathcal{N}(u^\rho(s))\big]ds\bigg\|_{\mathbb{L}^{\infty, p}_{2, r}(T \wedge{\tau_{M_\eta}^\varepsilon})}\nonumber\\
			& \leq C (T)^{\frac{p-r}{p}}	\big(\|{u^{\rho^\varepsilon}}\|_{L^{p}(0, T \wedge{\tau_{M_\eta}^\varepsilon}; L^{r}(\mathbb{R}^d))}+ \|u^\rho\|_{L^{p}(0, T \wedge{\tau_{M_\eta}^\varepsilon}; L^{r}(\mathbb{R}^d))}\big)^{\alpha-1} \| {u^{\rho^\varepsilon}} - u^\rho \|_{\mathbb{L}^{\infty, p}_{2, r}(T \wedge{\tau_{M_\eta}^\varepsilon})}.
		\end{align}
		Estimate \eqref{cliptypebeta} implies,
		\begin{align}\label{epsliptypebeta}
			&\bigg\|-\beta \int_{0}^{\cdot} S(\cdot-s) \big[ {u^{\rho^\varepsilon}}(s)-u^\rho(s)\big]ds\bigg\|_{\mathbb{L}^{\infty, p}_{2, r}(T \wedge{\tau_{M_\eta}^\varepsilon})}
			\leq C\beta \int_{0}^{T \wedge{\tau_{M_\eta}^\varepsilon}} \|{u^{\rho^\varepsilon}} - u^\rho\|_{\mathbb{L}^{\infty, p}_{2, r}(s)}ds.
		\end{align}
		Using \eqref{cdiffB}, we estimate
		\begin{align}\label{epsdiffB}
			&\bigg\|- i\int_{0}^{\cdot} S(\cdot-s)\big[B({u^{\rho^\varepsilon}}(s))\rho^\varepsilon_1(s)-B(u^\rho(s))\rho_1(s) \big]ds\bigg\|_{\mathbb{L}^{\infty, p}_{2, r}(T \wedge{\tau_{M_\eta}^\varepsilon})}\nonumber\\
			&\leq \int_{0}^{T} \|B\|_{\mathscr{L}(H,\mathcal{L}_2 (Y_1,H))} \, \|\rho_1^\varepsilon(s)\|_{Y_1} \,\| {u^{\rho^\varepsilon}}(s) - u^\rho (s)\|_{H}ds\nonumber\\
			&+\int_{0}^{T} \|B\|_{\mathscr{L}(H,\mathcal{L}_2 (Y_1,H))} \,\|u^\rho(s)\|_{H} \, \|\rho_1^\varepsilon(s)-\rho_1(s)\|_{Y_1}ds.
		\end{align}
		From the estimate \eqref{cdiffg}, we deduce
		\begin{align}\label{epsdiffg}
			&\bigg\|- i\int_{0}^{\cdot} S(\cdot-s)\big[G({u^{\rho^\varepsilon}}(s))\rho^\varepsilon_2(s)-G(u^\rho(s))\rho_2(s) \big]ds\bigg\|_{\mathbb{L}^{\infty, p}_{2, r}(T \wedge{\tau_{M_\eta}^\varepsilon})}\nonumber\\
			&\leq L_G \int_{0}^{T} \|\rho_2^\varepsilon(s)\|_{Y_2} \,\| {u^{\rho^\varepsilon}}(s) - u^\rho (s)\|_{H}ds+\int_{0}^{T} \big\{C_1+C_2 \|u^\rho(s)\|_{H}\big\}\, \|\rho_2^\varepsilon(s)-\rho_2(s)\|_{Y_2}ds.
		\end{align}
		We use \eqref{bx1} and \eqref{bx2} to estimate,
		\begin{align}\label{epsbdiff}
			&\bigg\|-\varepsilon\int_{0}^{\cdot} S(\cdot-s) b({u^{\rho^\varepsilon}}(s))ds\bigg\|_{\mathbb{L}^{\infty, p}_{2, r}(T \wedge{\tau_{M_\eta}^\varepsilon})}\leq \frac{\varepsilon CT}{2} \sum_{m=1}^{\infty} \|B_m^2\|_{\mathscr{L}(H)} \, \|{u^{\rho^\varepsilon}}\|_{L^\infty(0,T; H)}.
		\end{align}
		Keep the noise involving term as it is and denote it as,
		\begin{align*}
			\tilde{N}(u^{\rho^\varepsilon}):=- i \sqrt{\varepsilon}\int_{0}^{t} S(t-s)\big[B{u^{\rho^\varepsilon}}(s)d\mathcal{W}_1(s)+G({u^{\rho^\varepsilon}}(s))d\mathcal{W}_2(s) \big].
		\end{align*}
		For the above term, using \eqref{wckbb}, \eqref{wckgb} and \Cref{lemuklinfbound}, we estimate 
		\begin{align}\label{expnoise}
			\Big[\mathbb{E} \big\{\|\tilde{N}(u^{\rho^\varepsilon})\|_{\mathbb{L}^{\infty, p}_{2, r}(T)}\big\}\Big]^q &\leq  \mathbb{E} \big\{(\|\tilde{N}(u^{\rho^\varepsilon})\|_{\mathbb{L}^{\infty, p}_{2, r}(T)})^q\big\}\nonumber\\
			&\leq\mathbb{E} \big\{(\sqrt{\varepsilon})^q \big\{\|K_B(\cdot)\|_{\mathbb{L}^{\infty, p}_{2, r}(T)} +\|K_G(\cdot)\|_{\mathbb{L}^{\infty, p}_{2, r}(T)}   \big\}^q\big\}\nonumber\\
			&\leq C(\sqrt{\varepsilon})^q  \mathbb{E} \big\{\|K_B(\cdot)\|^q_{\mathbb{L}^{\infty, p}_{2, r}(T)} +\|K_G(\cdot)\|^q_{\mathbb{L}^{\infty, p}_{2, r}(T)}   \big\}\nonumber\\
			&\leq  (\sqrt{\varepsilon})^q C_{\|u_0\|,q,T,N}.
		\end{align}
		From \eqref{diffepsilon}, using the estimates \eqref{epsliptypeF}-\eqref{epsbdiff}, we deduce
		\begin{align*}
			&\|{u^{\rho^\varepsilon}}-u^\rho\|_{\mathbb{L}^{\infty, p}_{2, r}(T \wedge{\tau_{M_\eta}^\varepsilon})}\\
			&\leq  C (T)^{\frac{p-r}{p}}	\big(\|{u^{\rho^\varepsilon}}\|_{L^{p}(0, T \wedge{\tau_{M_\eta}^\varepsilon}; L^{r}(\mathbb{R}^d))}+ \|u^\rho\|_{L^{p}(0, T \wedge{\tau_{M_\eta}^\varepsilon}; L^{r}(\mathbb{R}^d))}\big)^{\alpha-1} \| {u^{\rho^\varepsilon}} - u^\rho \|_{\mathbb{L}^{\infty, p}_{2, r}(T \wedge{\tau_{M_\eta}^\varepsilon})}\\
			&\quad +C\beta \int_{0}^{T} \|{u^{\rho^\varepsilon}} - u^\rho\|_{\mathbb{L}^{\infty, p}_{2, r}(s)}ds+\int_{0}^{T} \|B\|_{\mathscr{L}(H,\mathcal{L}_2 (Y_1,H))} \, \|\rho_1^\varepsilon(s)\|_{Y_1}  \,\| {u^{\rho^\varepsilon}}(s) - u^\rho (s)\|_{H}ds\\
			&\quad+\int_{0}^{T} \|B\|_{\mathscr{L}(H,\mathcal{L}_2 (Y_1,H))} \,\|u^\rho(s)\|_{H} \, \|\rho_1^\varepsilon(s)-\rho_1(s)\|_{Y_1} ds+\int_{0}^{T} \big\{C_1+C_2 \|u^\rho(s)\|_{H}\big\}\, \|\rho_2^\varepsilon(s)-\rho_2(s)\|_{Y_2}ds\\
			&\quad+L_G \int_{0}^{T} \|\rho_2^\varepsilon(s)\|_{Y_2} \,\| {u^{\rho^\varepsilon}}(s) - u^\rho (s)\|_{H}ds+\frac{\varepsilon CT}{2} \sum_{m=1}^{\infty} \|B_m^2\|_{\mathscr{L}(H)} \, \|{u^{\rho^\varepsilon}}\|_{L^\infty(0,T; H)}+\|\tilde{N}(u^{\rho^\varepsilon})\|_{\mathbb{L}^{\infty, p}_{2, r}(T)}.
		\end{align*}
		Since, we have uniform bound (independent of $\varepsilon$) of the solution $u^{\rho^\varepsilon}$ in $L^{p}(0, T; L^{r}(\mathbb{R}^d))$ and $u^\rho$ is the solution of the skeleton equation having bound $C_{N,M}$, from the above estimate, we infer
		\begin{align*}
			&\|{u^{\rho^\varepsilon}}-u^\rho\|_{\mathbb{L}^{\infty, p}_{2, r}(T \wedge{\tau_{M_\eta}^\varepsilon})} \\
			&\leq  C(T)^{\frac{p-r}{p}} 	\big(M_\eta+ C_{N,M}\big)^{\alpha-1}	 \| {u^{\rho^\varepsilon}} - u^\rho \|_{\mathbb{L}^{\infty, p}_{2, r}(T \wedge{\tau_{M_\eta}^\varepsilon})}
			+C\beta \int_{0}^{T} \|{u^{\rho^\varepsilon}} - u^\rho\|_{\mathbb{L}^{\infty, p}_{2, r}(s)}ds\\
			&\quad+ \int_{0}^{T}  \|B\|_{\mathscr{L}(H,\mathcal{L}_2 (Y_1,H))} \, \|\rho_1^\varepsilon(s)\|_{Y_1} \| {u^{\rho^\varepsilon}}(s) - u^\rho (s)\|_{H}ds\\
			&\quad+\int_{0}^{T} \|B\|_{\mathscr{L}(H,\mathcal{L}_2 (Y_1,H))} \,\|u^\rho(s)\|_{H} \, \|\rho_1^\varepsilon(s)-\rho_1(s)\|_{Y_1}ds+\int_{0}^{T} \big\{C_1+C_2 \|u^\rho(s)\|_{H}\big\}\, \|\rho_2^\varepsilon(s)-\rho_2(s)\|_{Y_2} ds \\
			&\quad+L_G \int_{0}^{T} \|\rho_2^\varepsilon(s)\|_{Y_2} \,\| {u^{\rho^\varepsilon}}(s) - u^\rho (s)\|_{H}ds+\frac{\varepsilon CT}{2} \sum_{m=1}^{\infty} \|B_m^2\|_{\mathscr{L}(H)} \, \|{u^{\rho^\varepsilon}}\|_{L^\infty(0,T; H)} +\|\tilde{N}(u^{\rho^\varepsilon})\|_{\mathbb{L}^{\infty, p}_{2, r}(T)}.
		\end{align*}
		Choosing $T_\eta \in (0,T]$ small enough such that $ C(T)^{\frac{p-r}{p}} 	\big(M_\eta+ C_{N,M}\big)^{\alpha-1}	< \frac{1}{2}$, we deduce from the above inequality that
		\begin{align*}
			& \frac{1}{2}\|{u^{\rho^\varepsilon}}-u^\rho\|_{\mathbb{L}^{\infty, p}_{2, r}(T_\eta \wedge{\tau_{M_\eta}^\varepsilon})}\\
			&\leq C\beta \int_{0}^{T} \|{u^{\rho^\varepsilon}} - u^\rho\|_{\mathbb{L}^{\infty, p}_{2, r}(s)}ds+ \int_{0}^{T}  \,\|B\|_{\mathscr{L}(H,\mathcal{L}_2 (Y_1,H))} \, \|\rho_1^\varepsilon(s)\|_{Y_1} \| {u^{\rho^\varepsilon}}(s) - u^\rho (s)\|_{H}ds\\
			&\quad +\int_{0}^{T} \|B\|_{\mathscr{L}(H,\mathcal{L}_2 (Y_1,H))} \,\|u^\rho(s)\|_{H} \, \|\rho_1^\varepsilon(s)-\rho_1(s)\|_{Y_1}ds+\int_{0}^{T} \big\{C_1+C_2 \|u^\rho(s)\|_{H}\big\}\, \|\rho_2^\varepsilon(s)-\rho_2(s)\|_{Y_2} ds\\
			&\quad +L_G \int_{0}^{T} \|\rho_2^\varepsilon(s)\|_{Y_2} \,\| {u^{\rho^\varepsilon}}(s) - u^\rho (s)\|_{H}ds+\frac{\varepsilon CT}{2} \sum_{m=1}^{\infty} \|B_m^2\|_{\mathscr{L}(H)} \, \|{u^{\rho^\varepsilon}}\|_{L^\infty(0,T; H)} +\|\tilde{N}(u^{\rho^\varepsilon})\|_{\mathbb{L}^{\infty, p}_{2, r}(T)}.
		\end{align*}
		This implies,
		\begin{align*}
			&\|{u^{\rho^\varepsilon}}-u^\rho\|_{\mathbb{L}^{\infty, p}_{2, r}(T_\eta \wedge{\tau_{M_\eta}^\varepsilon})} \\
			&\leq  C\beta \int_{0}^{T} \|{u^{\rho^\varepsilon}} - u^\rho\|_{\mathbb{L}^{\infty, p}_{2, r}(s)}ds+  \int_{0}^{T}  \, 2\|B\|_{\mathscr{L}(H,\mathcal{L}_2 (Y_1,H))} \, \|\rho_1^\varepsilon(s)\|_{Y_1}\| {u^{\rho^\varepsilon}}(s) - u^\rho (s)\|_{H}ds\\
			&\quad+2\int_{0}^{T} \|B\|_{\mathscr{L}(H,\mathcal{L}_2 (Y_1,H))} \,\|u^\rho(s)\|_{H} \, \|\rho_1^\varepsilon(s)-\rho_1(s)\|_{Y_1}ds+2\int_{0}^{T} \big\{C_1+C_2 \|u^\rho(s)\|_{H}\big\}\, \|\rho_2^\varepsilon(s)-\rho_2(s)\|_{Y_2}ds \\
			&\quad +2L_G \int_{0}^{T} \|\rho_2^\varepsilon(s)\|_{Y_2} \,\| {u^{\rho^\varepsilon}}(s) - u^\rho (s)\|_{H}ds+{\varepsilon CT} \sum_{m=1}^{\infty} \|B_m^2\|_{\mathscr{L}(H)} \, \|{u^{\rho^\varepsilon}}\|_{L^\infty(0,T; H)} +2\|\tilde{N}(u^{\rho^\varepsilon})\|_{\mathbb{L}^{\infty, p}_{2, r}(T)}.
		\end{align*}
		Iterating similar estimates for $1 \leq k\leq [\frac{T}{T_\eta}]$, there exists a constant $C_{k,N,M}$ such that
		\begin{align*}
			&\|{u^{\rho^\varepsilon}}-u^\rho\|_{\mathbb{L}^{\infty, p}_{2, r}(kT_\eta \wedge{\tau_{M_\eta}^\varepsilon},(k+1)T_\eta \wedge{\tau_{M_\eta}^\varepsilon})} \\
			&\leq  C\beta \int_{0}^{T} \|{u^{\rho^\varepsilon}} - u^\rho\|_{\mathbb{L}^{\infty, p}_{2, r}(s)}ds+ \int_{0}^{T}  2\|B\|_{\mathscr{L}(H,\mathcal{L}_2 (Y_1,H))} \, \|\rho_1^\varepsilon(s)\|_{Y_1} \| {u^{\rho^\varepsilon}}(s) - u^\rho (s)\|_{H}ds\\
			&\quad+2\int_{0}^{T} \|B\|_{\mathscr{L}(H,\mathcal{L}_2 (Y_1,H))} \,\|u^\rho(s)\|_{H} \, \|\rho_1^\varepsilon(s)-\rho_1(s)\|_{Y_1}ds+2\int_{0}^{T} \big\{C_1+C_2 \|u^\rho(s)\|_{H}\big\}\, \|\rho_2^\varepsilon(s)-\rho_2(s)\|_{Y_2} ds\\
			&\quad+2L_G \int_{0}^{T} \|\rho_2^\varepsilon(s)\|_{Y_2} \,\| {u^{\rho^\varepsilon}}(s) - u^\rho (s)\|_{H}ds +{\varepsilon CT} \sum_{m=1}^{\infty} \|B_m^2\|_{\mathscr{L}(H)} \, \|{u^{\rho^\varepsilon}}\|_{L^\infty(0,T; H)} +2\|\tilde{N}(u^{\rho^\varepsilon})\|_{\mathbb{L}^{\infty, p}_{2, r}(T)}.
		\end{align*}
		Using \eqref{BmH}, we have
		\begin{align}\label{forlim}
			&\|{u^{\rho^\varepsilon}}-u^\rho\|_{\mathbb{L}^{\infty, p}_{2, r}(T \wedge{\tau_{M_\eta}^\varepsilon})}\nonumber\\
			&\leq   \Big[\frac{T}{T_\eta}+1\Big]\Big[ C\beta \int_{0}^{T} \|{u^{\rho^\varepsilon}} - u^\rho\|_{\mathbb{L}^{\infty, p}_{2, r}(s)}ds+2 \|u^\rho(s)\|_{L^\infty(0,T;H)}\int_{0}^{T} \|B\|_{\mathscr{L}(H,\mathcal{L}_2 (Y_1,H))} \, \|\rho_1^\varepsilon(s)-\rho_1(s)\|_{Y_1}ds\nonumber\\
			&\quad\quad+\int_{0}^{T}  2\|B\|_{\mathscr{L}(H,\mathcal{L}_2 (Y_1,H))} \, \Big(\sup_{\varepsilon\in (0,1)}\|\rho_1^\varepsilon(s)\|_{Y_1}\Big)\| {u^{\rho^\varepsilon}}(s) - u^\rho (s)\|_{H}ds\nonumber\\
			&\quad \quad+2\big\{C_1+C_2 \|u^\rho(s)\|_{L^\infty(0,T;H)}\big\}\,\int_{0}^{T}  \|\rho_2^\varepsilon(s)-\rho_2(s)\|_{Y_2} ds\nonumber\\
			&\quad\quad +2L_G \int_{0}^{T}\Big(\sup_{\varepsilon\in (0,1)} \|\rho_2^\varepsilon(s)\|_{Y_2}\Big) \,\| {u^{\rho^\varepsilon}}(s) - u^\rho (s)\|_{H}ds+{\varepsilon CT}\, \|{u^{\rho^\varepsilon}}\|_{L^\infty(0,T; H)} +2\|\tilde{N}(u^{\rho^\varepsilon})\|_{\mathbb{L}^{\infty, p}_{2, r}(T)}\Big].
		\end{align}
		From \eqref{epsindb} and \eqref{epsindrho2}, we have
		\begin{align*}
			\int_{0}^{T} \Big\{2\|B\|_{\mathscr{L}(H,\mathcal{L}_2 (Y_1,H))}\Big(\sup_{\varepsilon\in (0,1)}\|\rho_1^\varepsilon(s)\|_{Y_1}\Big)+C\beta+2 L_G\Big(\sup_{\varepsilon\in (0,1)} \|\rho_2^\varepsilon(s)\|_{Y_2}\Big)\Big\}ds < C_{\| u_0\|_{H}, N,T}.
		\end{align*}
		We use Gr\"onwall's inequality in \eqref{forlim} to get,
		\begin{align}\label{forexp}
			&\|{u^{\rho^\varepsilon}}-u^\rho\|_{\mathbb{L}^{\infty, p}_{2, r}(T \wedge{\tau_{M_\eta}^\varepsilon})} \nonumber\\
			&\leq   \Bigl[\frac{T}{T_\eta}+1\Bigl] \Big\{2\|u^\rho(s)\|_{L^\infty(0,T;H)}\int_{0}^{T} \|B\|_{\mathscr{L}(H,\mathcal{L}_2 (Y_1,H))} \, \|\rho_1^\varepsilon(s)-\rho_1(s)\|_{Y_1}ds\nonumber\\
			&\quad\quad+2\big\{C_1+C_2 \|u^\rho(s)\|_{L^\infty(0,T;H)}\big\}\int_{0}^{T} \|\rho_2^\varepsilon(s)-\rho_2(s)\|_{Y_2} ds+{\varepsilon CT} \|{u^{\rho^\varepsilon}}\|_{L^\infty(0,T; H)} +2\|\tilde{N}(u^{\rho^\varepsilon})\|_{\mathbb{L}^{\infty, p}_{2, r}(T)}\Big\}\nonumber\\
			&\quad\quad e^{\int_{0}^{T} \big\{2\|B\|_{\mathscr{L}(H,\mathcal{L}_2 (Y_1,H))}\big(\sup_{\varepsilon\in (0,1)}\|\rho_1^\varepsilon(s)\|_{Y_1}\big)+C\beta+2 L_G\big(\sup_{\varepsilon\in (0,1)} \|\rho_2^\varepsilon(s)\|_{Y_2}\big)\big\}ds}.
		\end{align}
		Taking integration over $\Omega$ in \eqref{forexp}, we obtain
		\begin{align}\label{forlimsup}
			&\mathbb{E}\big\{\|{u^{\rho^\varepsilon}}-u^\rho\|_{\mathbb{L}^{\infty, p}_{2, r}(T \wedge{\tau_{M_\eta}^\varepsilon})}\big\} \nonumber\\
			&\leq   \Bigl[\frac{T}{T_\eta}+1\Bigl] \Big[2\mathbb{E}\big\{\|u^\rho(s)\|_{L^\infty(0,T;H)}\big\}\int_{0}^{T} \|B\|_{\mathscr{L}(H,\mathcal{L}_2 (Y_1,H))} \, \|\rho_1^\varepsilon(s)-\rho_1(s)\|_{Y_1}ds\nonumber\\
			&\quad\quad+2\Big(C_1+C_2 \mathbb{E}\big\{\|u^\rho(s)\|_{L^\infty(0,T;H)}\big\}\Big)\int_{0}^{T} \|\rho_2^\varepsilon(s)-\rho_2(s)\|_{Y_2} ds+{\varepsilon CT}\, \mathbb{E}\big\{\|{u^{\rho^\varepsilon}}\|_{L^\infty(0,T; H)}\big\} \nonumber\\
			&\quad\quad +2 \mathbb{E}\big\{\|\tilde{N}(u^{\rho^\varepsilon})\|_{\mathbb{L}^{\infty, p}_{2, r}(T)}\big\}\Big] e^{C_{\| u_0\|_{H}, N,T}}.
		\end{align}
		Taking limit supremum $\varepsilon \to 0$ both side of \eqref{forlimsup} and using Fatou's lemma, we see that
		\begin{align*}
			&\limsup_{\varepsilon \to 0}\mathbb{E}\big\{\|{u^{\rho^\varepsilon}}-u^\rho\|_{\mathbb{L}^{\infty, p}_{2, r}(T \wedge{\tau_{M_\eta}^\varepsilon})}\big\} \nonumber\\
			&\leq   \Bigl[\frac{T}{T_\eta}+1\Bigl] \Big[2\mathbb{E}\big\{\|u^\rho(s)\|_{L^\infty(0,T;H)}\big\}\int_{0}^{T} \|B\|_{\mathscr{L}(H,\mathcal{L}_2 (Y_1,H))} \, \limsup_{\varepsilon \to 0}\|\rho_1^\varepsilon(s)-\rho_1(s)\|_{Y_1}ds\nonumber\\
			&\quad\quad+2\Big(C_1+C_2 \mathbb{E}\big\{\|u^\rho(s)\|_{L^\infty(0,T;H)}\big\}\Big)\int_{0}^{T} \limsup_{\varepsilon \to 0}\|\rho_2^\varepsilon(s)-\rho_2(s)\|_{Y_2} ds \nonumber\\
			&\quad\quad +\limsup_{\varepsilon \to 0}{\varepsilon CT}\, \mathbb{E}\big\{\|{u^{\rho^\varepsilon}}\|_{L^\infty(0,T; H)}\big\}+2 \limsup_{\varepsilon \to 0}\mathbb{E}\big\{\|\tilde{N}(u^{\rho^\varepsilon})\|_{\mathbb{L}^{\infty, p}_{2, r}(T)}\big\}\Big] e^{C_{\| u_0\|_{H}, N,T}}.
		\end{align*}
		Using the convergence of $\rho^{\varepsilon}$ to $\rho$ in $\mathbb{S}_N$, we assert that
		\begin{align*}
			&\limsup_{\varepsilon \to 0} \mathbb{E}\big\{\|{u^{\rho^\varepsilon}}-u^\rho\|_{\mathbb{L}^{\infty, p}_{2, r}(T \wedge{\tau_{M_\eta}^\varepsilon})} \big\}\\
			&\leq   \Bigl[\frac{T}{T_\eta}+1\Bigl]\Big( \limsup_{\varepsilon \to 0} {\varepsilon CT}\mathbb{E} \big\{\|{u^{\rho^\varepsilon}}\|_{L^\infty(0,T; H)}\big\}+2\limsup_{\varepsilon \to 0} \mathbb{E}\big\{\|\tilde{N}(u^{\rho^\varepsilon})\|_{\mathbb{L}^{\infty, p}_{2, r}(T)}\big\}\Big)e^{ C_{\| u_0\|_{H}, N,T}}\\
			&\leq   \Bigl[\frac{T}{T_\eta}+1\Bigl]\Big( \limsup_{\varepsilon \to 0}{\varepsilon CT} C_{\|u_0\|,N,M}+2\limsup_{\varepsilon \to 0}  \sqrt{\varepsilon}C_{\|u_0\|,q,T,N}\Big)e^{ C_{\| u_0\|_{H}, N,T}}\\
			&= 0.
		\end{align*}
		Now, for any $\delta >0, \eta>0$ and $M_\eta$, 
		\begin{align*}
			&\limsup_{\varepsilon \to 0} \mathbb{P} \Big( 
			\left\| {u^{\rho^\varepsilon}}-u^\rho \right\|_{\mathbb{L}^{\infty, p}_{2, r}(T)} > \delta 
			\Big) \\
			&\leq \limsup_{\varepsilon \to 0} \mathbb{P} \Big( 
			\left\| {u^{\rho^\varepsilon}}-u^\rho \right\|_{\mathbb{L}^{\infty, p}_{2, r}(T \wedge{\tau_{M_\eta}^\varepsilon})} > \delta, \; {\tau_{M_\eta}^\varepsilon} = T 
			\Big) + \limsup_{\varepsilon \to 0} \mathbb{P} \Big( {\tau_{M_\eta}^\varepsilon} < T \Big) \\
			&\leq \limsup_{\varepsilon \to 0} \mathbb{P} \Big( 
			\left\| {u^{\rho^\varepsilon}}-u^\rho \right\|_{\mathbb{L}^{\infty, p}_{2, r}(T \wedge{\tau_{M_\eta}^\varepsilon})} > \delta, \; {\tau_{M_\eta}^\varepsilon} = T 
			\Big)  + \limsup_{\varepsilon \to 0} \mathbb{P} \Big( 
			\left\| {u^{\rho^\varepsilon}} \right\|_{L^p(0, T; L^r(\mathbb{R}^d))} > M_\eta 
			\Big) \\
			&\leq \limsup_{\varepsilon \to 0} \frac{1}{\delta} \mathbb{E} \Big\{
			\left\| {u^{\rho^\varepsilon}}-u^\rho \right\|_{\mathbb{L}^{\infty, p}_{2, r}(T \wedge{\tau_{M_\eta}^\varepsilon})} 
			\Big\}+ \eta \\
			&\leq \eta.
		\end{align*}
		Since $\eta>0$ is arbitrary, we conclude \eqref{wcaim}. This completes the proof of \Cref{weakconv}.
	\end{proof}
	\begin{appendix}
		\section{Appendix}\label{sectionappendix}
		In this section, we provide some standard results and references, which we used in the proof of \Cref{wellskl}, \Cref{wellsce}.
		\begin{prop}[Absolute continuity of Lebesgue integration, {\cite[Proposition 23]{MR1013117}}]\label{abscts}
			If f is Lebesgue integrable then for every $\epsilon>0$ there exists a $\delta>0$ such that $\int_{D}|f(x)|d \mu < \epsilon$ for every measurable set $D$ with $\mu{(D)}<\delta.$	
		\end{prop}
		\begin{prop}\label{yosidaoperator}
			If \(\mathbb{X} \) is either of the spaces \( H^1(\mathbb{R}^d) \), \( H^{-1}(\mathbb{R}^d) \), or \( L^p(\mathbb{R}^d) \) for \( p \in (1, \infty) \), then
			\begin{itemize}
				\item[(i)] \(\| J_\mu g \|_\mathbb{X} \leq \| g \|_\mathbb{X}, \quad \forall g \in\mathbb{X},\)
				\item[(ii)] \(	_\mathbb{X}\langle J_\mu g, h \rangle_{\mathbb{X}^*} =   {_\mathbb{X}}\langle g, J_\mu h \rangle_{\mathbb{X}^*}, \quad \forall g \in\mathbb{X}, \, h \in\mathbb{X}^*,\)
				\item[(iii)]\(	\lim_{\mu \to \infty} \| J_\mu g - g \|_\mathbb{X} = 0, \quad \forall g \in\mathbb{X}.\)
			\end{itemize}
		\end{prop}
		\begin{proof}
			For the detailed proof, we refer to \cite[Proposition 1.5.2]{MR2002047}.
		\end{proof}
		\begin{prop}
			For any \( s \in \mathbb{R} \), \( J_\mu \) is a contraction of \( H^s(\mathbb{R}^d) \) and \( J_\mu \in \mathscr{L}(H^s(\mathbb{R}^d), H^{s+2}(\mathbb{R}^d)) \) with \(\|J_\mu\|_{\mathscr{L}(H^s(\mathbb{R}^d), H^{s+2}(\mathbb{R}^d))} \leq \max\{1,\frac{\mu}{4 \pi^2}\}\).
		\end{prop}
		\begin{proof}
			For the detailed proof, we refer to \cite[Proposition 1.5.3]{MR2002047}.
		\end{proof}
		\begin{lem}[Lions-Magenes lemma, {\cite[Lemma 1.2]{MR609732}}]\label{lionsmgnslem}
			Let $V, H, V'$ be three Hilbert spaces, each space included in the following one as $V\subset H\equiv H'\subset V'$, where $V'$ is the dual of $V$. 
			If a function $u$ belongs to $L^2(0,T;V)$ and its derivative $u'$ belongs to $L^2(0,T;V')$, 
			then $u$ is almost everywhere equal to a \textit{function continuous from $[0,T]$ into $H$} and 
			we have the following equality, which holds in the scalar distribution sense on $(0,T)$:
			\begin{equation*}
				\frac{d}{dt}\|u\|^2 = 2\langle u', u \rangle.
			\end{equation*}
		\end{lem}
		\begin{lem}[Burkholder-Davis-Gundy inequality, {\cite[Theorem 1]{MR3463679}}]\label{BDG}
			Let $M$ be an $H$-valued continuous $\mathbb{F}$-local martingale with $M_0=0$. 
			For any $p \in ]0,\infty[$ and any $\mathbb{F}$-stopping time $\tau$, one has
			\begin{equation}
				\mathbb{E}\sup_{t \leq \tau} \|M_t\|^p \;\approx_p\; 
				\mathbb{E}[M,M]_\tau^{p/2} \;=\; \mathbb{E}\langle M,M\rangle_\tau^{p/2},
			\end{equation}
			where $[M, M]$ and $\langle M, M \rangle$ are the quadratic variation and Meyer process of M, respectively.
		\end{lem}
	\end{appendix}
	
	\noindent \textbf{Data availibility:} Data sharing not applicable to this article as no datasets were generated or analysed during the current study.\\
	
	\noindent \textbf{Funding:}  Not applicable.\\
	
	\noindent \textbf{Author contributions:} All authors contributed to the study conception, investigation and methodology. The first draft of the manuscript was written by Sandip Roy and all authors commented on previous versions of the manuscript. All authors read and approved the final manuscript.\\
	
	\noindent \textbf{\large Declarations}\\
	
	\noindent\textbf{Conflict of interest:} The author has no competing interests to declare that are relevant to the content of this article.
	\bibliography{LDP_Mathscinet.bib}
	\bibliographystyle{abbrv}
\end{document}